\documentclass{amsart}
\usepackage{amsfonts, amsbsy, amsmath, amssymb}

\usepackage[figuresright]{rotating} 

\newtheorem{thm}{Theorem}[section]
\newtheorem{lem}[thm]{Lemma}
\newtheorem{cor}[thm]{Corollary}
\newtheorem{prop}[thm]{Proposition}
\newtheorem{exmp}[thm]{Example}

\newtheorem{conj}[thm]{Conjecture}

\newtheorem{rmk}[thm]{Remark}
\numberwithin{equation}{section}

\theoremstyle{definition}

\begin{document}

\title{A New Approach to Permutation Polynomials over Finite Fields, II}

\author{Neranga Fernando}
\address{Department of Mathematics and Statistics,
University of South Florida, Tampa, FL 33620}
\email{wfernand@mail.usf.edu}

\author[Xiang-dong Hou]{Xiang-dong Hou*}
\address{Department of Mathematics and Statistics,
University of South Florida, Tampa, FL 33620}
\email{xhou@usf.edu}
\thanks{* Research partially supported by NSA Grant H98230-12-1-0245.}

\author{Stephen D. Lappano}
\address{Department of Mathematics and Statistics,
University of South Florida, Tampa, FL 33620}
\email{slappano@mail.usf.edu}



\begin{abstract}
Let $p$ be a prime and $q$ a power of $p$. For $n\ge 0$, let $g_{n,q}\in\Bbb F_p[{\tt x}]$ be the polynomial defined by the functional equation $\sum_{a\in\Bbb F_q}({\tt x}+a)^n=g_{n,q}({\tt x}^q-{\tt x})$. 
When is $g_{n,q}$ a permutation polynomial (PP) of $\Bbb F_{q^e}$? This turns out to be a challenging question with remarkable breath and depth, as shown in the predecessor of the present paper. 
We call a triple of positive integers $(n,e;q)$ {\em desirable} if $g_{n,q}$ is a PP of $\Bbb F_{q^e}$.
In the present paper, we find many new classes of desirable triples whose corresponding PPs were previously unknown.    
Several new techniques are introduced for proving a given polynomial is a PP.
\end{abstract}

\maketitle

\section{Introduction}

Let $p$ be a prime and $q=p^s$, where $s$ is a positive integer. For each integer $n\ge 0$, there is a polynomial $g_{n,q}\in \Bbb F_p[{\tt x}]$ satisfying 
\begin{equation}\label{1.1}
\sum_{a\in\Bbb F_q}({\tt x}+a)^n=g_{n,q}({\tt x}^q-{\tt x}).
\end{equation}
The origin of the polynomial $g_{n,q}$ can be traced to the {\em reversed Dickson polynomials} \cite{HMSY09}. Eq.~\eqref{1.1} is a natural generalization of the functional equation that defines the reversed Dickson polynomial \cite{Hou11}. We refer the reader to the introduction of \cite{Hou-12} for more background of the polynomial $g_{n,q}$.

The present paper and its predecessor \cite{Hou-12} concern the following question: When is $g_{n,q}$ a permutation polynomial (PP) of a finite field $\Bbb F_{q^e}$? We have seen in \cite{Hou-12, HMSY09}, and will continue to see in the present paper, that this is a far reaching question that does not have a simple answer. 
We call a triple of integers $(n,e;q)$ {\em desirable} if $g_{n,q}$ is a PP of $\Bbb F_{q^e}$.
When $q=2$, all known desirable triples $(n,e;2)$ are covered by four classes, and a conjecture states that there are no other desirable triples with $q=2$. The question with a general $q$ was studied in \cite{Hou-12}. The results of \cite{Hou-12} demonstrate that there are abundant desirable triples $(n,e;q)$, many of which correspond to interesting new PPs.
However, \cite{Hou-12} was only a first step to understand the permutation properties of the polynomial $g_{n,q}$, and the results there were far from being conclusive. There is clear evidence (both theoretic and numerical) that the desirable triples determined in \cite{Hou-12} only constitute a small portion of all desirable triples. In \cite{Hou-12}, there is a table (Table 3) generated by a computer search that contains all desirable triples $(n,e;3)$ (up to equivalence) with $e\le 4$. Even in this isolated case, there are still many instances where no theoretic explanation for the desirable triples has been found. Each unexplained instance is itself a question: Is the PP sporadic or does it belong to a previously unknown class? As a sequel of \cite{Hou-12}, part of the present paper will be devoted to answering some of the questions raised in \cite{Hou-12}. However, for the better part of the present paper, we find ourselves in new fronts, dealing with questions about $g_{n,q}$ that were not touched in \cite{Hou-12}

Here is an overview of the paper. In Section 2, we determine all desirable triples $(n,1;q)$ by computing the generating function $\sum_{n\ge 0}g_{n,q}({\tt x}){\tt t}^n$. Section 3 is based on a technique developed in \cite{Hou-12}. When $q=p$ and $n=\alpha(p^{0e}+p^{1e}+\cdots+p^{(p-1)e})+\beta$, $\alpha,\beta\ge 0$, it is known \cite[Lemma 3.5]{Hou-12} that for $x\in\Bbb F_{p^e}$,
\begin{equation}\label{1.2a}
g_{n,p}(x)=
\begin{cases}
g_{\alpha p+\beta,p}(x)&\text{if}\ \text{Tr}_{p^e/p}(x)=0,\cr
x^\alpha g_{\beta,p}(x)&\text{if}\ \text{Tr}_{p^e/p}(x)\ne 0.
\end{cases}
\end{equation}
Many desirable triples in \cite{Hou-12} were obtained by the above formula. Here we are able to generalize some of those desirable triples, and we also find a few new classes. These new results,
combined with those in \cite{Hou-12}, allow us to categorize most desirable triples with $q=3$ and $e\le 6$; see Table~\ref{Tb2} in Appendix A, which is an expansion and update of Table 3 in \cite{Hou-12}. (A desirable triple is considered {\em categorized} if an infinite class containing it has been found.) We notice that in Table~\ref{Tb2} most uncategorized cases occur with $e=3$ and a few with $e=4$; for $e=5,6$, every case is categorized. This is perhaps an indication that desirable triples $(n,e;q)$ are easier to understand when $e$ is large. The first two uncategorized desirable triples with $q=3$ are $(101,3;3)$ and $(407,3;3)$. These two cases and an additional case $(91525,4;3)$ are examined in Section 4. We find that in each of these cases, the reason for $g_{n,3}$ to be PP is quite unique. We believe, for the time being, that the three cases are sporadic. Section 5 is devoted to the study of desirable triples $(n,e;q)$, where $n$ is of the form $q^a-q^b-1$. The results of our initial computer search indicate that this type of desirable triples occur frequently. A separate computer search is conducted for this type of desirable triples only. When $e>2$, all known  desirable triples $(q^a-q^b-1,e;q)$ are covered by Corollary~\ref{C5.1} and Theorem~\ref{T5.2}, and we conjecture that there are no other cases. When $e=2$, the situation becomes very interesting; see Table~\ref{Tb1} in Appendix A. We discover several new classes of desirable triples $(q^a-q^b-1,2;q)$ that provide explanations for some of the computer results. But in many other cases, no theoretic explanation of the computer results is known. At the end of Section 5 is a conjecture that predicts several classes of PPs of $\Bbb F_{q^2}$ with surprising simplicity. Section 6 primarily deals with desirable triples with even $q$. Again we find numerous classes of desirable triples. Their corresponding PPs (of $\Bbb F_{q^e}$) are related to the trace function $\text{Tr}_{q^e/q}$ in various ways. In recent literature, the trace function played important roles in many constructions of PPs over finite fields; see for example, \cite{AGW, CK09,CK10,Mar,YDWP, YD}. However, the reader will find that the role of the trace function in the PPs of the present paper is rather different. The results of Section 6 allow us to categorize most desirable triples with $q=4$ and $e\le 6$; see Table~\ref{Tb3} in Appendix A. There are three appendices. Appendix A contains the three tables mentioned above. Appendix B is devoted to the determination of the parameters satisfying the conditions in Theorems~\ref{T4.1} and \ref{T4.3}. Appendix C contains some computational results used in the proof of Theorem~\ref{T3.1}.

Throughout the paper, various techniques are employed to prove a given polynomial is a PP. It is the authors' hope that some of these techniques will be useful in other situations.

In our notation, letters in typewriter typeface, {\tt x}, {\tt y}, {\tt t}, are reserved for indeterminates. The trace function $\text{Tr}_{q^e/q}$ 
and the norm function $\text{N}_{q^e/q}$
from $\Bbb F_{q^e}$ to $\Bbb F_q$ are also treated as polynomials, that is, $\text{Tr}_{q^e/q}({\tt x})={\tt x}+{\tt x}^q+\cdots+{\tt x}^{q^{e-1}}$,
$\text{N}_{q^e/q}({\tt x})={\tt x}^{1+q+\cdots+q^{e-1}}$. When $q$ is given, we define $S_a={\tt x}+{\tt x}^q+\cdots+{\tt x}^{q^{a-1}}$ for every integer $a\ge 0$. Note that $\text{Tr}_{q^e/q}=S_e$. If two integers $m,n>0$ belong to the same $p$-cyclotomic coset modulo $q^{pe}-1$, the two triples $(m,e;q)$ and $(n,e;q)$ are called {\em equivalent}, and we write $(m,e;q)\sim(n,e;q)$ or $m\sim_{(e,q)}n$. (Note that the meaning of ``$\sim_{(e,q)}$'' here is slightly different from that of ``$\sim_{(q,e)}$'' in \cite{Hou-12}.) Desirability of triples is preserved under the $\sim$ equivalence \cite[Proposition 2.4]{Hou-12}.      

\section{The Case $e=1$}

In this section, we determine all desirable triples $(n,1;q)$.

\begin{thm}\label{T2.1}
We have
\begin{equation}\label{2.1}
\sum_{n\ge 0}g_{n,q}({\tt x}){\tt t}^n\equiv \frac{-({\tt xt})^{q-1}}{1-({\tt xt})^{q-1}-({\tt xt})^q}+(1-{\tt x}^{q-1})\frac{-{\tt t}^{q-1}}{1-{\tt t}^{q-1}}\pmod{{\tt x}^q-{\tt x}}.
\end{equation}
Namely, modulo ${\tt x}^q-{\tt x}$,
\begin{equation}\label{2.2}
g_{n,q}({\tt x})\equiv a_n{\tt x}^n+
\begin{cases}
{\tt x}^{q-1}-1&\text{if}\ n>0,\ n\equiv 0\pmod{q-1},\cr
0&\text{otherwise},
\end{cases}
\end{equation}
where 
\begin{equation}\label{2.3a}
\sum_{n\ge 0}a_n{\tt t}^n=\frac{-{\tt t}^{q-1}}{1-{\tt t}^{q-1}-{\tt t}^q}.
\end{equation}
\end{thm}

\begin{proof}
Recall from \cite[Eq.~(5.1)]{Hou-12} that 
\[
\sum_{n\ge 0}g_{n,q}{\tt t}^n=\frac{-{\tt t}^{q-1}}{1-{\tt t}^{q-1}-{\tt xt}^q}. 
\]
Clearly,
\[
\frac{-{\tt t}^{q-1}}{1-{\tt t}^{q-1}-{\tt xt}^q}\equiv \frac{-({\tt xt})^{q-1}}{1-({\tt xt})^{q-1}-({\tt xt})^q}+(1-{\tt x}^{q-1})\frac{-{\tt t}^{q-1}}{1-{\tt t}^{q-1}}\pmod{{\tt x}^{q-1}-1},
\]
and
\[
\frac{-{\tt t}^{q-1}}{1-{\tt t}^{q-1}-{\tt xt}^q}\equiv \frac{-({\tt xt})^{q-1}}{1-({\tt xt})^{q-1}-({\tt xt})^q}+(1-{\tt x}^{q-1})\frac{-{\tt t}^{q-1}}{1-{\tt t}^{q-1}}\pmod{{\tt x}}.
\]
Thus \eqref{2.1} is proved.
\end{proof}

\begin{cor}\label{C2.2}
\begin{itemize}
  \item [(i)] Assume $q>2$. Then
$(n,1;q)$ is desirable if and only if $\text{\rm gcd}(n,q-1)=1$ and $a_n\ne 0$ (in $\Bbb F_p$).
\item [(ii)] Assume $q=2$. Then $(n,1;2)$ is desirable if and only if $a_n=0$ (in $\Bbb F_2$).
\end{itemize}
\end{cor}

\begin{proof}
(i)
By \eqref{2.2}, $g_{n,q}(x)=a_nx^n$ for all $x\in\Bbb F_q^*$. If $g_{n,q}$ is a PP of $\Bbb F_q$, then $a_n\ne 0$ and $\text{gcd}(n,q-1)=1$. On the other hand, assume $a_n\ne 0$ and $\text{gcd}(n,q-1)=1$. By \eqref{2.2}, we have $g_{n,q}\equiv a_n{\tt x}^n\pmod{{\tt x}^q-{\tt x}}$, which is a PP of $\Bbb F_q$.

(ii) By \eqref{2.2}, $g_{n,2}\equiv a_n{\tt x}+{\tt x}-1\pmod{{\tt x}^2-{\tt x}}$. Hence the conclusion.
\end{proof}

\noindent{\bf Remark.} From \eqref{2.3a} one can easily derive an explicit expression for $a_n$. But that expression does not give any simple pattern of those $n$ with $a_n\ne 0$ (in $\Bbb F_p$).

\section{New Desirable Triples by a Previous Method}

In this section, we further exploit the method of \cite{Hou-12} based on \eqref{1.2a}.

Given integers $d>1$ and $a=a_0d^0+\cdots+a_td^t$, $0\le a_i\le d-1$, we write $a=(a_0,\dots,a_t)_d$; the {\em base $d$ weight} of $a$ is $w_d(a)=a_0+\cdots+a_t$.

\begin{thm}\label{T4.1}
Let $p$ be a prime and $e>1$. Assume that $e\equiv 0\pmod 2$ if $p=2$. Let $0<\alpha< p^e-1$ and $0<\beta<p^{pe}-1$ such that 
\begin{itemize}
  \item [(i)] $\alpha\equiv p^l\pmod{\frac{p^e-1}{p-1}}$ for some $0\le l<e$;
  \item [(ii)] $w_p(\beta)=p-1$;
  \item [(iii)] $w_p\bigl((\alpha p+\beta)^\dagger\bigr)=p$.
\end{itemize}  
(For $m\in \Bbb Z$, $m^\dagger$ is defined to be the integer such that $1\le m^\dagger\le p^e-1$ and $m^\dagger\equiv m\pmod{p^e-1}$.)
Let $n=\alpha(p^{0e}+p^{1e}+\cdots+p^{(p-1)e})+\beta$ and write $(\alpha p+\beta)^\dagger=a_0p^0+\cdots+a_{e-1}p^{e-1}$, $0\le a_i\le p-1$. Then $(n,e;p)$ is desirable if and only if
\[
\text{\rm gcd}(a_0+a_1{\tt x}+\cdots+a_{e-1}{\tt x}^{e-1},\ {\tt x}^e-1)={\tt x}-1.
\]
\end{thm}

\begin{proof}
$1^\circ$ We first assume $p>2$. By \eqref{1.2a}, for $x\in\Bbb F_{p^e}$,
\[
g_{n,p}(x)=
\begin{cases} 
g_{\alpha p+\beta,p}(x)&\text{if}\ \text{Tr}_{p^e/p}(x)=0,\cr
x^\alpha g_{\beta,p}(x)&\text{if}\ \text{Tr}_{p^e/p}(x)\ne 0.
\end{cases}
\]
By \cite[Lemma 3.3]{Hou-12}, $g_{\beta,p}=-1$. When $x\in \text{Tr}_{p^e/p}^{-1}(0)$, we have 
$g_{\alpha p+\beta,p}(x)=g_{(\alpha p+\beta)^\dagger,p}(x)$. 
(In fact, if $n_1,n_2>0$ and $n_1\equiv n_2\pmod{p^e-1}$, then $g_{n_1,p}(x)=g_{n_2,p}(x)$ for all $x\in\text{Tr}_{p^e/p}^{-1}(0)$; see \cite[Proof of Lemma 3.5, Step $1^\circ$]{Hou-12}.)
So by \cite[Lemma 3.3]{Hou-12},
\[
g_{\alpha p+\beta,p}(x)=g_{(\alpha p+\beta)^\dagger,p}(x) =a_0x^{p^0}+(a_0+a_1)x^{p^1}+\cdots+(a_0+\cdots+a_{e-2})x^{p^{e-2}},
\]
for $x\in\text{Tr}_{p^e/p}^{-1}(0)$.
Therefore, for $x\in\Bbb F_{p^e}$,
\begin{equation}\label{0}
g_{n,p}(x)=
\begin{cases}
a_0x^{p^0}+(a_0+a_1)x^{p^1}+\cdots+(a_0+\cdots+a_{e-2})x^{p^{e-2}} &\text{if}\ \text{Tr}_{p^e/p}(x)=0,\cr
-x^\alpha &\text{if}\ \text{Tr}_{p^e/p}(x)\ne 0.
\end{cases}
\end{equation}
Clearly, $g_{n,p}$ maps $\text{Tr}_{p^e/p}^{-1}(0)$ to itself.
Write $\alpha=p^l+s\frac{p^e-1}{p-1}$, $s\in\Bbb Z$. We have
\begin{equation}\label{1}
x^\alpha=x^{p^l}\text{N}_{p^e/p}(x)^s,\qquad x\in\Bbb F_{p^e}.
\end{equation}
By \eqref{1}, ${\tt x}^\alpha$ maps $\{x\in\Bbb F_{p^e}: \text{Tr}_{p^e/p}(x)\ne 0\}$ to itself. We claim that $\text{gcd}(\alpha,p^e-1)=1$. By (i), $\text{gcd}(\alpha,\frac{p^e-1}{p-1})=1$. Also
\[
\begin{split}
\alpha\;&\equiv \alpha p+\beta-\beta\pmod{p-1}\cr
&\equiv (\alpha p+\beta)^\dagger-\beta\pmod{p-1}\cr
&\equiv w_p\bigl((\alpha p+\beta)^\dagger\bigr)-w_p\bigl(\beta\bigr)\pmod{p-1}\cr
&=1.
\end{split}
\]
So $\text{gcd}(\alpha,p-1)=1$. Hence the claim is proved. Now ${\tt x}^\alpha$ maps $\{x\in\Bbb F_{p^e}: \text{Tr}_{p^e/p}(x)\ne 0\}$ bijectively to itself. By \eqref{0}, $g_{n,p}$ is a PP of $\Bbb F_{p^e}$ if and only if the $p$-linearized polynomial $L=a_0{\tt x}^{p^0}+(a_0+a_1){\tt x}^{p^1}+\cdots+(a_0+\cdots+a_{e-2}){\tt x}^{p^{e-2}}$ is 1-1 on 
$\text{Tr}_{p^e/p}^{-1}(0)$. This happens if and only if $\text{gcd}(L({\tt x}^p-{\tt x}),{\tt x}^{p^e}-{\tt x})={\tt x}^p-{\tt x}$, which, by \cite[Theorem 3.62]{LN}, is equivalent to
\[
\text{gcd}(a_0+a_1{\tt x}+\cdots+a_{e-1}{\tt x}^{e-1},\ {\tt x}^e-1)={\tt x}-1.
\]

\medskip

$2^\circ$
Assume $p=2$. The only difference from $1^\circ$ is that in \eqref{0}, we have
\[
g_{n,2}(x)=
a_0x^{2^0}+(a_0+a_1)x^{2^1}+\cdots+(a_0+\cdots+a_{e-2})x^{2^{e-2}}+1 \quad \text{if}\ \text{Tr}_{2^e/2}(x)=0;
\]
see \cite[Lemma 3.3]{Hou-12}.
Since $e$ is even, $a_0{\tt x}^{2^0}+(a_0+a_1){\tt x}^{2^1}+\cdots+(a_0+\cdots+a_{e-2}){\tt x}^{2^{e-2}}+1$ still maps $\text{Tr}_{2^e/2}^{-1}(0)$ to itself.
\end{proof}


The next theorem is a generalization of \cite[Theorem 3.9]{Hou-12}.

\begin{thm}\label{T4.3}
Let $e>1$ and $n=\alpha(p^{0e}+\cdots+p^{(p-1)e})+\beta$, where 
\begin{itemize}
  \item [(i)] $0\le\alpha< p^e-1$, $\alpha\equiv p^l-1\pmod{\frac{p^e-1}{p-1}}$ for some $0\le l<e$;
  \item [(ii)] $\beta=b+(p-b)p$ for some $0<b<p$;
  \item [(iii)] $w_p\bigl((\alpha p+\beta)^\dagger\bigr)=p$.
\end{itemize}  
Write $(\alpha p+\beta)^\dagger=a_0p^0+\cdots+a_{e-1}p^{e-1}$, $0\le a_i\le p-1$. Then $(n,e;p)$ is desirable if and only if 
\[
\text{\rm gcd}(a_0+a_1{\tt x}+\cdots+a_{e-1}{\tt x}^{e-1},{\tt x}^e-1)={\tt x}-1.
\]
\end{thm}

\begin{proof}
The proof is similar to that of Theorem~\ref{T4.1}. Recall that for $x\in\Bbb F_{p^e}$, 
\[
g_{n,p}(x)=
\begin{cases}
g_{\alpha p+\beta,p}(x)&\text{if}\ \text{Tr}_{p^e/p}(x)=0,\cr
x^\alpha g_{\beta,p}(x)&\text{if}\ \text{Tr}_{p^e/p}(x)\ne0.
\end{cases}
\] 
By \cite[Lemma 3.3]{Hou-12}, we have 
\[
x^\alpha g_{\beta,p}(x)=x^\alpha bx=bx^{\alpha+1}=bx^{p^l}\text{N}_{p^e/p}(x)^s,\qquad x\in\Bbb F_{p^e},
\]
where $\alpha+1=p^l+s\frac{p^e-1}{p-1}$, $s\in \Bbb Z$. The rest of the proof is the same as that of Theorem~\ref{T4.1}.
\end{proof} 

For descriptions of the integers $\alpha,\beta$ satisfying the conditions in Theorems~\ref{T4.1} and \ref{T4.3}, see Appendix B.

\begin{rmk}\label{R4.4}
\rm
In Theorem~\ref{T4.3}, condition (ii) can be weakened as 
\[
\beta=b+b_0p+b_1p^{1+e}+\cdots+b_{p-1}p^{1+(p-1)e},
\]
where $0<b<p$, $b_i\ge 0$, $b+b_0+\cdots+b_{p-1}=p$, $b_1+2b_2+\cdots+(p-1)b_{p-1}\equiv 0\pmod p$. Under these conditions we have 
\[
\begin{split}
&g_{\beta,p}\cr
\equiv\,&b{\tt x}+(b+b_0)S_e+(b+b_0+b_1)S_e+\cdots+(b+b_0+\cdots+b_{p-1})S_e\pmod{{\tt x}^{p^e}-{\tt x}}\cr
=\,&b{\tt x}.
\end{split}
\]
So the proof of Theorem~\ref{T4.3} goes through.
\end{rmk}
  
\begin{lem}\label{L4.8}
Let $L$ be a $q$-linearized polynomial over $\Bbb F_q$ which is a PP of $\Bbb F_{q^e}$. Let $s\ge 0$ be an integer such that $\text{\rm gcd}(se+1,q-1)=1$. Then $\text{\rm N}_{q^e/q}({\tt x})^sL({\tt x})$ is a PP of $\Bbb F_{q^e}$.
\end{lem}

\begin{proof}
It suffices to show that $\text{N}_{q^e/q}({\tt x})^sL({\tt x})$ is 1-1 on $\Bbb F_{q^e}^*$. Assume $x,y\in\Bbb F_{q^e}^*$ such that 
\[
\text{N}_{q^e/q}(x)^sL(x)=\text{N}_{q^e/q}(y)^sL(y).
\]
Then
\[
L\bigl(\text{N}_{q^e/q}(x)^s x\bigr)=L\bigl(\text{N}_{q^e/q}(y)^s y\bigr).
\]
Since $L$ is a PP of $\Bbb F_{q^e}$, we have $\text{N}_{q^e/q}(x)^s x= \text{N}_{q^e/q}(y)^s y$, i.e., 
\[
\Bigl(\frac xy\Bigr)^{s(1+q+\cdots+q^{e-1})+1}=1.
\]
Since $\frac xy=\text{N}_{q^e/q}(\frac yx)^s\in\Bbb F_q^*$ and
$\text{gcd}\bigl(s(1+q+\cdots+q^{e-1})+1,\,q-1\bigr)=\text{\rm gcd}(se+1,q-1)=1$, we have $\frac xy=1$, i.e., $x=y$.
\end{proof}

\begin{thm}\label{T4.9}
Let $\alpha=s\frac{p^e-1}{p-1}$, $0\le s<p-1$, and $0<\beta<p^{pe}$ such that
\begin{itemize}
  \item [(i)] $w_p(\beta)=p$; 
  \item [(ii)] $w_p\bigl((\alpha p+\beta)^\dagger\bigr)=p$;
  \item [(iii)] $\text{\rm gcd}(se+1,p-1)=1$.
\end{itemize} 
Let $(\alpha p+\beta)^\dagger=a_0p^0+\cdots+a_{e-1}p^{e-1}$ and $\beta=b_0p^0+\cdots+b_{pe-1}p^{pe-1}$ be the base $p$ representations. Assume that
\begin{equation}\label{4.17}
\text{\rm gcd}(a_0+a_1{\tt x}+\cdots+a_{e-1}{\tt x}^{e-1}, {\tt x}^e-1)={\tt x}-1
\end{equation}
and 
\begin{equation}\label{4.18}
\text{\rm gcd}\bigl(b_0+(b_0+b_1){\tt x}+\cdots+(b_0+\cdots+b_{pe-2}){\tt x}^{pe-2}, {\tt x}^e-1)=1.
\end{equation}
Then $(n,e;p)$ is a desirable triple, where $n=\alpha(p^{0e}+p^{1e}+\cdots+p^{(p-1)e})+\beta$.
\end{thm}

\begin{proof}
For $x\in\Bbb F_{p^e}$, we have
\[
g_{n,p}(x)=
\begin{cases}
g_{\alpha p+\beta,p}(x)&\text{if}\ \text{Tr}_{p^e/p}(x)=0,\cr
x^\alpha g_{\beta,p}(x)=\text{N}_{p^e/p}(x)^s g_{\beta,p}(x)&\text{if}\ \text{Tr}_{p^e/p}(x)\ne 0.
\end{cases}
\]
By (ii) and \eqref{4.17}, $g_{\alpha p+\beta,p}$ is a permutation of $\text{Tr}_{p^e/p}^{-1}(0)$. By (i) and \eqref{4.18}, $g_{\beta,p}$ is a $p$-linearized PP of $\Bbb F_{p^e}$. By Lemma~\ref{L4.8}, 
$\text{N}_{p^e/p}({\tt x})^s g_{\beta,p}({\tt x})$ is a permutation of $\Bbb F_{p^e}$. It is easy to see that $\text{N}_{p^e/p}({\tt x})^s g_{\beta,p}({\tt x})$ maps $\Bbb F_{p^e}\setminus \text{Tr}_{p^e/p}^{-1}(0)$ to itself: for $x\in\Bbb F_{p^e}\setminus \text{Tr}_{p^e/p}^{-1}(0)$, we have $\text{Tr}_{p^e/p}\bigl(\text{N}_{p^e/p}(x)^s g_{\beta,p}(x)\bigr)=\text{N}_{p^e/p}(x)^s g_{\beta,p}\bigl(\text{Tr}_{p^e/p}(x)\bigr)\ne 0$ since $\text{Tr}_{p^e/p}(x)\ne 0$. Hence $g_{n,p}$ is a PP of $\Bbb F_{p^e}$.
\end{proof}

\begin{thm}\label{T4.7}
Let $e\ge 3$ and $n=(3^{e-1}+3^{e-2}-1)(1+3^e+3^{2e})+2+3^{e-2}+3^{2e-2}$. Then $(n,e;3)$ is a desirable triple.
\end{thm}

\begin{proof}
Let $\alpha=3^{e-1}+3^{e-2}-1$ and $\beta=2+3^{e-2}+3^{2e-2}$. Then
\[
(3\alpha+\beta)^\dagger=(\,\underbrace{0,\dots,0,2,1}_e\,)_3.
\]
By \eqref{1.2a} and \cite[Lemma 3.3]{Hou-12}, for $x\in\Bbb F_{3^e}$, we have
\[
g_{n,3}(x)=
\begin{cases}
-x^{3^{e-2}}&\text{if}\ \text{Tr}_{3^e/3}(x)=0,\cr
x^\alpha g_{\beta,3}(x) &\text{if}\ \text{Tr}_{3^e/3}(x)\ne 0.
\end{cases}
\]
It remains to show that ${\tt x}^\alpha g_{\beta,3}$ is a permutation of $\Bbb F_{3^e}\setminus\text{Tr}_{3^e/3}^{-1}(0)$.

We have 
\[
\begin{split}
g_{\beta,3}\,&=g_{3+3^{e-2},3}+S_{2e-2}\cdot g_{2+3^{e-2},3}\kern0.95cm \text{(by \cite[(4.1)]{Hou-12})}\cr
&=-1+S_{2e-2}\cdot(-S_{e-2}) \kern1.95cm \text{(by \cite[Lemma 3.3]{Hou-12})}\cr
&\equiv -1-(2S_e-S_2^{3^{e-2}})(S_e-S_2^{3^{e-2}})\pmod{{\tt x}^{3^e}-{\tt x}}\cr
&=-1+(S_e+S_2^{3^{e-2}})(S_e-S_2^{3^{e-2}})\cr
&=-1+S_e^2-(S_2^{3^{e-2}})^2\cr
&=-1+\text{Tr}_{3^e/3}({\tt x})^2-(S_2^{3^{e-2}})^2.
\end{split}
\]
So, for $x\in \Bbb F_{3^e}\setminus\text{Tr}_{3^e/3}^{-1}(0)$, 
\[
g_{\beta,3}(x)=-\bigl(S_2^{3^{e-2}}(x)\bigr)^2.
\]
Hence
\[
\begin{split}
-x^\alpha g_{\beta,3}(x)\,&=x^{3^{e-1}+3^{e-2}-1}\bigl(x^{3^{e-2}}+x^{3^{e-1}}\bigr)^2\cr
&=x^{3^{e-1}+3^{e-2}-1}\bigl(x^{2\cdot 3^{e-2}}+x^{2\cdot 3^{e-1}}-x^{3^{e-1}+3^{e-2}}\bigr)\cr
&=x^{3^{e-1}+3^{e-2}-1+5\cdot 3^{e-2}}\bigl(x^{-3^{e-1}}+x^{3^{e-2}}-x^{-3^{e-2}}\bigr)\cr
&=x^{-3^{e-1}}+x^{3^{e-2}}-x^{-3^{e-2}}\cr
&=y-y^{-1}+y^{-3}\kern1cm (y=x^{3^{e-2}}).
\end{split}
\]
It is known that ${\tt y}-{\tt y}^{-1}+{\tt y}^{-3}$ is a permutation of $\Bbb F_{3^e}\setminus\text{Tr}_{3^e/3}^{-1}(0)$ (\cite[Proof of Proposition 3.10]{Hou-12} or \cite[Lemma~5.1]{YDWP}).
This completes the proof of Theorem~\ref{T4.7}.
\end{proof}

\noindent{\bf Question.} Is it possible to generalize Theorem~\ref{T4.7} to an arbitrary $p$?

\begin{lem}\label{L4.10}
Let $n=(q-1)(q^{a_1e}+\cdots+q^{a_se})$, where $0\le a_1<\cdots<a_s<p$. Then for $x\in\Bbb F_{q^e}$ with $\text{\rm Tr}_{q^e/q}(x)\ne 0$, we have
\[
g_{n,q}(x)=-s.
\]
\end{lem}

We will need a few formulas for the proof of Lemma~\ref{L4.10}. First, from \cite[(4.1)]{Hou-12}, we have
\begin{equation}\label{3.5a}
(S_a-S_b)g_{n,q}=g_{n+q^a,q}-g_{n+q^b,q}.
\end{equation}
Also note that
\[
S_a-S_b\equiv S_{a-b}\pmod{{\tt x}^{q^e}-{\tt x}}\quad \text{if}\ b\equiv 0\ \text{or}\ a\pmod e.
\]
If $a<0$, we define $S_a=S_{pe+a}$.

\begin{proof}[Proof of Lemma~\ref{L4.10}]
We use induction on $s$. We will write $g_m$ for $g_{m,q}$. 

When $s=0$, the claim is trivially true. When $s=1$, $n=(q-1)q^{a_1e}$, and we have $g_n=-1$.

Now assume $s>1$. In the following calculation, ``$\equiv$'' means ``$\equiv\!\pmod{{\tt x}^{q^e}-{\tt x}}$''. We have
\[
\begin{split}
&S_{(a_2-a_1)e}\cdot g_n\cr
\equiv\;& g_{n+q^{a_2e}}-g_{n+q^{a_1e}}\cr
=\;&g_{(q-1)(q^{a_1e}+q^{a_3e}+\cdots+q^{a_se})+q^{a_2e+1}}-g_{(q-1)(q^{a_2e}+\cdots+q^{a_se})+q^{a_1e+1}}\cr
\equiv\;& g_{(q-1)(q^{a_3e}+\cdots+q^{a_se})+q^{a_1e+1}}+S_{(a_2-a_1)e+1}\cdot g_{(q-1)(q^{a_1e}+q^{a_3e}+\cdots+q^{a_se})}\cr
& -g_{(q-1)(q^{a_3e}+\cdots+q^{a_se})+q^{a_2e+1}}-S_{(a_1-a_2)e+1}\cdot g_{(q-1)(q^{a_2e}+\cdots+q^{a_se})}\cr
\equiv\;& S_{a_1e+1-(a_2e+1)}\cdot g_{(q-1)(q^{a_3e}+\cdots+q^{a_se})}-(s-1)S_{2(a_2-a_1)e}\kern 3mm\text{(induction hypothesis)}\cr
\equiv\;& -(s-2)S_{(a_1-a_2)e}-2(s-1)S_{(a_2-a_1)e}\kern 2.8cm\text{(induction hypothesis)}\cr
\equiv\;& [s-2-2(s-1)]S_{(a_2-a_1)e}\cr
=\;&-s S_{(a_2-a_1)e}.
\end{split}
\]
Since $S_{(a_2-a_1)e}(x)=(a_2-a_1)S_e(x)=(a_2-a_1)\text{Tr}_{q^e/q}(x)\ne 0$, we have $g_n(x)=-s$.
\end{proof}

\begin{thm}\label{T4.11}
Let $e>2$, $\alpha\ge 0$, $\beta=2(1+3^e)+3^{2e+2}$ such that 
\begin{itemize}
  \item [(i)] $w_3((3\alpha+\beta)^\dagger)=3$,
  \item [(ii)] $\alpha+3= 3^l+s\frac{3^e-1}2$, where $0\le l<e$, $s\ge 0$, and $se\equiv 0\pmod 2$.
\end{itemize} 
Let $n=\alpha(1+3^e+3^{2e})+\beta$. Write $(3\alpha+\beta)^\dagger=a_03^0+\cdots+a_{e-1}3^{e-1}$, $0\le a_i\le 2$.
Then $(n,e;3)$ is a desirable triple if and only 
\begin{equation}\label{4.19}
\text{\rm gcd}(a_0+\cdots+a_{e-1}{\tt x}^{e-1},{\tt x}^e-1)={\tt x}-1.
\end{equation}
\end{thm}

\begin{proof}
For $x\in\Bbb F_{3^e}$ we have
\[
g_{n,3}(x)=
\begin{cases}
g_{3\alpha+\beta,3}(x)&\text{if}\ \text{Tr}_{3^e/3}(x)=0,\cr
x^\alpha g_{\beta,3}(x)&\text{if}\ \text{Tr}_{3^e/3}(x)\ne 0.
\end{cases}
\]
$g_{3\alpha+\beta,3}$ maps $\text{Tr}_{3^e/3}^{-1}(0)$ to itself; it is 1-1 on $\text{Tr}_{3^e/3}^{-1}(0)$ if and only if \eqref{4.19} holds. 
Therefore it remains to show that ${\tt x}^\alpha g_{\beta,3}$ is a permutation of $\Bbb F_{3^e}\setminus\text{Tr}_{3^e/3}^{-1}(0)$.

For $x\in \Bbb F_{3^e}\setminus\text{Tr}_{3^e/3}^{-1}(0)$, we have
\[
\begin{split}
g_{\beta,3}(x)\,&=g_{2(1+3^e)+3^{2e+2},3}(x)\cr
&=g_{3+2\cdot 3^e,3}(x)+S_{2e+2}(x)\cdot g_{2(1+3^e),3}(x)\cr
&=S_e(x)-x+S_{2e+2}(x)(-2)\kern1.7cm \text{(by Lemma~\ref{L4.10})}\cr
&=S_e(x)-x-2(2S_e+S_2)\cr
&=-x+S_2\cr
&=x^3.
\end{split}
\]
Thus
\[
x^\alpha g_{\beta,3}(x)=x^{\alpha+3}=\text{N}_{3^e/3}(x)^sx^{3^l}.
\]
Clearly, $\text{N}_{3^e/3}({\tt x})^s{\tt x}^{3^l}$ maps $\Bbb F_{3^e}\setminus\text{Tr}_{3^e/3}^{-1}(0)$ to itself. Since $\text{gcd}(3^l+s\frac{3^e-1}2,3^e-1)=\text{gcd}(se+1,2)=1$, this map is bijective.
\end{proof}

\begin{exmp}\label{E4.12}\rm
Let $e=4$, $\alpha=1+3^2+2\cdot 3^3$, and $\beta=2(1+3^4)+3^{2\cdot 4+2}$.  Then the conditions in Theorem~\ref{T4.11} are satisfied:
We have $\alpha+3=3^3+\frac{3^4-1}2$, 
$(3\alpha+\beta)^\dagger=(0,0,2,1)_3$, which has weight $3$, and $\text{gcd}(2{\tt x}^2+{\tt x^3},{\tt x}^4-1)={\tt x}-1$. Thus $(n,4;3)$ is desirable, where 
\[
n=\alpha(1+3^4+3^{2\cdot 4})+\beta=(0,1,1,2,0,1,1,2,1,0,2,2)_3.
\]
This desirable triple is equivalent to the entry $(107765,4;3)$ in Table~\ref{Tb2}. 
\end{exmp}

\section{Sporadic Cases?}

\subsection{The triple $(101,3;3)$}\label{ss3.1}\

In Table 3 of \cite{Hou-12}, the first unexplained case of desirable triple is $(101,3;3)$, where $101=2\cdot 3^0+2\cdot 3^2+3^4$. The result of this subsection (Theorem \ref{T3.1}) suggests that this might be a sporadic case.

By 
\cite[Eq.~(4.1) and Lemma 3.3]{Hou-12}, we have
\[
\begin{split}
g_{101,3}({\tt x})\,&=g_{2\cdot 3^0+2\cdot 3^2+3^4,\,3}\cr
&=g_{3+2\cdot 3^2,\,3}+S_4\cdot g_{2+2\cdot 3^2,\,3}\cr
&={\tt x}^3+S_4\cdot(g_{3+3^2,\,3}+S_2\cdot g_{2+3^2,\,3})\cr
&={\tt x}^3+S_4\cdot\bigl(-1+S_2\cdot(-{\tt x}-{\tt x}^3)\bigr)\cr
&=-{\tt x}-{\tt x}^3-{\tt x}^3-{\tt x}^{3^2}+{\tt x}^{3^2}+{\tt x}-S_4\cdot\bigl(1+({\tt x}+{\tt x}^3)^2\bigr)\cr
&=-{\tt y}-{\tt y}^3+{\tt y}^{3^2}-({\tt y}+{\tt y}^{3^2})(1+{\tt y}^2)\cr
&={\tt y}+{\tt y}^3-{\tt y}^{11},
\end{split}
\]
where ${\tt y}={\tt x}+{\tt x}^3$, which is a PP of $\Bbb F_{3^3}$. We can further write
\[
\begin{split}
g_{101,3}({\tt x})\,&={\tt y}+{\tt y}^3+{\tt y}^{3^2}-{\tt y}^{3^2}-{\tt y}^{11}\cr
&=\text{Tr}_{3^3/3}({\tt y})-{\tt y}^8({\tt y}+{\tt y}^3+{\tt y}^{3^2})+{\tt y}^{17}\cr
&=(1-{\tt y}^8)\,\text{Tr}_{3^3/3}({\tt y})+{\tt y}^{17}.
\end{split}
\]
For $x'\in\Bbb F_{3^3}^*$ and $y=x'+{x'}^3$, we have
\[
g_{101,3}(x')=(1-y^8)\,\text{Tr}_{3^3/3}(y)+y^{17}=(1-x^2)\,\text{Tr}_{3^3/3}\Bigl(\frac 1x\Bigr)+x,
\]
where $x=y^{-9}=(x'+{x'}^3)^{-9}$.
So the fact that $g_{101,3}$ is a PP of $\Bbb F_{3^3}$ is equivalent to the fact that the function 
\begin{equation}\label{3.1}
f(x)= (1-x^2)\,\text{Tr}_{3^3/3}\Bigl(\frac 1x\Bigr)+x
\end{equation}
is a permutation of $\Bbb F_{3^3}^*$.

In next theorem (and its proof), we investigate some peculiar properties of $f$ in \eqref{3.1} as a function defined on $\Bbb F_{q^3}^*$.

\begin{thm}\label{T3.1}
$f$ is a permutation of $\Bbb F_{q^3}^*$ if and only if $q=3$.
\end{thm}

\begin{proof}
($\Leftarrow$) This part is easily verified by computer. However, to illustrate the peculiarity of the function $f$, we include a computer-free proof.

We will show that for every $z\in\Bbb F_{3^3}^*$, there exists an $x\in\Bbb F_{3^3}^*$ such that
\begin{equation}\label{3.3}
(1-x^2)\,\text{Tr}_{3^3/3}\Bigl(\frac 1x\Bigr)+x=z.
\end{equation}
(The solution $x$ of \eqref{3.3} will be explicitly determined in the proof.) 

If $\text{Tr}_{3^3/3}(\frac 1z)=0$, then $x=z$ is the desired solution. So assume $\text{Tr}_{3^3/3}(\frac 1z)\ne 0$. Replacing $z$ and $x$ by $-z$ and $-x$ in \eqref{3.3} if necessary, we may assume 
$\text{Tr}_{3^3/3}(\frac 1z)=1$. Then $z\in\Bbb F_{3^3}\setminus\Bbb F_3$ and 
\begin{equation}\label{3.4}
z-1=az^2(z+b)\qquad\text{for some}\ a\in\Bbb F_3^*,\ b\in\Bbb F_3.
\end{equation}
Eq.\! \eqref{3.3} has two cases:
\[
\text{Tr}_{3^3/3}\Bigl(\frac 1x\Bigr)=1,\qquad 1-x^2+x=z,
\]
and 
\[
\text{Tr}_{3^3/3}\Bigl(\frac 1x\Bigr)=-1,\qquad -(1-x^2)+x=z.
\]
Thus it suffices to show that one of the following systems has a solution in $\Bbb F_{3^3}^*$.
\begin{equation}\label{3.5}
\begin{cases}
x^2-x-1+z=0,\cr
\text{Tr}_{3^3/3}\Bigl(\displaystyle \frac 1x\Bigr)=1;
\end{cases}
\end{equation}
\begin{equation}\label{3.6}
\begin{cases}
x^2+x-1-z=0,\cr
\text{Tr}_{3^3/3}\Bigl(\displaystyle \frac 1x\Bigr)=-1.
\end{cases}
\end{equation}
The solutions of the quadratic equation in \eqref{3.5} are given by $x=-1+w$, where $w^2=-z-1$; the solutions of the quadratic equation in \eqref{3.6} are given by $x=1+u$, where $u^2=z-1$.

{\bf Case 1.} Assume $b=0$ in \eqref{3.4}. Since $z$ has degree $3$ over $\Bbb F_3$, we must have $a=1$. So 
\begin{equation}\label{3.7}
z-1=z^3.
\end{equation}
It follows that 
\[
-z^2=z^3-z^2-z+1=(z-1)(z^2-1)=(z-1)^2(z+1),
\]
i.e.,
\[
-z-1=\Bigl(\frac z{z-1}\Bigr)^2.
\]
Let $w=-\frac z{z-1}$. Then $x=-1+w=\frac{z+1}{z-1}$ is a root of the quadratic equation in \eqref{3.5}. Moreover,
\[
\text{Tr}_{3^3/3}\Bigl(\frac 1x\Bigr)=\text{Tr}_{3^3/3}\Bigl(\frac{z-1}{z+1}\Bigr)=\text{Tr}_{3^3/3}\Bigl(1+\frac 1{z+1}\Bigr)=1.
\]
(It follows easily from \eqref{3.7} that $\text{Tr}_{3^3/3}(\frac 1{z+1})=1$.)

{\bf Case 2.} Assume $b\ne 0$ in \eqref{3.4}. Since $z$ has degree $3$ over $\Bbb F_3$, we must have $b=1$.

{\bf Case 2.1.} Assume $a=1$. Then \eqref{3.4} becomes
\begin{equation}\label{3.8}
z-1=z^2(z+1).
\end{equation}
Thus 
\[
1=z^2(z+1)-(z+1)=(z+1)(z^2-1)=(z+1)^2(z-1),
\]
i.e.,
\[
z-1=\Bigl(\frac 1{z+1}\Bigr)^2.
\]
Let $u=\frac 1{z+1}$. Then $x=1+u=\frac{z-1}{z+1}$ is a root of the quadratic equation in \eqref{3.6}. Moreover,
\[
\text{Tr}_{3^3/3}\Bigl(\frac 1x\Bigr)=\text{Tr}_{3^3/3}\Bigl(\frac{z+1}{z-1}\Bigr)=\text{Tr}_{3^3/3}\Bigl(1-\frac 1{z-1}\Bigr)=-1.
\]
(It follows easily from \eqref{3.8} that $\text{Tr}_{3^3/3}(\frac 1{z-1})=1$.)

{\bf Case 2.2.} Assume $a=-1$. Then \eqref{3.4} becomes
\[
z-1=-z^2(z+1).
\]
It follows that
\begin{equation}\label{3.9}
(z+1)^3=-(z-1)^2,
\end{equation}
i.e.,
\[
(-z-1)^3=(z-1)^2.
\]
Let $w^3=z-1$. Then $x=-1+w$ is a root of the quadratic equation in \eqref{3.5}. Moreover,
\[
\text{Tr}_{3^3/3}\Bigl(\frac 1x\Bigr)=\text{Tr}_{3^3/3}\Bigl(\frac 1{x^3}\Bigr)=\text{Tr}_{3^3/3}\Bigl(\frac1{-1+w^3}\Bigr)=\text{Tr}_{3^3/3}\Bigl(\frac 1{z+1}\Bigr)=1.
\]
(It follows easily from \eqref{3.9} that $\text{Tr}_{3^3/3}(\frac 1{z+1})=1$.)

($\Rightarrow$) We show that if $q\ne 3$, then $f$ is a not a permutation of $\Bbb F_{q^3}^*$. 

In general,
\begin{equation}\label{3.2}
\begin{split}
f(x)\,&=(1-x^2)(x^{-1}+x^{-q}+x^{-q^2})+x,\cr
&=x^{-1}+x^{-q}+x^{-q^2}-x^{2-q}-x^{2-q^2}\cr
&=y+y^q+y^{q^2}-y^{q-2}-y^{q^2-2}\cr
&=g(y),
\end{split}
\end{equation}
where $y=x^{-1}\in \Bbb F_{q^3}^*$, and $g({\tt y})={\tt y}+{\tt y}^q+{\tt y}^{q^2}-{\tt y}^{q-2}-{\tt y}^{q^2-2}$.

First assume $q=2$. We have 
\[
g(y)=y^4+y+1,\qquad y\in\Bbb F_{2^3}^*.
\]
It is obvious that $g$ is not 1-1 on $\Bbb F_{2^3}^*$. 

Now Assume $q>3$. We show that $g$ is not a PP of $\Bbb F_{q^3}$. (Since $g(0)=0$, it follows from \eqref{3.2} that $f$ is not a permutation of $\Bbb F_{q^3}^*$.)

{\bf Case 1.} Assume $q>3$ is odd. We have
\[
g({\tt y})^{2q^2+2}\equiv 8{\tt y}^{q^3-1}+\text{terms of lower degree}\pmod{{\tt y}^{q^3}-{\tt y}}.
\]
(The complete expression of $g^{2q^2+2} \pmod{{\tt y}^{q^3}-{\tt y}}$ is given in Appendix C.)
By Hermite's criterion, $g$ is not a PP of $\Bbb F_{q^3}$.

{\bf Case 2.} Assume $q>3$ is even. We have
\[
g({\tt y})^{2q^2+q+3}\equiv {\tt y}^{q^3-1}+\text{terms of lower degree}\pmod{{\tt y}^{q^3}-{\tt y}}.
\]
(The complete expression of $g^{2q^2+q+3} \pmod{{\tt y}^{q^3}-{\tt y}}$ is given in Appendix C.)
By Hermite's criterion, $g$ is not a PP of $\Bbb F_{q^3}$.

\end{proof}
 
\subsection{The triple $(407,3;3)$}\label{ss3.2}\

The second unexplained case of desirable triple in Table 3 of \cite{Hou-12} is $(407,3;3)$, where $407=2\cdot 3^0+2\cdot 3^4+3^5$.
We have
\[
\begin{split}
&g_{407,3}({\tt x})\cr
=\,&g_{2\cdot 3^0+2\cdot 3^4+3^5,\,3}\cr
=\,&g_{3+2\cdot 3^4,\,3}+S_5\cdot g_{2+2\cdot 3^4,\,3}\cr
=\,&g_{3+2\cdot 3^4,\,3}+S_5\cdot(g_{3+3^4,\,3}+S_4\cdot g_{2+3^4,\,3})\cr
=\,&{\tt x}^3+{\tt x}^{3^2}+{\tt x}^{3^3}+S_5\cdot(-1+S_4\cdot(-{\tt x}-{\tt x}^3-{\tt x}^{3^2}-{\tt x}^{3^3}))\cr
\equiv\,&\text{Tr}_{3^3/3}({\tt x}) -S_5(1+S_4^2) \pmod {{\tt x}^{3^3}-{\tt x}}\cr
\equiv\,&\text{Tr}_{3^3/3}({\tt x}) +S_4^{3^2}(1+S_4^2) \pmod {{\tt x}^{3^3}-{\tt x}}\kern 1cm(S_5\equiv-S_4^{3^2}\pmod{{\tt x}^{3^3}-{\tt x}})\cr
\equiv\,&\text{Tr}_{3^3/3}({\tt y})+{\tt y}^{3^2}(1+{\tt y}^2)  \pmod {{\tt x}^{3^3}-{\tt x}},
\end{split}
\]
where ${\tt y}=S_4({\tt x})$, which is a PP of $\Bbb F_{3^3}$. We can further write
\[
\begin{split}
g_{407,3}({\tt x})\,&\equiv \text{Tr}_{3^3/3}({\tt y})+{\tt y}^8\bigl(\text{Tr}_{3^3/3}({\tt y})-{\tt y}^{3^2}\bigr) \pmod {{\tt x}^{3^3}-{\tt x}}\cr
&=(1+{\tt y}^8)\text{Tr}_{3^3/3}({\tt y})-{\tt y}^{17}.
\end{split}
\]
For $x'\in \Bbb F_{3^3}^*$, $y=S_4(x')$, we have
\[
g_{407,3}(x')=(1+y^8)\text{Tr}_{3^3/3}(y)-y^{17}=(1+x^2)\text{Tr}_{3^3/3}\Bigl(\frac 1x\Bigr)-x,
\]
where $x=y^{-9}=S_4(x')^{-9}$. The function $h(x)=(1+x^2)\text{Tr}_{3^3/3}(\frac 1x)-x$ is very similar to the function $f$ in \eqref{3.1}, but they do not seem to be related through a simple substitution. The behavior of $h$ is also similar to that of $f$.

\begin{thm}\label{T3.}
$h$ is a permutation of $\Bbb F_{q^3}^*$ if and only if $q=3$.
\end{thm}
  
\begin{proof}
The proof of the ``only if'' part is the same as in the proof of Theorem~\ref{T3.1}. 




For the ``if'' part, the calculation is a little different from that in the proof of Theorem~\ref{T3.1}. Let $z\in\Bbb F_{3^3}$. We try to solve the equation
\begin{equation}\label{}
(1+x^2)\text{Tr}_{3^3/3}\Bigl(\frac 1x\Bigr)-x=z
\end{equation}
in $\Bbb F_{3^3}^*$. 

If $\text{Tr}_{3^3/3}(\frac 1z)=0$, $x=-z$ is the solution. 
If $\text{Tr}_{3^3/3}(\frac 1z)\ne0$, we may assume $\text{Tr}_{3^3/3}(\frac 1z)=1$.
Then 
\begin{equation}\label{}
z-1=az^2(z+b),\quad (a,b)=(1,0),\ (1,1),\ (-1,1).
\end{equation}
We show that one of the following systems has a solution $x\in\Bbb F_{3^3}^*$:
\begin{equation}\label{3.10}
\begin{cases}
x^2-x+1-z=0,\cr
\displaystyle \text{Tr}_{3^3/3}\Bigl(\frac 1x\Bigr)=1;
\end{cases}
\end{equation}
\begin{equation}\label{3.11}
\begin{cases}
x^2+x+1+z=0,\cr
\displaystyle \text{Tr}_{3^3/3}\Bigl(\frac 1x\Bigr)=-1.
\end{cases}
\end{equation}
The solutions of the quadratic equation in \eqref{3.10} are $x=-1+w$, where $w^2=z$; the solutions of the quadratic equation in \eqref{3.11} are $x=1+u$, where $u^2=-z$.

\medskip

{\bf Case 1.} Assume $(a,b)=(1,0)$. Then $z-1=z^3$, from which we have $-z=(\frac{z-1}{z+1})^2$. Let $u=\frac{z-1}{z+1}$. Then $x=1+u=-\frac z{z+1}$ is a solution of the quadratic equation in \eqref{3.11}, and $\text{Tr}_{3^3/3}(\frac 1x)=\text{Tr}_{3^3/3}(-1-\frac 1z)=-1$.

\medskip

{\bf Case 2.} Assume $(a,b)=(1,1)$. Then $z-1=z^2(z+1)$, from which we have $(-z)^3=(z+1)^2$. Let $u^3=-(z+1)$. Then $x=1+u$ is a solution of the quadratic equation in \eqref{3.11}, and
\[
\text{Tr}_{3^3/3}\Bigl(\frac 1x\Bigr)=\text{Tr}_{3^3/3}\Bigl(\frac 1{x^3}\Bigr)=\text{Tr}_{3^3/3}\Bigl(\frac 1{1-u^3}\Bigr)=\text{Tr}_{3^3/3}\Bigl(-\frac 1z\Bigr)=-1.
\]

\medskip

{\bf Case 3.} Assume $(a,b)=(-1,1)$. Then $z-1=-z^2(z+1)$, from which we have $z=(\frac 1{z-1})^2$. Let $w=-\frac 1{z-1}$. Then $x=-1+w=\frac{-z}{z-1}$ is a solution of the quadratic equation in 
\eqref{3.10}, and  $\text{Tr}_{3^3/3}(\frac 1x)=\text{Tr}_{3^3/3}(-1+\frac 1z)=1$.
\end{proof}

\subsection{The triple $(91525,4;3)$}\label{ss3.3}\

This case is related to \cite[Theorem 3.10]{Hou-12}. We have $n=91525\sim_{(4,3)} \alpha(3^0+3^4+3^{2\cdot 4})-7$, where $\alpha=(2,2,1,1)_3=4+(1,1,1,1)_3$. Note that 
\begin{equation}\label{3.14}
(\alpha\cdot 3-7)^\dagger=(0,0,2,1)_3,
\end{equation}
and
\[
g_{-7,3}={\tt x}^{-4}(-{\tt x}+{\tt x}^{-1}-{\tt x}^{-3}).
\]
Thus for $x\in\Bbb F_{3^4}$,
\[
g_{n,3}(x)=
\begin{cases}
-x^{3^2}&\text{if}\ \text{Tr}_{3^4/3}(x)=0,\cr
x^\alpha g_{-7,3}(x)=\text{N}_{3^4/3}(x)(-x+x^{-1}-x^{-3})&\text{if}\ \text{Tr}_{3^4/3}(x)\ne 0.
\end{cases}
\]
We only have to show that $\text{N}_{3^4/3}({\tt x})(-{\tt x}+{\tt x}^{-1}-{\tt x}^{-3})$ permutes $\Bbb F_{3^4}\setminus\text{Tr}_{3^4/3}^{-1}(0)$. 
For $x\in\Bbb F_{3^4}\setminus\text{Tr}_{3^4/3}^{-1}(0)$, since $\text{N}_{3^4/3}(x)^2=1$, we have
\[
\begin{split}
\text{N}_{3^4/3}(x)(-x+x^{-1}-x^{-3})\,&=-x\text{N}_{3^4/3}(x)+\bigl(x\text{N}_{3^4/3}(x)\bigr)^{-1}-\bigl(x\text{N}_{3^4/3}(x)\bigr)^{-3}\cr
&=f\bigl(x\text{N}_{3^4/3}(x)\bigr),
\end{split}
\]
where
$f({\tt x})=-{\tt x}+{\tt x}^{-1}-{\tt x}^{-3}$. Recall that $f({\tt x})$ is a permutation of $\Bbb F_{3^4}\setminus\text{Tr}_{3^4/3}^{-1}(0)$ (\cite[Proof of Theorem 3.10]{Hou-12}). 
Clearly, ${\tt x}\text{N}_{3^4/s}({\tt x})={\tt x}^{1+\frac{3^4-1}2}$ permutes $\Bbb F_{3^4}\setminus\text{Tr}_{3^4/3}^{-1}(0)$. Hence $f\bigl({\tt x}\text{N}_{3^4/s}({\tt x})\bigr)$ 
permutes $\Bbb F_{3^4}\setminus\text{Tr}_{3^4/3}^{-1}(0)$.

A crucial argument in the above proof is that \eqref{3.14} gives $w_3((\alpha\cdot 3-7)^\dagger)=3$. However, this argument holds only for $e=4$. We have not found any generalization of the desirable triple
$(91525,4;3)$.


\section{The case $n=q^a-q^b-1$}

Our computer search turns out many desirable triples $(n,e;q)$ where $n$ is of the form $q^a-q^b-1$. We shall see that such desirable triples are also interesting theoretically. 

Assume $n>0$ and $n\equiv q^a-q^b-1\pmod{q^{pe}-1}$ for some integers $a,b\ge 0$. Of course, we may assume $0\le a,b<pe$. If $a=0$ or $b$, then $n\sim_{(e,q)}q^{pe}-2$, where $(q^{pe-2},e;q)$ is desirable if and only if $q>2$ \cite[Proposition 3.2 (i)]{Hou-12}. If $b=0$ and $a>0$, we have $n\equiv q^a-2\pmod{q^{pe}-1}$. By \cite[Proposition~2.1 and Lemma~ 3.3]{Hou-12},
\begin{equation}\label{5.1-new}
\begin{split}
g_{q^a-2,q}\,&=\frac 1{\tt x}(g_{q^a+q-2,q}-g_{q^a-1,q})\cr
&=\frac 1{\tt x}\Bigl[-1-\frac 1{\tt x}(g_{q^a+q-1,q}-g_{q^a,q})\Bigr]\cr
&=\frac 1{\tt x}\Bigl(-1+\frac{S_a}{\tt x}\Bigr)\cr
&=\frac{S_{a-1}^q}{{\tt x}^2}\cr
&={\tt x}^{q-2}+{\tt x}^{q^2-2}+\cdots+{\tt x}^{q^{a-1}-2}.
\end{split}
\end{equation}
For which $a$, $e$ and $q$ is $g_{q^a-2,q}$ a PP of $\Bbb F_{q^e}$? The complete answer is not known. We have the following conjecture.
 
\begin{conj}\label{Conj5.1}
Let $e\ge 2$ and $2\le a<pe$. Then $(q^a-2,e;q)$ is desirable if and only if
\begin{itemize}
  \item [(i)] $a=3$ and $q=2$, or
  \item [(ii)] $a=2$ and $\text{\rm gcd}(q-2,q^e-1)=1$.
\end{itemize}  
\end{conj} 

\noindent{\bf Note.} When $q$ is even,
\[
g_{q^a-2,q}=\Bigl(\frac{{\tt x}^{\frac 12q^1}+{\tt x}^{\frac 12q^2}+\cdots+{\tt x}^{\frac 12q^{a-1}}}{\tt x}\Bigr)^2,
\]
and the claim of the conjecture follows from Payne's Theorem \cite[\S8.5]{Hir98}, \cite{Hou04,Pay71a,Pay71}. 
For a general $q$, the ``if'' part is obvious. So for the conjecture, one only has to prove that if $q$ is odd, $e\ge 2$, and $a>2$, then $(q^a-2,e;q)$ is not desirable.

\medskip

Now assume $n>0$ and $n\equiv q^a-q^b-1\pmod{q^{pe}-1}$, where $0<a,b<pe$ and $a\ne b$. If $a<b$, we have 
\[
n\sim_{(e,q)}q^{pe-b}n\equiv q^{pe-b}(q^a-q^b-1)\equiv q^{pe+a-b}-q^{pe-b}-1 \pmod{q^{pe}-1},
\]
where $0<pe-b<pe+a-b<pe$. Therefore we may assume $0<b<a<pe$.

By \cite[Eq.~(4.1)]{Hou-12}, we have
\[
\begin{split}
S_bg_{q^a-q^b-1,q}\,&=g_{q^a-1,q}-g_{q^a-q^b,q}\cr
&=g_{q^a-1,q}-(g_{q^{a-b}-1,q})^{q^b}\cr
&=-\frac{S_a}{\tt x}+\Bigl(\frac{S_{a-b}}{\tt x}\Bigr)^{q^b}\cr
&=-\frac{S_a-S_{a-b}^{q^b}}{\tt x}+\Bigl(\frac 1{{\tt x}^{q^b}}-\frac 1{\tt x}\Bigr)S_{a-b}^{q^b}\cr
&=-\frac{S_b}{\tt x}-\frac{S_b^q-S_b}{{\tt x}^{q^b+1}} S_{a-b}^{q^b}.
\end{split}
\]
So
\begin{equation}\label{B}
g_{q^a-q^b-1,q}=-\frac 1{\tt x}-\frac{(S_b^{q-1}-1)S_{a-b}^{q^b}}{{\tt x}^{q^b+1}}.
\end{equation}
(Note that \eqref{B} also holds for $b=0$; see \eqref{5.1-new}.)
Assume $e\ge 2$. Write
\[
a-b=a_0+a_1e,\quad b=b_0+b_1e,
\]
where $a_0,a_1,b_0,b_1\in \Bbb Z$ and $0\le a_0,b_0<e$. Then from \eqref{B} we have
\begin{equation}\label{5.1}
g_{q^a-q^b-1,q}\equiv-{\tt x}^{q^e-2}-{\tt x}^{q^e-q^{b_0}-2}(a_1S_e+S_{a_0}^{q^{b_0}})\bigl((b_1S_e+S_{b_0})^{q-1}-1\bigr)\pmod{{\tt x}^{q^e}-{\tt x}}.
\end{equation}

\begin{cor}\label{C5.1} 
We have
\[
g_{q^2-q-1,q}=-{\tt x}^{q-2}.
\]
In particular, $(q^2-q-1,e;q)$ is desirable if and only if $q>2$ and $\text{\rm gcd}(q-2,q^e-1)=1$.
\end{cor}

\begin{proof} It follows from \eqref{B}.
\end{proof}

The following theorem is a generalization of \cite[Proposition 3.2 (i)]{Hou-12}.

\begin{thm}\label{T5.2}
Assume $e\ge 2$. Let $0<b<a< pe$. Then 
\begin{equation}\label{A}
g_{q^a-q^b-1,q}\equiv -{\tt x}^{q^e-2}\pmod{{\tt x}^{q^e}-{\tt x}}
\end{equation}
if and only if $a\equiv b\equiv 0\pmod e$. In particular, if $0<b<a< pe$, and $a\equiv b\equiv 0\pmod e$, then $(q^a-q^b-1,e;q)$ is a desirable triple.
\end{thm}

\begin{proof}
($\Leftarrow$)
In the notation of \eqref{5.1}, we have $a_0=b_0=0$ and $0<b_1<p$. So
\begin{equation}\label{C}
\begin{split}
g_{q^a-q^b-1,q}\,&\equiv-{\tt x}^{q^e-2}-{\tt x}^{q^e-3}a_1S_e\bigl((b_1S_e)^{q-1}-1\bigr)\pmod{{\tt x}^{q^e}-{\tt x}}\cr
&= -{\tt x}^{q^e-2}-{\tt x}^{q^e-3}a_1S_e(S_e^{q-1}-1)\cr
&=-{\tt x}^{q^e-2}-{\tt x}^{q^e-3}a_1(S_e^q-S_e)\cr
&\equiv -{\tt x}^{q^e-2}\pmod{{\tt x}^{q^e}-{\tt x}}.
\end{split}
\end{equation}

($\Rightarrow$)
Assume \eqref{A} holds. Then by \eqref{B},
\[
({\tt x}^{q^b}-{\tt x})S_{a-b}^{q^b}=(S_b^q-S_b)S_{a-b}^{q^b}\equiv 0\pmod{{\tt x}^{q^e}-{\tt x}}.
\]
For $f\in \Bbb F_q[{\tt x}]$, denote $\{x\in\overline{\Bbb F}_q:f(x)=0\}$ by $V(f)$, where $\overline{\Bbb F}_q$ is the algebraic closure of $\Bbb F_q$. Then $V({\tt x}^{q^e}-{\tt x})\subset V({\tt x}^{q^b}-{\tt x})\cup V(S_{a-b})$, i.e., $\Bbb F_{q^e}\subset\Bbb F_{q^b}\cup V(S_{a-b})$. Since $V(S_{a-b})$ is a vector space over $\Bbb F_q$, we must have $\Bbb F_{q^e}\subset\Bbb F_{q^b}$ or $\Bbb F_{q^e}\subset V(S_{a-b})$. However, since $0<a<pe$, 
\[
S_{a-b}=S_{a_1e+a_0}\equiv a_1S_e+S_{a_0}\not\equiv 0\pmod{{\tt x}^{q^e}-{\tt x}}.
\]
So we must have $\Bbb F_{q^e}\subset\Bbb F_{q^b}$. Hence $b\equiv 0\pmod e$. Now by \eqref{5.1} and the calculation in \eqref{C}, we have
\begin{equation}\label{D}
S_{a_0}(S_e^{q-1}-1)\equiv 0\pmod{{\tt x}^{q^e}-{\tt x}}.
\end{equation}
If $a_0>0$, then 
\[
\deg S_{a_0}(S_e^{q-1}-1)=(q-1)q^{e-1}+q^{a_0-1}=q^e-q^{e-1}+q^{a_0-1}<q^e,
\]
which is a contradiction to \eqref{D}. So we must have $a_0=0$, i.e., $a\equiv 0\pmod e$.
\end{proof}

\begin{rmk}
\rm
If $(q^a-q^b-1,2;q)$ is desirable, where $0<b<a<2p$ and $b\equiv 0\pmod 2$, then we must have $a\equiv 0\pmod 2$. Otherwise, with $e=2$, $a_0=1$, $b_0=0$ in \eqref{5.1}, we have
\[
g_{q^a-q^b-1,q}\equiv -{\tt x}^{q^2-2}-{\tt x}^{q^2-3}(a_1S_2+{\tt x})\bigl((b_1S_2)^{q-1}-1\bigr) \pmod{{\tt x}^{q^2}-{\tt x}}.
\]
Then $g_{q^a-q^b-1,q}(x)=0$ for every $x\in\Bbb F_{q^2}$ with $\text{Tr}_{q^2/q}(x)=0$, which is a contradiction.
\end{rmk}

The results of our computer search suggest that when $e\ge 3$, the only desirable triples $(q^a-q^b-1,e;q)$, $0<b<a<pe$, are those given by Corollary~\ref{C5.1} and Theorem~\ref{T5.2}.
 
\begin{conj}\label{}
Let $e\ge 3$ and $n=q^a-q^b-1$, $0<b<a<pe$. Then $(n,e;q)$ is desirable if and only if
\begin{itemize}
  \item [(i)] $a=2$, $b=1$, and $\text{\rm gcd}(q-2,q^e-1)=1$, or
  \item [(ii)] $a\equiv b\equiv 0\pmod e$.
\end{itemize}
\end{conj}

For the rest of this section, we will focus on desirable triples of the form $(q^a-q^b-1,2;q)$, $0<b<a<2p$.

\begin{thm}\label{T5.3}
Let $p$ be an odd prime and $q$ a power of $p$. 
\begin{itemize}
  \item [(i)] $\Bbb F_{q^2}\setminus\Bbb F_q$ consists of the roots of $({\tt x}-{\tt x}^q)^{q-1}+1$. 
  \item [(ii)] Let $0<i\le\frac 12(p-1)$ and $n=q^{p+2i}-q^p-1$. Then 
\[
g_{n,q}(x)=
\begin{cases}
(2i-1)x^{q-2}&\text{if}\ x\in\Bbb F_q,\vspace{2mm}\cr
\displaystyle \frac{2i-1}x+\frac{2i}{x^q}&\text{if}\ x\in\Bbb F_{q^2}\setminus\Bbb F_q.
\end{cases}
\]
\item[(iii)] For the $n$ in (ii), $(n,2;q)$ is desirable if and only if $4i\not\equiv 1\pmod p$.
\end{itemize}
\end{thm}

\begin{proof}
(i) We have
\[
({\tt x}^q-{\tt x})\bigl[({\tt x}-{\tt x}^q)^{q-1}+1\bigr]=-({\tt x}-{\tt x}^q)^q+{\tt x}^q-{\tt x}={\tt x}^{q^2}-{\tt x}.
\]
Hence the claim.

(ii) Let $e=2$, $a=p+2i$, $b=p$. In the notation of \eqref{5.1}, $a_0=0$, $a_1=i$, $b_0=1$, $b_1=\frac{p-1}2$. Thus
\[
\begin{split}
g_{n,q}\,&\equiv-{\tt x}^{q^2-2}-i{\tt x}^{q^2-q-2}S_2\Bigl[\Bigl(-\frac 12S_2+{\tt x}\Bigr)^{q-1}-1\Bigr] \pmod{{\tt x}^{q^2}-{\tt x}}\cr
&=-{\tt x}^{q^2-2}-i{\tt x}^{q^2-q-2}({\tt x}+{\tt x}^q)\bigl[({\tt x}-{\tt x}^q)^{q-1}-1\bigr].
\end{split}
\]
When $x\in\Bbb F_q$, $x-x^q=0$, so
\[
g_{n,q}(x)=-x^{q^2-2}+ix^{q^2-q-2}(x+x^q)=(2i-1)x^{q-2}.
\]
When $x\in\Bbb F_{q^2}\setminus\Bbb F_q$, by (i), $(x-x^q)^{q-1}=-1$. Thus
\[
\begin{split}
g_{n,q}(x)\,&=-x^{-1}+2ix^{q^2-q-2}(x+x^q)\cr
&=-x^{-1}+2ix^{q^2-q-1}+2ix^{q^2-2}\cr
&=(2i-1)x^{-1}+2ix^{-q}.
\end{split}
\]

(iii) Since $0<2i-1<p$, $(2i-1){\tt x}^{q-2}$ permutes $\Bbb F_q$. We claim that $(2i-1){\tt x}^{-1}+2i{\tt x}^{-q}$ maps $\Bbb F_{q^2}\setminus\Bbb F_q$ to itself. 
In fact, for $x\in\Bbb F_{q^2}\setminus\Bbb F_q$, 
\[
\Bigl[\frac{2i-1}x+\frac{2i}{x^q}-\Bigl(\frac{2i-1}x+\frac{2i}{x^q}\Bigr)^q\Bigr]^{q-1}=\Bigl(-\frac 1x+\frac 1{x^q}\Bigr)^{q-1}=\Bigl(\frac{x-x^q}{x^{q+1}}\Bigr)^{q-1}=-1
\]
since $(x-x^q)^{q-1}=-1$.

Therefore, $g_{n,q}$ is a PP of $\Bbb F_{q^2}$ if and only if $(2i-1){\tt x}^{-1}+2i{\tt x}^{-q}$ is 1-1 on $\Bbb F_{q^2}\setminus\Bbb F_q$, i.e., if and only if $(2i-1){\tt x}+2i{\tt x}^q$ is 1-1 on $\Bbb F_{q^2}\setminus\Bbb F_q$. So, it remains to show that $(2i-1){\tt x}+2i{\tt x}^q$ is 1-1 on $\Bbb F_{q^2}\setminus\Bbb F_q$ if and only if $4i\not\equiv 1\pmod p$.

($\Leftarrow$) Assume $4i\not\equiv 1\pmod p$. We claim that $(2i-1){\tt x}+2i{\tt x}^q$ is a PP of $\Bbb F_{q^2}$. Otherwise, there exists $0\ne x\in\Bbb F_{q^2}$ such that $(2i-1)x+2ix^q=0$. Then $x^{q-1}=-\frac{2i-1}{2i}$. Hence
\[
1=(x^{q-1})^{q+1}=\Bigl(-\frac{2i-1}{2i}\Bigr)^{q+1}=\Bigl(\frac{2i-1}{2i}\Bigr)^2.
\]
So $(2i-1)^2\equiv (2i)^2\pmod p$, i.e., $4i-1\equiv 0\pmod p$, which is a contradiction.

($\Rightarrow$) Assume $4i\equiv 1\pmod p$. Then $(2i-1){\tt x}+2i{\tt x}^q=2i({\tt x}^q-{\tt x})$, which is clearly not 1-1 on $\Bbb F_{q^2}\setminus\Bbb F_q$.
\end{proof}

\begin{thm}\label{T5.6a}
Let $p$ be an odd prime and $q$ a power of $p$. 
\begin{itemize}
  \item [(i)] Let $0<i\le\frac 12(p-1)$ and $n=q^{p+2i-1}-q^p-1$. Then 
\[
g_{n,q}(x)=
\begin{cases}
2(i-1)x^{q-2}&\text{if}\ x\in\Bbb F_q,\vspace{2mm}\cr
\displaystyle\frac{2i-1}x+\frac{2i-2}{x^q}&\text{if}\ x\in\Bbb F_{q^2}\setminus\Bbb F_q.
\end{cases}
\]
\item[(ii)] For the $n$ in (i), $(n,2;q)$ is desirable if and only if $i>1$ and $4i\not\equiv 3\pmod p$.
\end{itemize}
\end{thm}

\begin{proof}
Similar to the proof of Theorem~\ref{T5.3}.
\end{proof}

\begin{prop}\label{}
Let $p$ be an odd prime and $q=p^k$. Let $i>0$. If $i$ is even,
\[
g_{q^{p+i}-q^p-1,q}\equiv -{\tt x}^{q^2-2}-i{\tt x}^{q-2}\sum_{j=0}^{q-2}{\tt x}^{(q-1)j}
\pmod{{\tt x}^{q^2}-{\tt x}}.
\]
If $i$ is odd,
\[
g_{q^{p+i}-q^p-1,q}\equiv -{\tt x}^{q^2-q-1}-i{\tt x}^{q-2}\sum_{j=0}^{q-2}{\tt x}^{(q-1)j}
\pmod{{\tt x}^{q^2}-{\tt x}}.
\]
\end{prop}

\begin{proof}
Let $n=q^{p+i}-q^p-1$. Throughout the proof, ``$\equiv$'' means ``$\equiv\pmod{{\tt x}^{q^2}-{\tt x}}$''. 

\medskip

 {\bf Case 1.} Assume that $i$ is even. Let $e=2$, $a=p+i$, $b=p$. In the notation of \eqref{5.1}, $a_0=0$, $a_1=\frac i2$, $b_0=1$, $b_1=\frac{p-1}2$. By \eqref{5.1}, we have
\[
\begin{split}
g_{n,q}\,&\equiv-{\tt x}^{q^2-2}-{\tt x}^{q^2-q-2}\,\frac i2S_2\cdot\Bigl[\Bigl(-\frac 12S_2+S_1\Bigr)^{q-1}-1\Bigr]\cr
&=-{\tt x}^{q^2-2}-\frac i2{\tt x}^{q^2-q-2}({\tt x}+{\tt x}^q)\bigl(({\tt x}-{\tt x}^q)^{q-1}-1\bigr)\cr
&=-{\tt x}^{q^2-2}-\frac i2{\tt x}^{q^2-q-1}(1+{\tt x}^{q-1})\bigl({\tt x}^{q-1}(1-{\tt x}^{q-1})^{q-1}-1\bigr).
\end{split}
\]
Note that
\[
(1-{\tt x}^{q-1})^{q-1}=\frac{1-{\tt x}^{(q-1)q}}{1-{\tt x}^{q-1}}=\sum_{j=0}^{q-1}{\tt x}^{(q-1)j}.
\]
So
\[
\begin{split}
g_{n,q}\,&\equiv-{\tt x}^{q^2-2}-\frac i2{\tt x}^{q^2-q-1}(1+{\tt x}^{q-1})\Bigl[{\tt x}^{q-1}\sum_{j=0}^{q-1}{\tt x}^{(q-1)j}-1\Bigr]\cr
&=-{\tt x}^{q^2-2}-\frac i2{\tt x}^{q^2-q-1}\Bigl[\sum_{j=1}^q{\tt x}^{(q-1)j}+\sum_{j=2}^{q+1}{\tt x}^{(q-1)j}-1-{\tt x}^{q-1}\Bigr]\cr
&=-{\tt x}^{q^2-2}-\frac i2{\tt x}^{q^2-q-1}\cdot 2\sum_{j=2}^q{\tt x}^{(q-1)j}\cr
&=-{\tt x}^{q^2-2}-i{\tt x}^{q-2}\sum_{j=0}^{q-2}{\tt x}^{(q-1)j}.
\end{split}
\]

{\bf Case 2.} Assume that $i$ is odd. In the notation of \eqref{5.1}, $a_0=1$, $a_1=\frac{i-1}2$, $b_0=1$, $b_1=\frac{p-1}2$. By \eqref{5.1},
\[
\begin{split}
g_{n,q}\equiv\,&-{\tt x}^{q^2-2}-{\tt x}^{q^2-q-2}\Bigl(\frac{i-1}2S_2+S_1^q\Bigr)\Bigl[\Bigl(-\frac 12S_2+S_1\Bigr)^{q-1}-1\Bigr]\cr
=\,&-{\tt x}^{q^2-2}-{\tt x}^{q^2-q-2}\Bigl(-\frac 12 S_2+S_1^q\Bigr)\Bigl[\Bigl(-\frac 12S_2+S_1\Bigr)^{q-1}-1\Bigr]\cr
&-\frac i2{\tt x}^{q^2-q-2}S_2\cdot\Bigl[\Bigl(-\frac 12S_2+S_1\Bigr)^{q-1}-1\Bigr].
\end{split}
\]
In the above,
\[
\begin{split}
&-{\tt x}^{q^2-2}-{\tt x}^{q^2-q-2}\Bigl(-\frac 12 S_2+S_1^q\Bigr)\Bigl[\Bigl(-\frac 12S_2+S_1\Bigr)^{q-1}-1\Bigr]\cr
=\,&-{\tt x}^{q^2-2}-{\tt x}^{q^2-q-2}\,\frac 12\,({\tt x}^q-{\tt x})\bigl(({\tt x}-{\tt x}^q)^{q-1}-1\bigr)\cr
=\,&-{\tt x}^{q^2-2}-\frac 12{\tt x}^{q^2-q-2}\bigl(({\tt x}^q-{\tt x})^q-({\tt x}^q-{\tt x})\bigr)\cr
\equiv\,&-{\tt x}^{q^2-2}-\frac 12{\tt x}^{q^2-q-2}\cdot 2({\tt x}-{\tt x}^q)\cr
=\,&-{\tt x}^{q^2-q-1},
\end{split}
\]
and, by the calculation in Case 1,
\[
-\frac i2{\tt x}^{q^2-q-2}S_2\cdot\Bigl[\Bigl(-\frac 12S_2+S_1\Bigr)^{q-1}-1\Bigr]\equiv -i{\tt x}^{q-2}\sum_{j=0}^{q-2}{\tt x}^{(q-1)j}.
\]
So
\[
g_{n,q}\equiv-{\tt x}^{q^2-q-1}-i{\tt x}^{q-2}\sum_{j=0}^{q-2}{\tt x}^{(q-1)j}.
\]
\end{proof}

\begin{thm}\label{T5.6}
Let $q=2^s$, $n=q^3-q-1$. 

\begin{itemize}
  \item [(i)] For $x\in\Bbb F_{q^2}$,
\[
g_{n,q}(x)=
\begin{cases}
0&\text{if}\ x=0,\cr
x^{q-2}+\text{\rm Tr}_{q^2/q}(x^{-1})&\text{if}\ x\ne 0.
\end{cases}
\]

\item[(ii)] $g_{n,q}$ is a PP of $\Bbb F_{q^2}$ if and only if $s$ is even.
\end{itemize}
\end{thm}

\begin{proof}
(i) It is obvious that $g(0)=0$.
Let $0\ne x\in\Bbb F_{q^2}$. By \eqref{5.1} (with $a_0=0$, $a_1=1$, $b_0=1$, $b_1=0$),
\[
\begin{split}
g_{n,q}(x)\,&=x^{-1}+x^{-q-1}S_2(x)(x^{q-1}+1)\cr
&=x^{-1}+x^{-q-1}(x+x^q)(x^{q-1}+1)\cr
&=x^{-1}+x^{q-2}+x^{-q}\cr
&=x^{q-2}+\text{Tr}_{q^2/q}(x^{-1}).
\end{split}
\]

(ii) $1^\circ$ We show that for every $c\in\Bbb F_{q^2}^*$, the equation
\begin{equation}\label{5.2}
x^{q-2}+x^{-1}+x^{-q}=c
\end{equation}
has at most one solution $x\in\Bbb F_{q^2}^*$. 

Assume that $x\in\Bbb F_{q^2}^*$ is a solution of \eqref{5.2}. Then
\[
cx^{-q}=x^{-2}+x^{-q-1}+x^{-2q}=\text{N}_{q^2/q}(x^{-1})+\text{Tr}_{q^2/q}(x^{-2})\in\Bbb F_q.
\]
Let $t=c^{-q}x=(cx^{-q})^{-q}\in\Bbb F_q^*$. Then $x=tc^q$. Making this substitution in \eqref{5.2}, we have
\[
\frac 1t\bigl(c^{q(q-2)}+c^{-q}+c^{-1}\bigr)=c.
\]
So
\[
t=c^{-2}+c^{-2q}+c^{-q-1}.
\]
Hence $x$ is unique.

$2^\circ$ Assume $s$ is even. We show that
\begin{equation}\label{5.3}
x^{q-2}+\text{Tr}_{q^2/q}(x^{-1})=0
\end{equation}
has no solution in $\Bbb F_{q^2}^*$. Assume to the contrary that $x\in\Bbb F_{q^2}^*$ is a solution of \eqref{5.3}. Then $x^{q-2}\in\Bbb F_q$. Since $s$ is even, we have $\text{gcd}(q-2,q^2-1)=1$. So $x\in\Bbb F_q$. Then $\text{Tr}_{q^2/q}(x^{-1})=0$, and $x^{q-2}=0$, which is a contradiction.

$3^\circ$ Assume $s$ is odd. We show that \eqref{5.3} has a solution in $\Bbb F_{q^2}^*$. Let $x\in\Bbb F_{2^2}\setminus \Bbb F_2$. Then $x^2+x+1=0$ and $x^3=1$. So
\[
\begin{split}
x^{q-2}+\text{Tr}_{q^2/q}(x^{-1})\,&=x^{q-2}+x^{-1}+x^{-q}\cr
&=1+x^2+x\kern 1.5cm\text{(since $q\equiv 2\pmod 3$)}\cr
&=0.
\end{split}
\]
\end{proof}  

\begin{thm}\label{T5.9}
\begin{itemize}
  \item [(i)] Assume $q>2$. We have
\[
g_{q^{2i}-q-1,q}\equiv (i-1){\tt x}^{q^2-q-1}-i{\tt x}^{q-2}\pmod{{\tt x}^{q^2}-{\tt x}}.
\]

\item[(ii)] Assume that $q$ is odd. Then ${\tt x}^{q^2-q-1}+{\tt x}^{q-2}$ is a PP of $\Bbb F_{q^2}$ if and only if $q\equiv 1\pmod 4$.
\item[(iii)] Assume that $q$ is odd.  Then $(q^{p+1}-q-1,2;q)$ is desirable if and only if $q\equiv 1\pmod 4$.
\end{itemize}   
\end{thm}

\begin{proof}
In the notation of \eqref{5.1}, we have $e=2$, $a=2i$, $b=1$, $a_0=1$, $a_1=i-1$, $b_0=1$, $b_1=0$. Thus
\[
\begin{split}
g_{q^{2i}-q-1,q}\,&\equiv -{\tt x}^{q^2-2}-{\tt x}^{q^2-q-2}\bigl((i-1)S_2+{\tt x}^q\bigr)({\tt x}^{q-1}-1)\pmod{{\tt x}^{q^2}-{\tt x}}\cr
&=-{\tt x}^{q^2-2}-{\tt x}^{q^2-q-2}\bigl((i-1){\tt x}+i{\tt x}^q\bigr)({\tt x}^{q-1}-1)\cr
&=-{\tt x}^{q^2-2}-{\tt x}^{q^2-q-2}\bigl(-{\tt x}^q-(i-1){\tt x}+i{\tt x}^{2q-1}\bigr)\cr
&\equiv (i-1){\tt x}^{q^2-q-1}-i{\tt x}^{q-2}\pmod{{\tt x}^{q^2}-{\tt x}}.
\end{split}
\]

(ii) ($\Leftarrow$) Let $f={\tt x}^{q^2-q-1}+{\tt x}^{q-2}$. Then 
\[
f(x)=
\begin{cases}
0&\text{if}\ x=0,\cr
x^{-q}+x^{q-2}&\text{if}\ x\in\Bbb F_{q^2}^*.
\end{cases}
\]

$1^\circ$ We show that for every $c\in\Bbb F_{q^2}^*$, the equation
\begin{equation}\label{5.8}
x^{-q}+x^{q-2}=c
\end{equation}
has at most one solution $x\in\Bbb F_{q^2}^*$.

Assume $x\in\Bbb F_{q^2}^*$ is a solution of \eqref{5.8}. Then
\[
cx^{-q}=x^{-2q}+x^{-2}=\text{Tr}_{q^2/q}(x^{-2})\in\Bbb F_q.
\]
Let $t=c^{-q}x=(cx^{-q})^{-q}\in\Bbb F_q^*$. Then $x=tc^q$. So \eqref{5.8} becomes 
\[
\frac 1t\bigl(c^{-1}+c^{q(q-2)}\bigr)=c.
\]
Thus $t=c^{-2}+c^{-2q}$. Hence $x$ is unique.

$2^\circ$ We show that $x^{-q}+x^{q-2}=0$ has no solution $x\in \Bbb F_{q^2}^*$.

Assume that $x\in\Bbb F_{q^2}^*$ is a solution. Then $x^{2q-2}=-1$. Since $\frac 12(q+1)$ is odd, we have $-1=(x^{2q-2})^{\frac 12(q+1)}=x^{q^2-1}=1$, which is a contradiction.

($\Rightarrow$) Assume to the contrary that $q\equiv -1\pmod 4$. We show that $x^{-q}+x^{q-2}=0$ has a solution $x\in \Bbb F_{q^2}^*$. Since $4(q-1)\mid q^2-1$, there exists $x\in\Bbb F_{q^2}^*$ with $o(x)=4(q-1)$. Then $x^{2(q-1)}=-1$, i.e., $x^{-q}+x^{q-2}=0$.

(iii) It follows from (i) and (ii).
\end{proof}

Table~\ref{Tb1} in Appendix A contains all desirable triples $(q^a-q^b-1,2;q)$, $q\le 67$, $0<b<a<2p$, that are not covered by Corollary~\ref{C5.1} and 
Theorems~\ref{T5.2}, \ref{T5.3}, \ref{T5.6a}. Most entries in Table~\ref{Tb1} have not been explained. We conclude this section with a conjecture that grew out of Theorem~\ref{T5.9}.

\begin{conj}\label{}
Let $f={\tt x}^{q-2}+t{\tt x}^{q^2-q-1}$, $t\in\Bbb F_q^*$. Then $f$ is a PP of $\Bbb F_{q^2}$ if and only if one of the following occurs:
\begin{itemize}
  \item [(i)] $t=1$, $q\equiv 1\pmod 4$;
  \item [(ii)] $t=-3$, $q\equiv\pm 1\pmod{12}$;
  \item [(iii)] $t=3$, $q\equiv -1\pmod 6$.
\end{itemize} 
\end{conj}

\section{More Results (Mostly with Even $q$)}

This section primarily deals with the case of even characteristic. However, we remind the reader that in Lemmas~\ref{L6.7}, \ref{L6.10}, \ref{L6.12} and Theorem~\ref{T6.9}, the characteristic is assumed to be arbitrary.

\begin{thm}\label{T6.1a}
Let $q\ge 4$ be even, and let 
\[
n=1+q^{a_1}+q^{b_1}+\cdots+q^{a_{q/2}}+q^{b_{q/2}},
\]
where $a_i,b_i\ge 0$ are integers.
Then
\[
g_{n,q}=\sum_iS_{a_i}S_{b_i}+\sum_{i<j}(S_{a_i}+S_{b_i})(S_{a_j}+S_{b_j}).
\]
\end{thm}

\begin{proof}
We write $g_n$ for $g_{n,q}$. By \eqref{3.5a} we have 
\[
\begin{split}
g_n=\,&g_{1+2q^{a_1}+q^{a_2}+q^{b_2}+\cdots+q^{a_{q/2}}+q^{b_{q/2}}}+(S_{b_1}-S_{a_1})g_{1+q^{a_1}+q^{a_2}+q^{b_2}+\cdots+q^{a_{q/2}}+q^{b_{q/2}}}\cr
=\,&g_{1+2q^{a_1}+q^{a_2}+q^{b_2}+\cdots+q^{a_{q/2}}+q^{b_{q/2}}}\cr
&+(S_{a_1}+S_{b_1})(S_{a_1}+S_{a_2}+S_{b_2}+\cdots+S_{a_{q/2}}+S_{b_{q/2}})\cr
=\,&\cdots\cdots\cr
=\,&g_{1+2q^{a_1}+\cdots+2q^{a_{q/2}}}+\sum_{i=1}^{q/2}(S_{a_i}+S_{b_i})\Bigl(S_{a_i}+\sum_{j=i+1}^{q/2}(S_{a_j}+S_{b_j})\Bigr)\cr
=\,&S_{a_1}^2+\cdots+S_{a_{q/2}}^2+\sum_{i=1}^{q/2}(S_{a_i}+S_{b_i})\Bigl(S_{a_i}+\sum_{j=i+1}^{q/2}(S_{a_j}+S_{b_j})\Bigr)\cr
=\,&\sum_iS_{a_i}S_{b_i}+\sum_{i<j}(S_{a_i}+S_{b_i})(S_{a_j}+S_{b_j}).
\end{split}
\]
\end{proof}

\begin{cor}\label{C6.2}
Let $q\ge 4$ be even, and let
\[
n=t_0+2t_1q^{a_1}+\cdots+2t_kq^{a_k},
\]
where $t_0,\dots,t_k$ and $a_1,\dots,a_k$ are nonnegative integers with $t_0+2t_1+\cdots+2t_k=q+1$. Then
\[
g_{n,q}=(t_1S_{a_1}+\cdots+t_kS_{a_k})^2.
\]
In particular, $g_{n,q}$ is a PP of $\Bbb F_{q^e}$ if and only if 
\[
\text{\rm gcd}\Bigl(\sum_{i=1}^kt_i(1+{\tt x}+\cdots+{\tt x}^{a_i-1}),\; {\tt x}^e-1\Bigr)=1.
\]
\end{cor}

\begin{proof}
By Theorem~\ref{T6.1a},
\[
g_{n,q}=t_1S_{a_1}^2+\cdots+t_kS_{a_k}^2=(t_1S_{a_1}+\cdots+t_kS_{a_k})^2.
\]
The rest is obvious.
\end{proof}

In Theorem~\ref{T6.1a}, the mapping $g_{n,q}:\Bbb F_{q^e}\to\Bbb F_{q^e}$ is quadratic in the multivariate sense, i.e., with the identification $\Bbb F_{q^e}\cong\Bbb F_q^e$. In general, it is difficult to tell whether a quadratic mapping is bijective. However, in some cases, such as Corollary~\ref{C6.2}, $g_{n,q}$ can be reduced to a suitable form which allows a quick determination whether it is a PP.
Here are some additional examples of Theorem~\ref{T6.1a}:

\begin{exmp}\label{E6.4}
\rm
Let $q=2^s$, $s>1$, $e>1$ odd, $n=q^0+(q-1)q^1+q^2$. Then
\[
g_{n,q}=S_1^2+S_1S_2={\tt x}^{q+1},
\]
which is a PP of $\Bbb F_{q^e}$.
\end{exmp}

\begin{exmp}\label{E6.3}
\rm 
Let $q=4$, $e>1$, $n=q^0+q^1+q^e+q^{e+1}+q^a$, $a\ge 0$. Then
\[
\begin{split}
g_{n,q}\,&=S_1S_a+S_eS_{e+1}+(S_1+S_a)(S_e+S_{e+1})\cr
&\equiv S_1S_a+S_eS_{e+1}+(S_1+S_a)S_1 \pmod{{\tt x}^{q^e}-{\tt x}}\cr
&=S_1^2+S_eS_{e+1}\cr
&={\tt x}^2+{\tt x}\text{Tr}_{q^e/q}({\tt x})+\text{Tr}_{q^e/q}({\tt x})^2.
\end{split}
\]
We claim that when $e$ is odd, $g_{n,q}$ is a PP of $\Bbb F_{q^e}$.

Assume to the contrary that there exist $x,y\in\Bbb F_{q^e}$, $x\ne y$, such that $g_{n,q}(x)=g_{n,q}(y)$. From $\text{Tr}_{q^e/q}(g_{n,q}(x))=\text{Tr}_{q^e/q}(g_{n,q}(y))$, we derive that $\text{Tr}_{q^e/q}(x)=\text{Tr}_{q^e/q}(y)$ $=c$. Then the equation $g_{n,q}(x)=g_{n,q}(y)$ becomes 
\[
(x+y+c)(x+y)=0.
\]
So $x+y+c=0$. Thus $c=\text{Tr}_{q^e/q}(c)=\text{Tr}_{q^e/q}(x+y)=0$. Hence $(x+y)^2=0$, which is a contradiction.
\end{exmp}

\begin{lem}\label{L6.7}
Let $n=(q-1)q^a+(q-1)q^b$, where $a,b\ge 0$. Then
\[
g_{n,q}=-1-(S_b-S_a)^{q-1}.
\]
\end{lem}

\begin{proof}
If $a=b$, then $n=(q-2)q^a+q^{a+1}$. By \cite[Lemma 3.3]{Hou-12}, $g_{n,q}=-1$.

Now assume $a<b$. We have
\[
\begin{split}
(S_b-S_a)g_{n,q}\,&=g_{(q-1)q^a+q^{b+1},q}-g_{q^{a+1}+(q-1)q^b,q}\cr
&=-({\tt x}^{q^a}+\cdots+{\tt x}^{q^b})-({\tt x}^{q^{a+1}}+\cdots+{\tt x}^{q^{b-1}})\cr
&=-({\tt x}^{q^a}+\cdots+{\tt x}^{q^{b-1}})-({\tt x}^{q^a}+\cdots+{\tt x}^{q^{b-1}})^q\cr
&=-(S_b-S_a)-(S_b-S_a)^q.
\end{split}
\]
Thus $g_{n,q}=-1-(S_b-S_a)^{q-1}$.
\end{proof}

\begin{thm}\label{T6.8}
Let $q=2^s$, $s>1$, $e>0$, and $n=(q-1)q^0+(q-1)q^e+2q^a$, $a\ge 0$. Then
\[
g_{n,q}\equiv{\tt x}\text{\rm Tr}_{q^e/q}({\tt x})+\text{\rm Tr}_{q^e/q}({\tt x})^2 +S_a^2\cdot\bigl(1+\text{\rm Tr}_{q^e/q}({\tt x})^{q-1}\bigr)\pmod{{\tt x}^{q^e}-{\tt x}},
\]
Assume that $e$ is even and $\text{\rm gcd}(a,e)=1$. Then
$g_{n,q}$ is a PP of $\Bbb F_{q^e}$.
\end{thm}

\begin{proof}
Write $g_n=g_{n,q}$. We have
\[
\begin{split}
g_n=\,&g_{q+q^{a}+(q-1)q^e}+S_{a}\cdot g_{(q-1)q^0+q^{a}+(q-1)q^e}\cr
=\,&g_{q+q^{e+1}}+(S_{a}-S_e)g_{q+(q-1)q^e}
+S_{a}\cdot(g_{q+(q-1)q^e}+S_{a}\cdot g_{(q-1)q^0+(q-1)q^e})\cr
\equiv\,&S_e(S_e+S_1)+S_{a}^2(1+S_e^{q-1})\pmod{{\tt x}^{q^e}-{\tt x}}\kern 2.1cm \text{(Lemma~\ref{L6.7})}\cr
=\,&{\tt x}\text{Tr}_{q^e/q}({\tt x})+\text{Tr}_{q^e/q}({\tt x})^2+S_a^2\cdot\bigl(1+\text{Tr}_{q^e/q}({\tt x})^{q-1}\bigr).
\end{split}
\]

To prove that $g_{n}$ is a PP of $\Bbb F_{q^e}$, we assume that $g_n(x)=g_n(y)$, $x,y\in\Bbb F_{q^e}$, and try to show that $x=y$. From $\text{Tr}_{q^e/q}(g_n(x))=\text{Tr}_{q^e/q}(g_n(y))$, we derive that $\text{Tr}_{q^e/q}(x)=\text{Tr}_{q^e/q}(y)=c$. If $c=0$, the equation $g_n(x)=g_n(y)$ becomes $S_a(x)^2=S_a(y)^2$, i.e., $S_a(x+y)=0$. Since $\text{gcd}(1+{\tt x}+\cdots+{\tt x}^{a-1},{\tt x}^e+1)=1$, we have $x=y$.  If $c\ne 0$, the equation $g_n(x)=g_n(y)$ becomes $c(x+y)=0$, which also gives $x=y$.
\end{proof}

\begin{lem}\label{L6.10}
Let $a_1,\dots,a_q\ge 0$, and $n=(q-1)+q^{a_1}+\cdots+q^{a_q}$. Then 
\[
g_{n,q}=-S_1-S_{a_1}-\cdots-S_{a_q}-S_{a_1}\cdots S_{a_q}.
\]
\end{lem}

\begin{proof} 
Write $g_n=g_{n,q}$. We have
\[
\begin{split}
g_n\,&=g_{q+q^{a_2}+\cdots+q^{a_q}}+S_{a_1}\cdot g_{(q-1)+q^{a_2}+\cdots+q^{a_q}}\cr
&=g_{q+q^{a_2}+\cdots+q^{a_q}}+S_{a_1}\cdot(g_{q+q^{a_3}+\cdots+q^{a_q}}+S_{a_2}\cdot g_{(q-1)+q^{a_3}+\cdots+q^{a_q}})\cr
&= g_{q+q^{a_2}+\cdots+q^{a_q}}-S_{a_1}+S_{a_1}S_{a_2}\cdot g_{(q-1)+q^{q_3}+\cdots+q^{a_q}}\cr
&=\cdots\cdots\cr
&=g_{q+q^{a_2}+\cdots+q^{a_q}}-S_{a_1}+S_{a_1}\cdots S_{a_q}\cdot g_{q-1}\cr
&=-S_1-S_{a_2}-\cdots-S_{a_q}-S_{a_1}-S_{a_1}\cdots S_{a_q}.
\end{split}
\]
\end{proof}

\begin{thm}\label{T6.9}
Let $q=p^s$, $e>0$, $a>0$, and $n=(q-1)q^0+(q-1)q^e+q^a$. Then 
\begin{equation}\label{6.1}
g_{n,q}=-{\tt x}-S_a+\text{\rm Tr}_{q^e/q}({\tt x})-S_a\text{\rm Tr}_{q^e/q}({\tt x})^{q-1}.
\end{equation}
Assume that 
\begin{itemize}
  \item [(i)] $-2a-1+e\not\equiv 0\pmod p$;
  \item [(ii)] $\text{\rm gcd}({\tt x}^a+{\tt x}-2,{\tt x}^e-1)={\tt x}-1$;
  \item [(iii)] $\text{\rm gcd}(2{\tt x}^a+{\tt x}-3,{\tt x}^e-1)={\tt x}-1$.
\end{itemize}
Then $g_{n,q}$ is a PP of $\Bbb F_{q^e}$.
\end{thm}

\begin{proof}
Eq.~\eqref{6.1} follows from Lemma~\ref{L6.10}. To prove that $g_{n,q}$ is a PP of $\Bbb F_{q^e}$ under the given conditions, we assume that $g_{n,q}(x)=g_{n,q}(y)$, $x,y\in\Bbb F_{q^e}$,
and try to show that $x=y$. From $\text{Tr}_{q^e/q}(g_{n,q}(x))=\text{Tr}_{q^e/q}(g_{n,q}(y))$, we derive that
\[
(-2a-1+e)\bigl(\text{Tr}_{q^e/q}(x)-\text{Tr}_{q^e/q}(y)\bigr)=0.
\]
Since $-2a-1+e\not\equiv 0\pmod p$, we have $\text{Tr}_{q^e/q}(x)=\text{Tr}_{q^e/q}(y)=c$. 

If $c=0$, the equation $g_{n,q}(x)=g_{n,q}(y)$ becomes
\[
2(x-y)+(x-y)^q+\cdots+(x-y)^{q^{a-1}}=0.
\]
Since 
\[
\text{gcd}(2+{\tt x}+\cdots+{\tt x}^{a-1},1+{\tt x}+\cdots+{\tt x}^{e-1})=\frac 1{{\tt x}-1}\text{gcd}({\tt x}^a+{\tt x}-2,{\tt x}^e-1)=1,
\]
we must have $x-y=0$.

If $c\ne 0$, the equation $g_{n,q}(x)=g_{n,q}(y)$ becomes
\[
3(x-y)+2(x-y)^q+\cdots+2(x-y)^{q^{a-1}}=0.
\]
Since
\[
\text{gcd}(3+2{\tt x}+\cdots+2{\tt x}^{a-1},1+{\tt x}+\cdots+{\tt x}^{e-1})=\frac 1{{\tt x}-1}\text{gcd}(2{\tt x}^a+{\tt x}-3,{\tt x}^e-1)=1,
\]
we also have $x-y=0$.
\end{proof}

\begin{thm}\label{T6.11}
Let $q=2^s$, $s>1$, $e>0$, and let $n=(q-1)q^0+\frac q2 q^{e-1}+\frac q2 q^e$. We have 
\[
g_{n,q}={\tt x}+\text{\rm Tr}_{q^e/q}({\tt x})+{\tt x}^{\frac 12 q^e}\text{\rm Tr}_{q^e/q}({\tt x})^{\frac 12 q}.
\]
When $e$ is odd, $g_{n,q}$ is a PP of $\Bbb F_{q^e}$. 
\end{thm}

\begin{proof}
By Lemma~\ref{L6.10},
\[
\begin{split}
g_{n,q}\,&=S_1+S_{e-1}^{\frac q2}S_e^{\frac q2}\cr
&={\tt x}+(S_e^{\frac 12 q}+{\tt x}^{\frac 12 q^e})S_e^{\frac 12 q}\cr
&={\tt x}+\text{Tr}_{q^e/q}({\tt x})+{\tt x}^{\frac 12 q^e}\text{Tr}_{q^e/q}({\tt x})^{\frac 12 q}.
\end{split}
\]

Assume that $e$ is odd.
To prove that $g_{n,q}$ is a PP of $\Bbb F_{q^e}$, assume to the contrary that there exist $x,y\in\Bbb F_{q^e}$, $x\ne y$, such that $g_{n,q}(x)= g_{n,q}(y)$. From $\text{Tr}_{q^e/q}(g_{n,q}(x))=\text{Tr}_{q^e/q}(g_{n,q}(y))$, we derive that $\text{Tr}_{q^e/q}(x)=\text{Tr}_{q^e/q}(y)=a$. 
If $a=0$, the equation $g_{n,q}(x)=g_{n,q}(y)$ becomes $x=y$, which is a contradiction. If $a\ne 0$, the equation $g_{n,q}(x)=g_{n,q}(y)$ becomes
\[
(x+y)^{\frac 12 q^e}a^{\frac 12 q}=x+y,
\]
i.e., 
\[
(x+y)^{q^e-2}=a^{-1}.
\]
So $x+y=a$. Then $a=\text{Tr}_{q^e/q}(a)=\text{Tr}_{q^e/q}(x+y)=0$, which is a contradiction.
\end{proof}

\begin{thm}\label{sc}
Let $q=4$, $e>2$, and $n=3q^0+2q^2+2q^e$. We have
\[
g_{n,q}={\tt x}+({\tt x}+{\tt x}^q)^2\text{\rm Tr}_{q^e/q}({\tt x})^2.
\]
Assume $\text{\rm gcd}(1+{\tt x}+{\tt x}^3,{\tt x}^e+1)=1$. Then $g_{n,q}$ is a PP of $\Bbb F_{q^e}$.
\end{thm}

\begin{proof}
By Lemma~\ref{L6.10}, 
\[
g_{n,q}=S_1+S_2^2S_e^2={\tt x}+({\tt x}+{\tt x}^q)^2\text{Tr}_{q^e/q}({\tt x})^2.
\]
To prove that $g_{n,q}$ is a PP under the given condition, we assume that $g_{n,q}(x)=g_{n,q}(y)$, $x,y\in\Bbb F_{q^e}$, and will show that $x=y$. From $\text{Tr}_{q^e/q}(g_{n,q}(x))=\text{Tr}_{q^e/q}(g_{n,q}(y))$, we derive that $\text{Tr}_{q^e/q}(x)=\text{Tr}_{q^e/q}(y)=a$. If $a=0$, the equation $g_{n,q}(x)=g_{n,q}(y)$ becomes $x=y$. If $a\ne 0$, the equation $g_{n,q}(x)=g_{n,q}(y)$ becomes
\begin{equation}\label{6.2}
z=a^2(z+z^q)^2,
\end{equation}
where $z=x+y$. Substitute \eqref{6.2} into itself, we have
\[
z=a^2\bigl[a^2(z+z^q)^2+a^2(z+z^q)^{2q}\bigr]^2=(z+z^{q^2})^4=z^q+z^{q^3}.
\]
Since $\text{gcd}(1+{\tt x}+{\tt x}^3,{\tt x}^e+1)=1$, we have $z=0$, i.e., $x=y$.
\end{proof}

The next theorem is similar to Theorem~\ref{sc}. Its proof is almost identical to that of Theorem~\ref{sc} and is thus omitted.

\begin{thm}\label{nc3}
Let $q=4$, $e>2$, and $n=3q^0+2q^{e-2}+2q^e$. We have
\[
g_{n,q}={\tt x}+\text{\rm Tr}_{q^e/q}({\tt x})+({\tt x}^{q^{e-2}}+{\tt x}^{q^{e-1}})^2\text{\rm Tr}_{q^e/q}({\tt x})^2.
\]
Assume that $e>2$ is even and $\text{\rm gcd}(1+{\tt x}^2+{\tt x}^{e-3},{\tt x}^e+1)=1$. Then $g_{n,q}$ is a PP of $\Bbb F_{q^e}$.
\end{thm}

\begin{thm}\label{nc4}
Let $q=4$, $e>0$, and $n=q^0+q^e+2q^a+q^b$, $a,b\ge 0$. Then
\begin{equation}\label{6.3a}
g_{n,q}=S_a^2+S_bS_e.
\end{equation}
Assume that $a+b\not\equiv 0\pmod 2$ and
\[
\text{\rm gcd}\bigl({\tt x}^{2a+1}+{\tt x}+\epsilon({\tt x}^{2b}+1),\,({\tt x}+1)({\tt x}^e+1)\bigr)=({\tt x}+1)^2,
\]
for $\epsilon=0,1$. Then $g_{n,q}$ is a PP of $\Bbb F_{q^e}$.
\end{thm}

\begin{proof}
Eq.~\eqref{6.3a} follows from Theorem~\ref{T6.1a}.

To prove that $g_{n,q}$ is a PP of $\Bbb F_{q^e}$ under the given conditions, we assume that $g_{n,q}(x)=g_{n,q}(y)$, $x,y\in\Bbb F_{q^e}$, and try to show that $x=y$. From $S_e(g_{n,q}(x))=S_e(g_{n,q}(y))$ we derive that
\[
(a+b)\bigl(S_e(x)+S_e(y)\bigr)^2=0.
\]
So $S_e(x)=S_e(y)=c\in\Bbb F_q$. Now the equation $g_{n,q}(x)=g_{n,q}(y)$ becomes 
\[
S_a(z)^2=cS_b(z),
\]
where $z=x+y$. Thus
\begin{equation}\label{6.4a}
S_a(z)=\bigl(cS_b(z)\bigr)^{2q^{e-1}}=c^2S_b(z^{2q^{e-1}}).
\end{equation}
We iterate both sides of \eqref{6.4a} to get
\[
(S_a\circ S_a)(z)=c^2S_b\Bigl(\bigl(c^2S_b(z^{2q^{e-1}})\bigr)^{2q^{e-1}}\Bigr)=c^3(S_b\circ S_b)(z^{q^{e-1}}),
\]
i.e.,
\begin{equation}\label{6.5a}
\bigl[(S_a\circ S_a)(z)\bigr]^q=c^3(S_b\circ S_b)(z).
\end{equation}
Note that $c^3=0$ or $1$. The conventional associates of the $q$-polynomials $(S_a\circ S_a)^q$ and $S_b\circ S_b$ are $(1+{\tt x}+\cdots+{\tt x}^{a-1})^2\cdot{\tt x}$ and $(1+{\tt x}+\cdots+{\tt x}^{b-1})^2$, respectively \cite[\S3.4]{LN}. Since for $\epsilon=0,1$,
\[
\begin{split}
&\text{gcd}\bigl((1+{\tt x}+\cdots+{\tt x}^{a-1})^2\cdot{\tt x}+\epsilon(1+{\tt x}+\cdots+{\tt x}^{b-1})^2,\,1+{\tt x}+\cdots+{\tt x}^{e-1}\bigr)\cr
=\,&\frac 1{({\tt x}+1)^2} \text{\rm gcd}\bigl({\tt x}^{2a+1}+{\tt x}+\epsilon({\tt x}^{2b}+1),\,({\tt x}+1)({\tt x}^e+1)\bigr)\cr
=\,&1,
\end{split}
\]
it follows from \eqref{6.5a} that $z=0$, i.e., $x=y$. 
\end{proof}

\begin{lem}\label{L6.12}
Let $f:\Bbb F_p^n\to \Bbb F_p$ be a function, and assume that there exists $y\in \Bbb F_p^n$ such that $f(x+y)-f(x)$ is a nonzero constant for all $x\in\Bbb F_p^n$. Then
\[
\sum_{x\in\Bbb F_p^n}\zeta_p^{f(x)}=0\qquad (\zeta_p=e^{2\pi i/p}).
\]
\end{lem}

\begin{proof}
Assume $f(x+y)-f(x)=c\in\Bbb F_p^*$. We have 
\[
\sum_{x\in\Bbb F_p^n}\zeta_p^{f(x)}=\sum_{x\in\Bbb F_p^n}\zeta_p^{f(x+y)}=\sum_{x\in\Bbb F_p^n}\zeta_p^{f(x)+c}=\zeta^c\sum_{x\in\Bbb F_p^n}\zeta_p^{f(x)}.
\]
Since $\zeta^c\ne 1$, we have $\sum_{x\in\Bbb F_p^n}\zeta_p^{f(x)}=0$.
\end{proof}

\begin{rmk}\rm
If $f:\Bbb F_p^n\to \Bbb F_p$ is quadratic, then $\sum_{x\in\Bbb F_p^n}\zeta_p^{f(x)}=0$ if and only if there exists $y\in\Bbb F_p^n$ such that $f(x+y)-f(x)$ is a nonzero constant for all $x\in\Bbb F_p^n$. See \cite[Ch. VII, VIII]{Dic}, \cite[\S5.1]{Hir98}, \cite[\S6.2]{LN}.
\end{rmk}

\begin{thm}\label{T6.1}
Let $q=2^s$, $e=3k$, and $g={\tt x}^2+S_{4k}S_{2k}$. Then $g$ is a PP of $\Bbb F_{q^e}$.
\end{thm}

\begin{proof}
Note that $S_{4k}\equiv S_{2k}^{q^k}\pmod{{\tt x}^{q^e}-{\tt x}}$. So
\[
g\equiv{\tt x}^2+S_{2k}^{q^k+1}\pmod{{\tt x}^{q^e}-{\tt x}}.
\]
By \cite[Theorem 7.7]{LN}, it suffices to show that
\[
\sum_{x\in\Bbb F_{q^e}}(-1)^{\text{Tr}_{q^e/2}(ag(x))}=0
\]
for all $0\ne a\in\Bbb F_{q^e}$. We write $\text{Tr}=\text{Tr}_{q^e/2}$.

\medskip

{\bf Case 1.} Assume $\text{Tr}_{q^e/q^k}(a)\ne0$. By Lemma~\ref{L6.12}, it suffices to show that there exists $y\in \Bbb F_{q^e}$ such that $\text{Tr}\bigl(ag(x+y)-ag(x)\bigr)$ is a nonzero constant for all $x\in\Bbb F_{q^e}$.

Since $\text{Tr}_{q^e/q^k}(a)\ne0$, there exists $y\in\Bbb F_{q^k}$ such that
$\text{Tr}_{q^k/2}\bigl(y^2\text{Tr}_{q^e/q^k}(a)\bigr)\ne 0$.
Then for $x\in\Bbb F_{q^e}$, we have
\[
\begin{split}
&\text{Tr}\bigl(ag(x+y)-ag(x)\bigr)\cr
=\,&\text{Tr}\bigl[a\bigl(x^2+y^2+(S_{2k}(x)+S_{2k}(y))^{q^k+1}+x^2+S_{2k}(x)^{q^k+1}\bigr)\bigr]\cr
=\,&\text{Tr}(ay^2)\kern 3.8cm (S_{2k}(y)=0\ \text{since}\ y\in\Bbb F_{q^k})\cr
=\,&\text{Tr}_{q^k/2}\bigl(\text{Tr}_{q^e/q^k}(ay^2)\bigr)\cr
=\,&\text{Tr}_{q^k/2}\bigl(y^2\text{Tr}_{q^e/q^k}(a)\bigr),
\end{split}
\]
which is a nonzero constant. 

{\bf Case 2.} Assume $\text{Tr}_{q^e/q^k}(a)=0$. Then $a=b+b^{q^k}$ for some $b\in\Bbb F_{q^e}$. Since $a\ne 0$, we have $b\notin\Bbb F_{q^k}$. For $x\in\Bbb F_{q^e}$, we have
\[
\begin{split}
\text{Tr}\bigl(ag(x)\bigr)\,&=\text{Tr}\bigl(bg(x)+b^{q^k}g(x)\bigr)\cr
&=\text{Tr}\bigl(bg(x)+bg(x)^{q^{2k}}\bigr)\cr
&=\text{Tr}\bigl(b(g(x)+g(x)^{q^{2k}})\bigr).
\end{split}
\]
Note that
\[
\begin{split}
g({\tt x})+g({\tt x})^{q^{2k}}\,&\equiv {\tt x}^2+S_{2k}^{q^k+1}+{\tt x}^{2q^{2k}}+S_{2k}^{q^{2k}+1}\pmod{{\tt x}^{q^e}-{\tt x}}\cr
&={\tt x}^2+{\tt x}^{2q^{2k}}+S_{2k}\cdot(S_{2k}^{q^k}+S_{2k}^{q^{2k}})\cr
&\equiv {\tt x}^2+{\tt x}^{2q^{2k}}+S_{2k}^2 \pmod{{\tt x}^{q^e}-{\tt x}}\cr
&=({\tt x}+{\tt x}^{q^{2k}}+S_{2k})^2\cr
&=(S_{2k}^q)^2.
\end{split}
\]
So for $x\in\Bbb F_{q^e}$,
\[
\begin{split}
\text{Tr}\bigl(ag(x)\bigr)\,&=\text{Tr}\bigl(bS_{2k}(x)^{2q}\bigr)\cr
&=\text{Tr}\bigl(cS_{2k}(x)\bigr)\kern 3.3cm (b=c^{2q})\cr
&=\text{Tr}\bigl(c(x+x^q+\cdots+x^{q^{2k-1}})\bigr)\cr
&=\text{Tr}\bigl(x(c^{q^{3k}}+c^{q^{3k-1}}+\cdots+c^{q^{k+1}})\bigr).
\end{split}
\]
Since
\[
\begin{split}
&\text{gcd}({\tt x}^{k+1}+{\tt x}^{k+2}+\cdots+{\tt x}^{3k},\; {\tt x}^{3k}+1)\cr
=\,&\text{gcd}(1+{\tt x}+\cdots+{\tt x}^{2k-1},\; {\tt x}^{3k}+1)\cr
=\,&{\tt x}^k+1,
\end{split}
\]
we see that for $z\in \Bbb F_{q^{3k}}$, 
\[
z^{q^{3k}}+z^{q^{3k-1}}+\cdots+z^{q^{k+1}}=0 \Leftrightarrow z\in\Bbb F_{q^k}.
\]
Since $c\notin \Bbb F_{q^k}$, we have $c^{q^{3k}}+c^{q^{3k-1}}+\cdots+c^{q^{k+1}}\ne 0$. Therefore 
\[
\sum_{x\in\Bbb F_{q^e}}(-1)^{\text{Tr}(ag(x))}=\sum_{x\in\Bbb F_{q^e}}(-1)^{\text{Tr}(x(c^{q^{3k}}+c^{q^{3k-1}}+\cdots+c^{q^{k+1}}))}=0.
\]
\end{proof}

\begin{cor}\label{C6.19}
Let $e=3k$, $k\ge 1$, $q=2^s$, $s\ge 2$, and $n=(q-3)q^0+2q^1+q^{2k}+q^{4k}$. Then
\[
g_{n,q}\equiv{\tt x}^2+S_{2k}S_{4k}\pmod{{\tt x}^{q^e}-{\tt x}},
\]
and $g_{n,q}$ is a PP of $\Bbb F_{q^e}$.
\end{cor}

\begin{proof} We write $g_n$ for $g_{n,q}$. We have
\[
\begin{split}
g_n\,&=g_{(q-2)q^0+2q^1+q^{2k}}+S_{4k}\cdot g_{(q-3)q^0+2q^1+q^{2k}}\cr
&=g_{(q-1)q^0+2q^1}+S_{2k}\cdot g_{(q-2)q^0+2q^1}+S_{4k}S_{2k}\cr
&=g_{2q^1}+S_1\cdot g_{(q-1)q^0+q^1}+S_{4k}S_{2k}\cr
&={\tt x}^2+S_{4k}S_{2k}.
\end{split}
\]
It follows from Theorem~\ref{T6.1} that $g_n$ is a PP of $\Bbb F_{q^e}$.
\end{proof} 

\begin{conj}\label{}
Let $q=4$, $e=3k$, $k\ge 1$, and $n=3q^0+3q^{2k}+q^{4k}$. It is easy to see that
\[
g_{n,q}\equiv{\tt x}+S_{2k}+S_{4k}+S_{4k}S_{2k}^3\equiv {\tt x}+S_{2k}^{q^{2k}}+S_{2k}^{q^k+3}\pmod{{\tt x}^{q^e}-{\tt x}}.
\]
We conjecture that $g_{n,q}$ is a PP of $\Bbb F_{q^e}$.
\end{conj}




\clearpage

\appendix
\section{Tables}

%
%


\begin{table}[!h]
\caption{\hbox{Desirable triples $(q^a-q^b-1,2;q)$, $q\le 67$, $0<b<a<2p$,} \hbox{\kern 1.65cm $b$ odd, $b\ne p$, $(a,b)\ne (2,1)$}}\label{Tb1}
\vspace{-5mm}
\[
\begin{tabular}{|ll||ll||ll||ll||ll||ll||ll||ll|}
\hline
\hfil$a$\hfil & \hfil$b$\hfil & \hfil$a$\hfil & \hfil$b$\hfil & \hfil$a$\hfil & \hfil$b$\hfil & \hfil$a$\hfil & \hfil$b$\hfil & \hfil$a$\hfil & \hfil$b$\hfil & \hfil$a$\hfil & \hfil$b$\hfil & \hfil$a$\hfil & \hfil$b$\hfil & \hfil$a$\hfil & \hfil$b$\hfil\\ \hline
\hline
\multicolumn{2}{|l||}{$\boldsymbol {q=2}$} & 10 & 5 & 24 & 13 & 40 & 7 & 38 & 13 & 50 & 25 & 60 & 37 & 66 & 45 \\ 
-- & -- & 13 & 11 & 25 & 1 & 40 & 33 & 40 & 7 & 51 & 27 & 61 & 39 & 67 & 47 \\ 
  &   &   &   & 25 & 15 & 41 & 35 & 40 & 17 & 52 & 37 & 62 & 1 & 71 & 35 \\ 
\multicolumn{2}{|l||}{$\boldsymbol {q=2^2}$} & \multicolumn{2}{|l||}{$\boldsymbol {q=7^2}$} & 26 & 1 & 42 & 37 & 40 & 31 & 54 & 33 & 63 & 43 & 73 & 59 \\ 
3 & 1 & 6 & 1 & 27 & 19 & 43 & 39 & 41 & 3 & 57 & 7 & 64 & 11 & 74 & 27 \\ 
  &   & 8 & 1 & 28 & 21 & 45 & 13 & 41 & 19 & 58 & 41 & 64 & 45 & 74 & 61 \\ 
\multicolumn{2}{|l||}{$\boldsymbol {q=2^3}$} & 8 & 3 & 30 & 25 &   &   & 41 & 31 & 59 & 5 & 65 & 49 & 76 & 51 \\ 
-- & -- & 9 & 3 & 33 & 5 & \multicolumn{2}{|l||}{$\boldsymbol {q=29}$} & 42 & 3 & 61 & 47 & 66 & 49 & 76 & 65 \\ 
  &   & 10 & 5 &   &   & 15 & 11 & 42 & 21 & 62 & 49 & 67 & 51 & 77 & 67 \\ 
\multicolumn{2}{|l||}{$\boldsymbol {q=2^4}$} & 12 & 5 & \multicolumn{2}{|l||}{$\boldsymbol {q=19}$} & 21 & 3 & 46 & 29 & 62 & 55 & 69 & 3 & 78 & 9 \\ 
3 & 1 & 12 & 9 & 17 & 9 & 26 & 21 & 49 & 35 & 63 & 39 & 70 & 57 & 78 & 69 \\ 
  &   & 13 & 11 & 23 & 7 & 30 & 1 & 49 & 37 & 64 & 39 & 70 & 65 & 79 & 65 \\ 
\multicolumn{2}{|l||}{$\boldsymbol {q=2^5}$} &   &   & 25 & 11 & 31 & 19 & 49 & 43 & 64 & 53 & 71 & 59 & 80 & 47 \\ 
-- & -- & \multicolumn{2}{|l||}{$\boldsymbol {q=11}$} & 26 & 13 & 32 & 5 & 50 & 9 & 65 & 7 & 72 & 47 & 80 & 73 \\ 
  &   & 6 & 1 & 30 & 21 & 32 & 27 & 50 & 37 & 67 & 53 & 72 & 61 & 82 & 77 \\ 
\multicolumn{2}{|l||}{$\boldsymbol {q=2^6}$} & 10 & 1 & 30 & 23 & 33 & 7 & 51 & 39 & 69 & 63 & 77 & 33 & 83 & 79 \\ 
3 & 1 & 13 & 3 & 31 & 17 & 34 & 5 & 55 & 41 & 70 & 65 & 78 & 73 & 85 & 19 \\ 
  &   & 17 & 13 & 31 & 23 & 34 & 9 & 55 & 47 & 71 & 67 & 80 & 5 & 85 & 59 \\ 
\multicolumn{2}{|l||}{$\boldsymbol {q=3}$} & 18 & 13 & 33 & 17 & 36 & 3 & 57 & 51 & 73 & 71 & 80 & 77 & 85 & 83 \\ 
-- & -- & 19 & 15 & 34 & 29 & 36 & 13 & 58 & 53 &   &   &   &   &   &   \\ 
  &   & 20 & 5 & 35 & 9 & 41 & 23 & 59 & 13 & \multicolumn{2}{|l||}{$\boldsymbol {q=41}$} & \multicolumn{2}{|l||}{$\boldsymbol {q=43}$} & \multicolumn{2}{|l|}{$\boldsymbol {q=47}$} \\ 
\multicolumn{2}{|l||}{$\boldsymbol {q=3^2}$} & 20 & 17 & 36 & 5 & 42 & 25 & 60 & 5 & 12 & 7 & 20 & 11 & 18 & 3 \\ 
3 & 1 &   &   & 36 & 33 & 44 & 1 & 60 & 57 & 31 & 1 & 21 & 11 & 20 & 9 \\ 
4 & 1 & \multicolumn{2}{|l||}{$\boldsymbol {q=13}$} & 37 & 35 & 46 & 33 & 61 & 59 & 31 & 5 & 32 & 13 & 24 & 1 \\ 
5 & 1 & 12 & 1 &   &   & 46 & 35 &   &   & 42 & 1 & 38 & 31 & 29 & 7 \\ 
  &   & 14 & 1 & \multicolumn{2}{|l||}{$\boldsymbol {q=23}$} & 47 & 35 & \multicolumn{2}{|l||}{$\boldsymbol {q=37}$} & 42 & 33 & 39 & 11 & 37 & 31 \\ 
\multicolumn{2}{|l||}{$\boldsymbol {q=3^3}$} & 15 & 3 & 10 & 7 & 52 & 19 & 19 & 15 & 44 & 5 & 46 & 5 & 44 & 21 \\ 
-- & -- & 18 & 5 & 12 & 1 & 52 & 45 & 29 & 23 & 46 & 5 & 46 & 39 & 45 & 37 \\ 
  &   & 18 & 9 & 21 & 13 & 53 & 23 & 32 & 19 & 46 & 9 & 49 & 11 & 46 & 1 \\ 
\multicolumn{2}{|l||}{$\boldsymbol {q=5}$} & 19 & 5 & 22 & 1 & 54 & 49 & 34 & 21 & 49 & 29 & 51 & 15 & 49 & 3 \\ 
6 & 1 & 22 & 17 & 25 & 3 & 55 & 51 & 36 & 1 & 52 & 21 & 55 & 23 & 50 & 5 \\ 
8 & 1 & 25 & 23 & 26 & 5 & 56 & 5 & 38 & 1 & 53 & 1 & 58 & 29 & 50 & 29 \\ 
  &   &   &   & 27 & 21 & 56 & 53 & 38 & 15 & 53 & 23 & 58 & 41 & 51 & 7 \\ 
\multicolumn{2}{|l||}{$\boldsymbol {q=5^2}$} & \multicolumn{2}{|l||}{$\boldsymbol {q=17}$} & 32 & 17 & 57 & 15 & 39 & 3 & 54 & 25 & 59 & 31 & 54 & 13 \\ 
4 & 1 & 11 & 1 & 34 & 21 &   &   & 41 & 7 & 55 & 33 & 60 & 33 & 54 & 51 \\ 
6 & 1 & 15 & 7 & 35 & 31 & \multicolumn{2}{|l||}{$\boldsymbol {q=31}$} & 42 & 5 & 57 & 31 & 60 & 39 & 57 & 41 \\ 
7 & 3 & 18 & 1 & 37 & 3 & 22 & 3 & 42 & 9 & 58 & 15 & 61 & 35 & 61 & 27 \\ 
9 & 7 & 18 & 7 & 37 & 11 & 28 & 21 & 43 & 11 & 58 & 33 & 62 & 37 & 62 & 13 \\ 
  &   & 22 & 5 & 37 & 27 & 29 & 21 & 48 & 21 & 59 & 5 & 62 & 53 & 62 & 29 \\ 
\multicolumn{2}{|l||}{$\boldsymbol {q=7}$} & 22 & 9 & 39 & 31 & 35 & 7 & 48 & 45 & 60 & 27 & 65 & 61 & 64 & 33 \\ 
\hline
\end{tabular}
\]
\end{table}


\clearpage 

\addtocounter{table}{-1}
\begin{table}[h]
\caption{continued}
\vspace{-5mm}

\[
\begin{tabular}{|ll||ll||ll||ll||ll||ll||ll||ll|}
\hline
\hfil$a$\hfil & \hfil$b$\hfil & \hfil$a$\hfil & \hfil$b$\hfil & \hfil$a$\hfil & \hfil$b$\hfil & \hfil$a$\hfil & \hfil$b$\hfil & \hfil$a$\hfil & \hfil$b$\hfil & \hfil$a$\hfil & \hfil$b$\hfil & \hfil$a$\hfil & \hfil$b$\hfil & \hfil$a$\hfil & \hfil$b$\hfil\\ \hline
\hline

68 & 41 & 58 & 5 & 98 & 19 & 77 & 35 & 115 & 73 & 90 & 39 & 50 & 47 & 110 & 35 \\ 
68 & 57 & 58 & 9 & 98 & 89 & 78 & 37 & 116 & 5 & 90 & 57 & 53 & 47 & 110 & 85 \\ 
70 & 7 & 59 & 11 & 99 & 91 & 83 & 7 & 116 & 113 & 93 & 25 & 70 & 5 & 113 & 91 \\ 
70 & 45 & 59 & 29 & 100 & 67 & 85 & 39 &    &   & 94 & 25 & 73 & 17 & 115 & 71 \\ 
73 & 39 & 59 & 37 & 100 & 93 & 85 & 51 & \multicolumn{2}{|l||}{$\boldsymbol {q=61}$} & 94 & 65 & 74 & 13 & 116 & 29 \\ 
73 & 51 & 60 & 13 & 101 & 95 & 87 & 35 & 38 & 33 & 97 & 73 & 74 & 57 & 116 & 97 \\ 
75 & 55 & 61 & 15 & 102 & 97 & 90 & 61 & 52 & 39 & 98 & 51 & 75 & 15 & 117 & 23 \\ 
76 & 33 & 62 & 17 & 103 & 37 & 90 & 81 & 59 & 51 & 98 & 73 & 77 & 19 & 118 & 23 \\ 
76 & 57 & 63 & 57 & 103 & 55 & 91 & 63 & 60 & 1 & 99 & 75 & 79 & 23 & 118 & 101 \\ 
77 & 59 & 66 & 25 & 103 & 99 & 92 & 65 & 62 & 1 & 100 & 67 & 85 & 35 & 119 & 23 \\ 
77 & 65 & 67 & 43 & 105 & 11 & 94 & 69 & 62 & 49 & 100 & 77 & 85 & 47 & 119 & 61 \\ 
77 & 67 & 68 & 29 &     &    & 94 & 71 & 63 & 3 & 101 & 91 & 85 & 69 & 119 & 103 \\ 
79 & 63 & 71 & 11 & \multicolumn{2}{|l||}{$\boldsymbol {q=59}$} & 95 & 71 & 64 & 5 & 102 & 25 & 87 & 17 & 120 & 13 \\ 
82 & 33 & 72 & 37 & 16 & 13 & 96 & 63 & 64 & 37 & 102 & 81 & 87 & 39 & 120 & 105 \\ 
82 & 69 & 72 & 49 & 20 & 3 & 96 & 73 & 66 & 5 & 103 & 83 & 88 & 13 & 121 & 107 \\ 
83 & 71 & 73 & 5 & 23 & 17 & 97 & 75 & 66 & 9 & 109 & 95 & 88 & 41 & 122 & 63 \\ 
84 & 59 & 75 & 21 & 24 & 15 & 98 & 77 & 67 & 25 & 110 & 97 & 89 & 63 & 122 & 109 \\ 
84 & 73 & 75 & 43 & 30 & 1 & 99 & 79 & 68 & 13 & 111 & 29 & 90 & 45 & 124 & 99 \\ 
85 & 75 & 77 & 47 & 31 & 9 & 101 & 9 & 69 & 15 & 113 & 103 & 90 & 83 & 124 & 113 \\ 
86 & 13 & 78 & 13 & 39 & 27 & 103 & 87 & 71 & 19 & 115 & 101 & 91 & 47 & 126 & 33 \\ 
86 & 77 & 78 & 49 & 50 & 13 & 104 & 89 & 74 & 25 & 115 & 107 & 94 & 53 & 126 & 117 \\ 
87 & 79 & 80 & 1 & 56 & 21 & 104 & 101 & 74 & 55 & 116 & 83 & 95 & 35 & 127 & 113 \\ 
89 & 83 & 81 & 35 & 58 & 1 & 105 & 7 & 75 & 27 & 116 & 109 & 95 & 55 & 129 & 123 \\ 
90 & 85 & 82 & 57 & 61 & 3 & 106 & 51 & 79 & 59 & 117 & 7 & 96 & 57 & 130 & 125 \\ 
91 & 27 & 82 & 79 & 61 & 27 & 106 & 93 & 81 & 39 & 118 & 113 & 97 & 59 & 131 & 127 \\ 
   &    & 83 & 59 & 61 & 33 & 107 & 95 & 82 & 41 & 120 & 5 & 98 & 61 & 133 & 131 \\ 
\multicolumn{2}{|l||}{$\boldsymbol {q=53}$} & 85 & 63 & 63 & 7 & 108 & 83 & 84 & 45 & 120 & 117 & 98 & 77 &  &  \\ 
27 & 23 & 88 & 13 & 66 & 13 & 108 & 97 & 84 & 79 & 121 & 119 & 99 & 77 &  &  \\ 
32 & 3 & 88 & 69 & 66 & 19 & 109 & 31 & 85 & 47 &    &    & 101 & 97 &  &  \\ 
50 & 21 & 91 & 29 & 67 & 15 & 110 & 25 & 86 & 17 & \multicolumn{2}{|l||}{$\boldsymbol {q=67}$} & 102 & 31 &  &  \\ 
51 & 43 & 92 & 23 & 69 & 19 & 110 & 101 & 86 & 49 & 40 & 17 & 102 & 69 &  &  \\ 
54 & 1 & 92 & 77 & 73 & 27 & 113 & 107 & 87 & 39 & 43 & 11 & 103 & 71 &  &  \\ 
57 & 7 & 94 & 81 & 76 & 33 & 114 & 109 & 89 & 35 & 48 & 31 & 109 & 83 &  &  \\ 

\hline
\end{tabular}
\]
\end{table}

%
%
\clearpage 

\begin{table}[h]
\caption{Desirable triples $(n,e;3)$, $e\le 4$, $w_3(n)>3$}\label{Tb2}
\vspace{-5mm}
\[
\begin{tabular}{|c|r|l|c|}
\hline
$e$ & $n \hfil$ & $3$-adic digits of $n$ & reference \\ \hline
\hline

1&17&{2\ 2\ 1}&  \cite{Hou-12} Prop 3.1  \\  \hline

2&71&{2\ 2\ 1\ 2}&  \cite{Hou-12} Prop 3.2 (i)  \\  \hline

2&95&{2\ 1\ 1\ 0\ 1}&  \cite{Hou-12} Table 2 No.2  \\  \hline

2&101&{2\ 0\ 2\ 0\ 1}&  \cite{Hou-12} Table 2 No.2   \\  \hline

2&103&{1\ 1\ 2\ 0\ 1}&  \cite{Hou-12} Table 2 No.2   \\  \hline

2&119&{2\ 0\ 1\ 1\ 1}&  \cite{Hou-12} Table 2 No.2   \\  \hline

2&151&{1\ 2\ 1\ 2\ 1}&  \cite{Hou-12} Table 2 No.5   \\  \hline

2&197&{2\ 2\ 0\ 1\ 2}&  \cite{Hou-12} Prop 3.2 (ii)  \\  \hline

2&485&{2\ 2\ 2\ 2\ 2\ 1}&  \cite{Hou-12} Prop 3.1  \\  \hline

3&101&{2\ 0\ 2\ 0\ 1}&  $\spadesuit$ \S\ref{ss3.1}  \\  \hline

3&407&{2\ 0\ 0\ 0\ 2\ 1}&  $\spadesuit$ \S\ref{ss3.2}  \\  \hline

3&475&{1\ 2\ 1\ 2\ 2\ 1}&    \\  \hline

3&605&{2\ 0\ 1\ 1\ 1\ 2}&    \\  \hline

3&619&{1\ 2\ 2\ 1\ 1\ 2}&    \\  \hline

3&671&{2\ 1\ 2\ 0\ 2\ 2}&    \\  \hline

3&701&{2\ 2\ 2\ 1\ 2\ 2}&  \cite{Hou-12} Prop 3.2 (i) \\  \hline

3&761&{2\ 1\ 0\ 1\ 0\ 0\ 1}&  \cite{Hou-12} Table 2 No.2  \\  \hline

3&769&{1\ 1\ 1\ 1\ 0\ 0\ 1}&  \cite{Hou-12} Table 2 No.2 \\  \hline

3&775&{1\ 0\ 2\ 1\ 0\ 0\ 1}&  \cite{Hou-12} Table 2 No.2  \\  \hline

3&779&{2\ 1\ 2\ 1\ 0\ 0\ 1}&    \\  \hline

3&785&{2\ 0\ 0\ 2\ 0\ 0\ 1}&  \cite{Hou-12} Table 2 No.2  \\  \hline

3&787&{1\ 1\ 0\ 2\ 0\ 0\ 1}&  \cite{Hou-12} Table 2 No.2  \\  \hline

3&827&{2\ 2\ 1\ 0\ 1\ 0\ 1}&    \\  \hline

3&839&{2\ 0\ 0\ 1\ 1\ 0\ 1}&  \cite{Hou-12} Table 2 No.2  \\  \hline

3&847&{1\ 0\ 1\ 1\ 1\ 0\ 1}&  \cite{Hou-12} Table 2 No.2  \\  \hline

3&925&{1\ 2\ 0\ 1\ 2\ 0\ 1}&  \cite{Hou-12} Table 2 No.5 \\  \hline

3&1003&{1\ 1\ 0\ 1\ 0\ 1\ 1}&  \cite{Hou-12} Table 2 No.2  \\  \hline

3&1007&{2\ 2\ 0\ 1\ 0\ 1\ 1}&   \cite{Hou-12} Thm 3.10 \\  \hline

3&1009&{1\ 0\ 1\ 1\ 0\ 1\ 1}&  \cite{Hou-12} Table 2 No.2  \\  \hline

3&1097&{2\ 2\ 1\ 1\ 1\ 1\ 1}&    \\  \hline

3&1175&{2\ 1\ 1\ 1\ 2\ 1\ 1}&    \\  \hline

3&1247&{2\ 1\ 0\ 1\ 0\ 2\ 1}&    \\  \hline

3&1423&{1\ 0\ 2\ 1\ 2\ 2\ 1}&    \\  \hline

3&1519&{1\ 2\ 0\ 2\ 0\ 0\ 2}&  \cite{Hou-12} Table 2 No.4  \\  \hline

3&1739&{2\ 0\ 1\ 1\ 0\ 1\ 2}&    \\  \hline

3&1753&{1\ 2\ 2\ 1\ 0\ 1\ 2}&    \\  \hline

3&1915&{1\ 2\ 2\ 1\ 2\ 1\ 2}&    \\  \hline

3&2021&{2\ 1\ 2\ 2\ 0\ 2\ 2}&  \cite{Hou-12} Thm 3.9  \\  \hline

3&2117&{2\ 0\ 1\ 0\ 2\ 2\ 2}&    \\  \hline

3&2131&{1\ 2\ 2\ 0\ 2\ 2\ 2}&  \cite{Hou-12} Prop 3.2 (ii)  \\  \hline

\multicolumn{4}{l} {Possible sporadic cases are indicated with $\spadesuit$. \phantom{$\displaystyle\int$}}\\
\end{tabular}
\]
\end{table}

\clearpage 

\addtocounter{table}{-1}
\begin{table}[h]
\caption{continued}
\vspace{-5mm}
\[
\begin{tabular}{|c|r|l|c|}
\hline
$e$ & $n \hfil$ & $3$-adic digits of $n$ & reference \\ \hline
\hline

3&2537&{2\ 2\ 2\ 0\ 1\ 1\ 0\ 1}&    \\  \hline

3&2723&{2\ 1\ 2\ 1\ 0\ 2\ 0\ 1}&    \\  \hline

3&2819&{2\ 0\ 1\ 2\ 1\ 2\ 0\ 1}&    \\  \hline

3&2897&{2\ 2\ 0\ 2\ 2\ 2\ 0\ 1}&    \\  \hline

3&3137&{2\ 1\ 0\ 2\ 2\ 0\ 1\ 1}&    \\  \hline

3&3317&{2\ 1\ 2\ 2\ 1\ 1\ 1\ 1}&    \\  \hline

3&3361&{1\ 1\ 1\ 1\ 2\ 1\ 1\ 1}&    \\  \hline

3&3517&{1\ 2\ 0\ 1\ 1\ 2\ 1\ 1}&    \\  \hline

3&3551&{2\ 1\ 1\ 2\ 1\ 2\ 1\ 1}&    \\  \hline

3&3559&{1\ 1\ 2\ 2\ 1\ 2\ 1\ 1}&    \\  \hline

3&3833&{2\ 2\ 2\ 0\ 2\ 0\ 2\ 1}&    \\  \hline

3&4019&{2\ 1\ 2\ 1\ 1\ 1\ 2\ 1}&    \\  \hline

3&4253&{2\ 1\ 1\ 1\ 1\ 2\ 2\ 1}&    \\  \hline

3&4261&{1\ 1\ 2\ 1\ 1\ 2\ 2\ 1}&    \\  \hline

3&5093&{2\ 2\ 1\ 2\ 2\ 2\ 0\ 2}&    \\  \hline

3&5507&{2\ 2\ 2\ 2\ 1\ 1\ 1\ 2}&    \\  \hline

3&5557&{1\ 1\ 2\ 1\ 2\ 1\ 1\ 2}&    \\  \hline

3&5665&{1\ 1\ 2\ 2\ 0\ 2\ 1\ 2}&    \\  \hline

3&5719&{1\ 1\ 2\ 1\ 1\ 2\ 1\ 2}&    \\  \hline

3&13121&{2\ 2\ 2\ 2\ 2\ 2\ 2\ 2\ 1}&  \cite{Hou-12} Prop 3.1 \\  \hline

4&173&{2\ 0\ 1\ 0\ 2}&  \cite{Hou-12} Table 2 No.3 \\  \hline

4&1477&{1\ 0\ 2\ 0\ 0\ 0\ 2}&  \cite{Hou-12} Table 2 No.3  \\  \hline

4&6479&{2\ 2\ 2\ 2\ 1\ 2\ 2\ 2}&  \cite{Hou-12} Prop 3.2 (i)  \\  \hline

4&6647&{2\ 1\ 0\ 0\ 1\ 0\ 0\ 0\ 1}&  \cite{Hou-12} Table 2 No.2  \\  \hline

4&6653&{2\ 0\ 1\ 0\ 1\ 0\ 0\ 0\ 1}&  \cite{Hou-12} Table 2 No.2  \\  \hline

4&6655&{1\ 1\ 1\ 0\ 1\ 0\ 0\ 0\ 1}&  \cite{Hou-12} Table 2 No.2  \\  \hline

4&6661&{1\ 0\ 2\ 0\ 1\ 0\ 0\ 0\ 1}&  \cite{Hou-12} Table 2 No.2  \\  \hline

4&6671&{2\ 0\ 0\ 1\ 1\ 0\ 0\ 0\ 1}&  \cite{Hou-12} Table 2 No.2  \\  \hline

4&6679&{1\ 0\ 1\ 1\ 1\ 0\ 0\ 0\ 1}&  \cite{Hou-12} Table 2 No.2  \\  \hline

4&6725&{2\ 0\ 0\ 0\ 2\ 0\ 0\ 0\ 1}&  \cite{Hou-12} Table 2 No.2 \\  \hline

4&6727&{1\ 1\ 0\ 0\ 2\ 0\ 0\ 0\ 1}&  \cite{Hou-12} Table 2 No.2  \\  \hline

4&6733&{1\ 0\ 1\ 0\ 2\ 0\ 0\ 0\ 1}&  \cite{Hou-12} Table 2 No.2  \\  \hline

4&6751&{1\ 0\ 0\ 1\ 2\ 0\ 0\ 0\ 1}&  \cite{Hou-12} Table 2 No.2  \\  \hline

4&6887&{2\ 0\ 0\ 0\ 1\ 1\ 0\ 0\ 1}&  \cite{Hou-12} Table 2 No.2  \\  \hline

4&6895&{1\ 0\ 1\ 0\ 1\ 1\ 0\ 0\ 1}&  \cite{Hou-12} Table 2 No.2  \\  \hline

4&7135&{1\ 2\ 0\ 0\ 1\ 2\ 0\ 0\ 1}&  \cite{Hou-12} Table 2 No.5  \\  \hline

4&7373&{2\ 0\ 0\ 0\ 1\ 0\ 1\ 0\ 1}&  \cite{Hou-12} Table 2 No.2  \\  \hline

4&7375&{1\ 1\ 0\ 0\ 1\ 0\ 1\ 0\ 1}&  \cite{Hou-12} Table 2 No.2  \\  \hline

4&7381&{1\ 0\ 1\ 0\ 1\ 0\ 1\ 0\ 1}&  \cite{Hou-12} Table 2 No.2  \\  \hline

4&7399&{1\ 0\ 0\ 1\ 1\ 0\ 1\ 0\ 1}&  \cite{Hou-12} Table 2 No.2  \\  \hline

4&8119&{1\ 0\ 2\ 0\ 1\ 0\ 2\ 0\ 1}&  \cite{Hou-12} Table 2 No.5  \\  \hline

4&8831&{2\ 0\ 0\ 0\ 1\ 0\ 0\ 1\ 1}&  \cite{Hou-12} Table 2 No.2  \\  \hline

\end{tabular}
\]
\end{table}

\clearpage 

\addtocounter{table}{-1}
\begin{table}[!h]
\caption{continued}
\vspace{-5mm}
\[
\begin{tabular}{|c|r|l|c|}
\hline
$e$ & $n \hfil$ & $3$-adic digits of $n$ & reference \\ \hline
\hline

4&8839&{1\ 0\ 1\ 0\ 1\ 0\ 0\ 1\ 1}&  \cite{Hou-12} Table 2 No.2  \\  \hline

4&8855&{2\ 2\ 2\ 0\ 1\ 0\ 0\ 1\ 1}&   \cite{Hou-12} Thm 3.10 \\  \hline

4&11071&{1\ 0\ 0\ 2\ 1\ 0\ 0\ 2\ 1}&  \cite{Hou-12} Table 2 No.5  \\  \hline

4&17717&{2\ 1\ 0\ 2\ 2\ 0\ 0\ 2\ 2}&   \cite{Hou-12} Thm 3.9 \\  \hline

4&19519&{1\ 2\ 2\ 2\ 0\ 2\ 2\ 2\ 2}&  \cite{Hou-12} Prop 3.2 (ii)  \\  \hline

4&26725&{1\ 1\ 2\ 2\ 2\ 1\ 0\ 0\ 1\ 1}&    \\  \hline

4&28669&{1\ 1\ 2\ 2\ 2\ 0\ 0\ 1\ 1\ 1}&    \\  \hline

4&29525&{2\ 1\ 1\ 1\ 1\ 1\ 1\ 1\ 1\ 1}&  Thm~\ref{T4.3}  \\  \hline

4&36997&{1\ 2\ 0\ 2\ 0\ 2\ 2\ 1\ 2\ 1}&    \\  \hline

4&43933&{1\ 1\ 0\ 1\ 2\ 0\ 0\ 2\ 0\ 2}&    \\  \hline

4&53149&{1\ 1\ 1\ 0\ 2\ 2\ 0\ 0\ 2\ 2}&  Thm~\ref{T4.3}  \\  \hline

4&57575&{2\ 0\ 1\ 2\ 2\ 2\ 0\ 2\ 2\ 2}&  Thm~\ref{T4.3}  \\  \hline

4&84965&{2\ 1\ 2\ 2\ 1\ 1\ 2\ 2\ 0\ 1\ 1}&  Thm~\ref{T4.7}  \\  \hline

4&88655&{2\ 1\ 1\ 1\ 2\ 1\ 1\ 1\ 1\ 1\ 1}&  Thm~\ref{T4.9}  \\  \hline

4&90815&{2\ 1\ 1\ 0\ 2\ 1\ 1\ 2\ 1\ 1\ 1}&  Thm~\ref{T4.1}  \\  \hline

4&91525&{1\ 1\ 2\ 2\ 1\ 1\ 2\ 2\ 1\ 1\ 1}&  $\spadesuit$ \S\ref{ss3.3} \\  \hline

4&107765&{2\ 2\ 0\ 1\ 1\ 2\ 0\ 1\ 1\ 2\ 1}&  Thm~\ref{T4.11}  \\  \hline

4&133079&{2\ 1\ 2\ 2\ 1\ 1\ 2\ 0\ 2\ 0\ 2}&    \\  \hline

4&148415&{2\ 1\ 2\ 0\ 2\ 1\ 2\ 1\ 1\ 1\ 2}&  Rmk~\ref{R4.4}  \\  \hline

4&167173&{1\ 2\ 1\ 2\ 2\ 0\ 1\ 1\ 1\ 2\ 2}&    \\  \hline

4&265805&{2\ 2\ 1\ 1\ 2\ 1\ 1\ 1\ 1\ 1\ 1\ 1}&  Thm~\ref{T4.9}  \\  \hline

4&267935&{2\ 1\ 1\ 2\ 1\ 1\ 1\ 2\ 1\ 1\ 1\ 1}&  Thm~\ref{T4.1}  \\  \hline

4&272375&{2\ 2\ 2\ 1\ 2\ 1\ 1\ 1\ 2\ 1\ 1\ 1}&  Thm~\ref{T4.1}  \\  \hline

4&272615&{2\ 1\ 2\ 1\ 2\ 2\ 1\ 1\ 2\ 1\ 1\ 1}&  Thm~\ref{T4.1}  \\  \hline

4&273095&{2\ 2\ 1\ 1\ 2\ 1\ 2\ 1\ 2\ 1\ 1\ 1}&  Thm~\ref{T4.1}  \\  \hline

4&354293&{2\ 2\ 2\ 2\ 2\ 2\ 2\ 2\ 2\ 2\ 2\ 1}&  \cite{Hou-12} Prop 3.1  \\  \hline
 
5&515&{2,0,0,1,0,2}&  \cite{Hou-12} Table 2 No.3  \\  \hline
5&569&{2,0,0,0,1,2}&  \cite{Hou-12} Table 2 No.3  \\  \hline
5&2675&{2,0,0,0,0,2,0,1}&  \cite{Hou-12} Table 2 No.3  \\  \hline
5&4393&{1,0,2,0,0,0,0,2}&  \cite{Hou-12} Table 2 No.3  \\  \hline
5&13177&{1,0,0,2,0,0,0,0,2}&  \cite{Hou-12} Table 2 No.3  \\  \hline
5&20171&{2,0,0,0,0,2,0,0,0,1}&  \cite{Hou-12} Table 2 No.3  \\  \hline
5&58805&{2,2,2,2,2,1,2,2,2,2}&  \cite{Hou-12} Prop 3.2 (i)  \\  \hline
5&59297&{2,1,0,0,0,1,0,0,0,0,1}&  Thm~\ref{T4.1}  \\  \hline
5&59303&{2,0,1,0,0,1,0,0,0,0,1}&  Thm~\ref{T4.1}  \\  \hline
5&59305&{1,1,1,0,0,1,0,0,0,0,1}&  Thm~\ref{T4.1}  \\  \hline
5&59311&{1,0,2,0,0,1,0,0,0,0,1}&  Thm~\ref{T4.1}  \\  \hline
5&59321&{2,0,0,1,0,1,0,0,0,0,1}&  Thm~\ref{T4.1}  \\  \hline
5&59323&{1,1,0,1,0,1,0,0,0,0,1}&  Thm~\ref{T4.1}  \\  \hline
5&59329&{1,0,1,1,0,1,0,0,0,0,1}&  Thm~\ref{T4.1}  \\  \hline
5&59347&{1,0,0,2,0,1,0,0,0,0,1}&  Thm~\ref{T4.1}  \\  \hline
5&59375&{2,0,0,0,1,1,0,0,0,0,1}&  Thm~\ref{T4.1}  \\  \hline

\end{tabular}
\]
\end{table}

%
%
\clearpage 

\addtocounter{table}{-1}
\begin{table}[!h]
\caption{continued}
\vspace{-5mm}
\[
\begin{tabular}{|c|r|l|c|}
\hline
$e$ & $n \hfil$ & $3$-adic digits of $n$ & reference \\ \hline
\hline

5&59377&{1,1,0,0,1,1,0,0,0,0,1}&  Thm~\ref{T4.1}  \\  \hline
5&59383&{1,0,1,0,1,1,0,0,0,0,1}&  Thm~\ref{T4.1}  \\  \hline
5&59401&{1,0,0,1,1,1,0,0,0,0,1}&  Thm~\ref{T4.1}  \\  \hline
5&59455&{1,0,0,0,2,1,0,0,0,0,1}&  Thm~\ref{T4.1}  \\  \hline
5&59537&{2,0,0,0,0,2,0,0,0,0,1}&  Thm~\ref{T4.1}  \\  \hline
5&59539&{1,1,0,0,0,2,0,0,0,0,1}&  Thm~\ref{T4.1}  \\  \hline
5&59545&{1,0,1,0,0,2,0,0,0,0,1}&  Thm~\ref{T4.1}  \\  \hline
5&59563&{1,0,0,1,0,2,0,0,0,0,1}&  Thm~\ref{T4.1}  \\  \hline
5&59617&{1,0,0,0,1,2,0,0,0,0,1}&  Thm~\ref{T4.1}  \\  \hline
5&60023&{2,0,0,0,0,1,1,0,0,0,1}&  Thm~\ref{T4.1}  \\  \hline
5&60031&{1,0,1,0,0,1,1,0,0,0,1}&  Thm~\ref{T4.1}  \\  \hline
5&60049&{1,0,0,1,0,1,1,0,0,0,1}&  Thm~\ref{T4.1}  \\  \hline
5&60103&{1,0,0,0,1,1,1,0,0,0,1}&  Thm~\ref{T4.1}  \\  \hline
5&60757&{1,2,0,0,0,1,2,0,0,0,1}&  \cite{Hou-12} Table 2 No.5  \\  \hline
5&61481&{2,0,0,0,0,1,0,1,0,0,1}&  Thm~\ref{T4.1}  \\  \hline
5&61483&{1,1,0,0,0,1,0,1,0,0,1}&  Thm~\ref{T4.1}  \\  \hline
5&61489&{1,0,1,0,0,1,0,1,0,0,1}&  Thm~\ref{T4.1}  \\  \hline
5&61507&{1,0,0,1,0,1,0,1,0,0,1}&  Thm~\ref{T4.1}  \\  \hline
5&61561&{1,0,0,0,1,1,0,1,0,0,1}&  Thm~\ref{T4.1}  \\  \hline
5&63685&{1,0,2,0,0,1,0,2,0,0,1}&  \cite{Hou-12} Table 2 No.5  \\  \hline
5&65855&{2,0,0,0,0,1,0,0,1,0,1}&  Thm~\ref{T4.1}  \\  \hline
5&65857&{1,1,0,0,0,1,0,0,1,0,1}&  Thm~\ref{T4.1}  \\  \hline
5&65863&{1,0,1,0,0,1,0,0,1,0,1}&  Thm~\ref{T4.1}  \\  \hline
5&65881&{1,0,0,1,0,1,0,0,1,0,1}&  Thm~\ref{T4.1}  \\  \hline
5&65935&{1,0,0,0,1,1,0,0,1,0,1}&  Thm~\ref{T4.1}  \\  \hline
5&72469&{1,0,0,2,0,1,0,0,2,0,1}&  \cite{Hou-12} Table 2 No.5  \\  \hline
5&78977&{2,0,0,0,0,1,0,0,0,1,1}&  Thm~\ref{T4.1}  \\  \hline
5&78979&{1,1,0,0,0,1,0,0,0,1,1}&  Thm~\ref{T4.1}  \\  \hline
5&78985&{1,0,1,0,0,1,0,0,0,1,1}&  Thm~\ref{T4.1}  \\  \hline
5&79003&{1,0,0,1,0,1,0,0,0,1,1}&  Thm~\ref{T4.1}  \\  \hline
5&79055&{2,2,2,2,0,1,0,0,0,1,1}&  \cite{Hou-12} Thm 3.10  \\  \hline
5&79057&{1,0,0,0,1,1,0,0,0,1,1}&  Thm~\ref{T4.1}  \\  \hline
5&98821&{1,0,0,0,2,1,0,0,0,2,1}&  \cite{Hou-12} Table 2 No.5  \\  \hline
5&118591&{1,2,0,0,0,2,0,0,0,0,2}&  \cite{Hou-12} Table 2 No.4  \\  \hline
5&158117&{2,1,0,0,2,2,0,0,0,2,2}&  Thm~\ref{T4.3}  \\  \hline
5&176659&{1,2,2,2,2,0,2,2,2,2,2}&  \cite{Hou-12} Prop 3.2 (ii)  \\  \hline
5&474349&{1,1,1,0,0,2,2,0,0,0,2,2}&  Thm~\ref{T4.3}  \\  \hline
5&513875&{2,0,1,0,2,2,2,0,0,2,2,2}&  Thm~\ref{T4.3}  \\  \hline
5&766661&{2,1,2,2,2,1,1,2,2,2,0,1,1}&  Thm~\ref{T4.7}  \\  \hline
5&1121443&{1,2,2,2,2,0,2,2,2,2,0,0,2}&  Thm~\ref{T4.3}  \\  \hline
5&1541623&{1,1,0,1,0,2,2,2,0,0,2,2,2}&  Thm~\ref{T4.3}  \\  \hline
5&9565937&{2,2,2,2,2,2,2,2,2,2,2,2,2,2,1}&  \cite{Hou-12} Prop 3.1  \\  \hline
\end{tabular}
\]
\end{table}

%
%
\clearpage

\addtocounter{table}{-1}
\begin{table}[!h]
\caption{continued}
\vspace{-5mm}
\[
\begin{tabular}{|c|r|l|c|}
\hline
$e$ & $n \hfil$ & $3$-adic digits of $n$ & reference \\ \hline
\hline
6&530711&{2,2,2,2,2,2,1,2,2,2,2,2}&  \cite{Hou-12} Prop 3.2 (i)  \\  \hline

6&532175&{2,1,0,0,0,0,1,0,0,0,0,0,1}&  Thm~\ref{T4.1}  \\  \hline

6&532183&{1,1,1,0,0,0,1,0,0,0,0,0,1}&  Thm~\ref{T4.1}  \\  \hline

6&532189&{1,0,2,0,0,0,1,0,0,0,0,0,1}&  Thm~\ref{T4.1}  \\  \hline

6&532199&{2,0,0,1,0,0,1,0,0,0,0,0,1}&  Thm~\ref{T4.1}  \\  \hline

6&532253&{2,0,0,0,1,0,1,0,0,0,0,0,1}&  Thm~\ref{T4.1}  \\  \hline

6&532261&{1,0,1,0,1,0,1,0,0,0,0,0,1}&  Thm~\ref{T4.1}  \\  \hline

6&532279&{1,0,0,1,1,0,1,0,0,0,0,0,1}&  Thm~\ref{T4.1}  \\  \hline

6&532423&{1,0,1,0,0,1,1,0,0,0,0,0,1}&  Thm~\ref{T4.1}  \\  \hline

6&532495&{1,0,0,0,1,1,1,0,0,0,0,0,1}&  Thm~\ref{T4.1}  \\  \hline

6&532901&{2,0,0,0,0,0,2,0,0,0,0,0,1}&  Thm~\ref{T4.1}  \\  \hline

6&532903&{1,1,0,0,0,0,2,0,0,0,0,0,1}&  Thm~\ref{T4.1}  \\  \hline

6&532927&{1,0,0,1,0,0,2,0,0,0,0,0,1}&  Thm~\ref{T4.1}  \\  \hline

6&532981&{1,0,0,0,1,0,2,0,0,0,0,0,1}&  Thm~\ref{T4.1}  \\  \hline

6&534359&{2,0,0,0,0,0,1,1,0,0,0,0,1}&  Thm~\ref{T4.1}  \\  \hline

6&534367&{1,0,1,0,0,0,1,1,0,0,0,0,1}&  Thm~\ref{T4.1}  \\  \hline

6&536551&{1,2,0,0,0,0,1,2,0,0,0,0,1}&  \cite{Hou-12} Table 2 No.5  \\  \hline

6&538735&{1,1,0,0,0,0,1,0,1,0,0,0,1}&  Thm~\ref{T4.1}  \\  \hline

6&538741&{1,0,1,0,0,0,1,0,1,0,0,0,1}&  Thm~\ref{T4.1}  \\  \hline

6&538813&{1,0,0,0,1,0,1,0,1,0,0,0,1}&  Thm~\ref{T4.1}  \\  \hline

6&538975&{1,0,0,0,0,1,1,0,1,0,0,0,1}&  Thm~\ref{T4.1}  \\  \hline

6&551855&{2,0,0,0,0,0,1,0,0,1,0,0,1}&  Thm~\ref{T4.1}  \\  \hline

6&551935&{1,0,0,0,1,0,1,0,0,1,0,0,1}&  Thm~\ref{T4.1}  \\  \hline

6&571591&{1,0,0,2,0,0,1,0,0,2,0,0,1}&  \cite{Hou-12} Table 2 No.5  \\  \hline

6&591221&{2,0,0,0,0,0,1,0,0,0,1,0,1}&  Thm~\ref{T4.1}  \\  \hline

6&591229&{1,0,1,0,0,0,1,0,0,0,1,0,1}&  Thm~\ref{T4.1}  \\  \hline

6&591247&{1,0,0,1,0,0,1,0,0,0,1,0,1}&  Thm~\ref{T4.1}  \\  \hline

6&591463&{1,0,0,0,0,1,1,0,0,0,1,0,1}&  Thm~\ref{T4.1}  \\  \hline

6&650431&{1,0,0,0,2,0,1,0,0,0,2,0,1}&  \cite{Hou-12} Table 2 No.5  \\  \hline

6&709327&{1,0,1,0,0,0,1,0,0,0,0,1,1}&  Thm~\ref{T4.1}  \\  \hline

6&709399&{1,0,0,0,1,0,1,0,0,0,0,1,1}&  Thm~\ref{T4.1}  \\  \hline

6&709559&{2,2,2,2,2,0,1,0,0,0,0,1,1}&  \cite{Hou-12} Thm 3.10  \\  \hline

6&1419125&{2,1,0,0,0,2,2,0,0,0,0,2,2}&  Thm~\ref{T4.3}  \\  \hline

6&1592863&{1,2,2,2,2,2,0,2,2,2,2,2,2}&  \cite{Hou-12} Prop 3.2 (ii)  \\  \hline

6&4612151&{2,0,1,0,0,2,2,2,0,0,0,2,2,2}&  Thm~\ref{T4.3}  \\  \hline

6&6905813&{2,1,2,2,2,2,1,1,2,2,2,2,0,1,1}&  Thm~\ref{T4.7}  \\  \hline

6&10095919&{1,2,2,2,2,2,0,2,2,2,2,2,0,0,2}&  Thm~\ref{T4.3}  \\  \hline

6&19657477&{1,0,2,2,2,2,0,0,2,2,2,2,0,0,1,1}&  Thm~\ref{T4.3}  \\  \hline

6&258280325&{2,2,2,2,2,2,2,2,2,2,2,2,2,2,2,2,2,1}&  \cite{Hou-12} Prop 3.1  \\  \hline

\end{tabular}
\]
\end{table}

%
%
\clearpage

\begin{table}[!h]
\caption{Desirable triples $(n,e;4)$, $e\le 6$, $w_4(n)>4$}\label{Tb3}
\vspace{-5mm}
\[
\begin{tabular}{|c|r|l|c|}
\hline
$e$ & $n \hfil$ & base $4$ digits of $n$ & reference \\ \hline
\hline

2&59&{3,2,3}&  Thm~\ref{T5.6} (ii)   \\  \hline
2&127&{3,3,3,1}&  \cite{Hou-12} Prop 3.1   \\  \hline
3&29&{1,3,1}&  Exmp~\ref{E6.4}   \\  \hline
3&101&{1,1,2,1}&  Thm~\ref{nc4}   \\  \hline
3&149&{1,1,1,2}&     \\  \hline
3&163&{3,0,2,2}&  Thm~\ref{sc}   \\  \hline
3&281&{1,2,1,0,1}&  Cor~\ref{C6.19}   \\  \hline
3&307&{3,0,3,0,1}&     \\  \hline
3&329&{1,2,0,1,1}&  Exmp~\ref{E6.3}   \\  \hline
3&341&{1,1,1,1,1}&  Exmp~\ref{E6.3}   \\  \hline
3&2047&{3,3,3,3,3,1}&  \cite{Hou-12} Prop 3.1   \\  \hline
4&281&{1,2,1,0,1}&   Thm~\ref{nc4}  \\  \hline
4&307&{3,0,3,0,1}&     \\  \hline
4&401&{1,0,1,2,1}&  Thm~\ref{nc4}   \\  \hline
4&547&{3,0,2,0,2}&  Thm~\ref{sc}   \\  \hline
4&779&{3,2,0,0,3}&  Thm~\ref{T6.8}   \\  \hline
4&787&{3,0,1,0,3}&  Thm~\ref{T6.9}   \\  \hline
4&817&{1,0,3,0,3}&     \\  \hline
4&899&{3,0,0,2,3}&  Thm~\ref{T6.8}   \\  \hline
4&1469&{1,3,3,2,1,1}&     \\  \hline
4&2201&{1,2,1,2,0,2}&     \\  \hline
4&2317&{1,3,0,0,1,2}&     \\  \hline
4&2321&{1,0,1,0,1,2}&  Thm~\ref{nc4}   \\  \hline
4&2377&{1,2,0,1,1,2}&     \\  \hline
4&2441&{1,2,0,2,1,2}&     \\  \hline
4&4387&{3,0,2,0,1,0,1}&     \\  \hline
4&32767&{3,3,3,3,3,3,3,1}&  \cite{Hou-12} Prop 3.1   \\  \hline
5&29&{1,3,1}&  Exmp~\ref{E6.4}   \\  \hline
5&1049&{1,2,1,0,0,1}&  Thm~\ref{nc4}   \\  \hline
5&1061&{1,1,2,0,0,1}&  Thm~\ref{nc4}   \\  \hline
5&1169&{1,0,1,2,0,1}&  Thm~\ref{nc4}   \\  \hline
5&1289&{1,2,0,0,1,1}&  Thm~\ref{nc4}   \\  \hline
5&1409&{1,0,0,2,1,1}&  Thm~\ref{nc4}   \\  \hline
5&1541&{1,1,0,0,2,1}&  Thm~\ref{nc4}   \\  \hline
5&1601&{1,0,0,1,2,1}&  Thm~\ref{nc4}   \\  \hline
5&2083&{3,0,2,0,0,2}&  Thm~\ref{sc}   \\  \hline
5&2563&{3,0,0,0,2,2}&  Thm~\ref{T6.11}   \\  \hline

\end{tabular}
\]
\end{table}

\clearpage 

\addtocounter{table}{-1}
\begin{table}[!h]
\caption{continued}
\vspace{-5mm}
\[
\begin{tabular}{|c|r|l|c|}
\hline
$e$ & $n \hfil$ & base $4$ digits of $n$ & reference \\ \hline
\hline

5&4229&{1,1,0,2,0,0,1}&  Thm~\ref{nc4}   \\  \hline
5&4289&{1,0,0,3,0,0,1}&     \\  \hline
5&4387&{3,0,2,0,1,0,1}&     \\  \hline
5&5129&{1,2,0,0,0,1,1}&  Exmp~\ref{E6.3}   \\  \hline
5&5141&{1,1,1,0,0,1,1}&  Exmp~\ref{E6.3}   \\  \hline
5&5189&{1,1,0,1,0,1,1}&  Exmp~\ref{E6.3}   \\  \hline
5&5249&{1,0,0,2,0,1,1}&  Thm~\ref{nc4}   \\  \hline
5&5381&{1,1,0,0,1,1,1}&  Exmp~\ref{E6.3}   \\  \hline
5&8713&{1,2,0,0,2,0,2}&     \\  \hline
5&9281&{1,0,0,1,0,1,2}&  Thm~\ref{nc4}   \\  \hline
5&17429&{1,1,1,0,0,1,0,1}&     \\  \hline
5&17441&{1,0,2,0,0,1,0,1}&  Thm~\ref{nc4}   \\  \hline
5&17489&{1,0,1,1,0,1,0,1}&     \\  \hline
5&17681&{1,0,1,0,1,1,0,1}&     \\  \hline
5&524287&{3,3,3,3,3,3,3,3,3,1}&  \cite{Hou-12} Prop 3.1   \\  \hline

6&4361&{1,2,0,0,1,0,1}&  Thm~\ref{nc4}   \\  \hline
6&6161&{1,0,1,0,0,2,1}&  Thm~\ref{nc4}   \\  \hline
6&6401&{1,0,0,0,1,2,1}&  Thm~\ref{nc4}   \\  \hline
6&8227&{3,0,2,0,0,0,2}&  Thm~\ref{sc}   \\  \hline
6&8707&{3,0,0,0,2,0,2}&  Thm~\ref{nc3}   \\  \hline
6&12299&{3,2,0,0,0,0,3}&  Thm~\ref{T6.8}   \\  \hline
6&12307&{3,0,1,0,0,0,3}&  Thm~\ref{T6.9}   \\  \hline
6&14339&{3,0,0,0,0,2,3}&  Thm~\ref{T6.8}   \\  \hline
6&37121&{1,0,0,0,1,0,1,2}&  Thm~\ref{nc4}   \\  \hline
6&65801&{1,2,0,0,1,0,0,0,1}&  Cor~\ref{C6.19}   \\  \hline
6&65921&{1,0,0,2,1,0,0,0,1}&     \\  \hline
6&66307&{3,0,0,0,3,0,0,0,1}&     \\  \hline
6&135209&{1,2,2,0,0,0,1,0,2}&     \\  \hline
6&135217&{1,0,3,0,0,0,1,0,2}&     \\  \hline
6&135457&{1,0,2,0,1,0,1,0,2}&     \\  \hline
6&137249&{1,0,2,0,0,2,1,0,2}&     \\  \hline
6&8388607&{3,3,3,3,3,3,3,3,3,3,3,1}&  \cite{Hou-12} Prop 3.1   \\  \hline

\end{tabular}
\]
\end{table}

\clearpage 

\section{\hbox{Parameters Satisfying the Conditions} \hbox{of Theorems~\ref{T4.1} and \ref{T4.3}\kern 2cm }}

To describe those $\alpha, \beta$ satisfying conditions of Theorems~\ref{T4.1} and \ref{T4.3}, we first review the notion of addition in base $p$ with carries in a cyclic order. Each integer $0\le a\le p^e-1$ has a unique representation
\begin{equation}\label{n1}
a=a_0p^0+\cdots+a_{e-1}p^{e-1},\qquad 0\le a_i\le p-1.
\end{equation}
We can also treat $a$ as a function (also denoted by $a$)
\[
\begin{array}{cccc}
a:&\Bbb Z_e&\longrightarrow&\{0,\dots,p-1\}\cr
&i&\longmapsto&a_i
\end{array}
\]
We will maintain this dual meaning for each integer $0\le a\le p^e-1$. 
Recall that we write the integer $a$ in \eqref{n1} as $a=(a_0,\dots,a_{e-1})_p$.
For $0\le a,b\le p^e-1$, we can perform the  addition of $a$ and $b$ in base $p$ with carries in the natural {\em cyclic} order of $\Bbb Z_e$. The result is denoted by $a\oplus b$. Note that 
\[
a\oplus b=
\begin{cases}
(a+b)^\dagger&\text{if $a,b$ are not both $0$},\cr
0&\text{if}\ a=b=0,
\end{cases}
\]
where $c^\dagger$ is the integer in $\{1,2,\dots,p^e-1\}$ such that $c^\dagger\equiv c\pmod{p^e-1}$.

\begin{exmp}
\rm
Assume $p=5$, $e=7$.

\begin{itemize}
  \item [(i)] Let
\begin{gather*}
a=(1,0,2,3,2,1,4)_5,\\
b=(4,3,3,4,1,3,2)_5.
\end{gather*}
Then
\[
a\oplus b=(1,4,0,3,4,4,1)_5.
\]
The labels of the positions that give a carry are $2,3$ and $6,0$; the labels of those receiving a carry are $3,4$ and $0,1$.

\item[(ii)]
Let
\begin{gather*}
a=(3,2,1,3,4,0,3)_5,\\
b=(2,2,4,4,1,4,2)_5.
\end{gather*}
Then
\[
a\oplus b=(1,0,1,3,1,0,1)_5.
\]  
In this case, every position gives and receives a carry.
\end{itemize}
\end{exmp} 

In general, for $0\le a,b\le p^e-1$, carry giving-receiving in $a\oplus b$ occurs in intervals of $\Bbb Z_e$, either interrupted or uninterrupted. 

\medskip

\noindent{\bf Interrupted carries.} In this case, there are disjoint intervals 
\[
\{s_i,s_i+1,\dots,s_i+t_i\},\qquad 0<t_i<e,\ 1\le i\le m, 
\]
of $\Bbb Z_e$ such that in $a\oplus b$, the positions that give a carry are precisely $s_i,s_i+1,\dots,s_i+t_i-1$, $1\le i\le m$. This case happens if and only if
\[
a(u)+b(u)
\begin{cases}
\ge p&\text{if}\ u\in\bigcup_i\{s_i\},\cr
\ge p-1&\text{if}\ u\in\bigcup_i\{s_i+1,\dots,s_i+t_i-1\},\cr
\le p-2&\text{if}\ u\in\bigcup_i\{s_i+t_i\},\cr
\le p-1&\text{if}\ u\notin\bigcup_i\{s_i,s_i+1,\dots,s_i+t_i\}.
\end{cases}
\]

\medskip

\noindent{\bf Uninterrupted carries.}
In this case, $a(u)+b(u)\ge p-1$ for all $u\in\Bbb Z_e$ and $a(u)+b(u)\ge p$ for at least one $u$. In $a\oplus b$, every position gives and receives a carry.


\subsection{Parameters satisfying the conditions of Theorem~\ref{T4.1}}\

Let $\alpha$ and $\beta$ satisfy (i) and (ii) of Theorem~\ref{T4.1}. Then $w_p(\beta^\dagger)=p-1$ and 
\[
(\alpha p)^\dagger=(a,\dots,a,\underset{l+1}{a+1},a,\dots,a)_p
\]
for some $0\le a\le p-2$ and $l\in\Bbb Z_e$. Since
\[
w_p((\alpha p+\beta)^\dagger)=w_p((\alpha p)^\dagger\oplus\beta^\dagger)=w_p((\alpha p)^\dagger)+w_p(\beta^\dagger)-k(p-1)=ae+p-k(p-1),
\]
where $k$ is the number of carries in $(\alpha p)^\dagger\oplus\beta^\dagger$, condition (iii) of Theorem~\ref{T4.1} is equivalent to
\[
k=\frac{ae}{p-1}.
\]
Since $\frac{ae}{p-1}<e$, $(\alpha p)^\dagger\oplus\beta^\dagger$ cannot have uninterrupted carries.

\medskip

Explicit enumeration of all $\alpha, \beta$ satisfying (i) -- (iii) of Theorem~\ref{T4.1} would be very cumbersome. But a method to generate all such $\alpha,\beta$ is easily described as follows.

We first choose $l\in\Bbb Z_e$ arbitrarily and let 
\[
\alpha=(a,\dots,a,\underset{l}{a+1},a,\dots,a)_p,
\]
where $0\le a\le p-2$ is to be further specified later. The construction of $\beta$ depends on whether $l+1$ is a  carry-giving/receiving position in $(\alpha p)^\dagger\oplus\beta^\dagger$.

\medskip

{\bf Case 1.} Assume that position $l+1$ is carry-giving but not carry-receiving. 

Choose disjoint intervals $\{s_i,s_i+1,\dots,s_i+t_i\}$, $1\le i\le m$, of $\Bbb Z_e$ such that $t_i>0$, $l+1\in \bigcup_i\{s_i\}$, and 
\begin{equation}\label{n2}
t_1+\cdots+t_m=\frac{ae}{p-1}. 
\end{equation}
Choose $0<g\le p^e-1$ such that $w_p(g)=p-1$ and 
\begin{equation}\label{n3}
g(u)
\begin{cases}
\ge p-1-a&\text{if}\ u=l+1,\cr
\ge p-a&\text{if}\ u\in\bigl(\bigcup_i\{s_i\}\bigr)\setminus\{l+1\},\cr
\ge p-1-a&\text{if}\ u\in\bigcup_i\{s_i+1,\dots,s_i+t_i-1\},\cr
\le p-2-a&\text{if}\ u\in\bigcup_i\{s_i+t_i\},\cr
\le p-1-a&\text{if}\ u\notin\bigcup_i\{s_i,s_i+1,\dots,s_i+t_i\}.
\end{cases}
\end{equation} 
Let 
\[
\beta=\sum_{0\le j\le pe-1}b_jp^j, 
\]
where $b_j\ge 0$, and $\sum_{j\equiv i\, (\text{mod}\, e)}b_j=g(i)$ for $0\le i\le e-1$.

\medskip

{\bf Case 2.} Assume that position $l+1$ is carry-giving and carry-receiving.

Proceed as in Case 1 but require $l+1\in \bigcup_i\{s_i+1,\dots,s_i+t_i-1\}$ and 
\begin{equation}\label{n4}
g(u)
\begin{cases}
\ge p-2-a&\text{if}\ u=l+1,\cr
\ge p-a&\text{if}\ u\in\bigcup_i\{s_i\},\cr
\ge p-1-a&\text{if}\ u\in\bigl(\bigcup_i\{s_i+1,\dots,s_i+t_i-1\}\bigr)\setminus\{l+1\},\cr
\le p-2-a&\text{if}\ u\in\bigcup_i\{s_i+t_i\},\cr
\le p-1-a&\text{if}\ u\notin\bigcup_i\{s_i,s_i+1,\dots,s_i+t_i\}.
\end{cases}
\end{equation}

{\bf Case 3.} Assume that position $l+1$ is carry-receiving but not carry-giving.

Proceed as in Case 1 but require $l+1\in \bigcup_i\{s_i+t_i\}$ and 
\begin{equation}\label{n5}
g(u)
\begin{cases}
\le p-3-a&\text{if}\ u=l+1,\cr
\ge p-a&\text{if}\ u\in\bigcup_i\{s_i\},\cr
\ge p-1-a&\text{if}\ u\in\bigcup_i\{s_i+1,\dots,s_i+t_i-1\},\cr
\le p-2-a&\text{if}\ u\in\bigl(\bigcup_i\{s_i+t_i\}\bigr)\setminus\{l+1\},\cr
\le p-1-a&\text{if}\ u\notin\bigcup_i\{s_i,s_i+1,\dots,s_i+t_i\}.
\end{cases}
\end{equation}

{\bf Case 4.} Assume that position $l+1$ is neither carry-giving nor carry-receiving.

Proceed as in Case 1 but require $l+1\notin\bigcup_i\{s_i,s_i+1,\dots,s_i+t_i\}$ and 
\begin{equation}\label{n6}
g(u)
\begin{cases}
\le p-2-a&\text{if}\ u=l+1,\cr
\ge p-a&\text{if}\ u\in\bigcup_i\{s_i\},\cr
\ge p-1-a&\text{if}\ u\in\bigcup_i\{s_i+1,\dots,s_i+t_i-1\},\cr
\le p-2-a&\text{if}\ u\in\bigcup_i\{s_i+t_i\},\cr
\le p-1-a&\text{if}\ u\notin\bigl(\bigcup_i\{s_i,s_i+1,\dots,s_i+t_i\}\bigr)\cup\{l+1\}.
\end{cases}
\end{equation}

Case 1 is nonempty if and only if the following conditions on $a$ and $m$ are satisfied:
\begin{equation}\label{n7}
\begin{cases}
\frac{ae}{p-1}\in\Bbb Z,\cr
0<m\le \frac{ae}{p-1}\le e-m,\cr
(t_1+\cdots+t_m-m+1)(p-1-a)+(m-1)(p-a)\le p-1,\cr
(t_1+\cdots+t_m)(p-1)+m(p-2-a)+(e-m-t_1-\cdots-t_m)(p-1-a)\ge p-1.
\end{cases}
\end{equation}
The last two inequalities in \eqref{n7} can be simplified so that \eqref{n7} becomes 
\begin{equation}\label{n8}
\begin{cases}
\frac{ae}{p-1}\in\Bbb Z,\cr
0<m\le \frac{ae}{p-1}\le e-m,\cr
a(p-1-a)\le \frac{p-1}e\min\{p-m,\ (e-1)(p-1)-m\}.
\end{cases}
\end{equation}
In the same way, we find that Case 2 is nonempty  if and only if 
\begin{equation}\label{n8.0}
\begin{cases}
\frac{ae}{p-1}\in\Bbb Z,\cr
0<m< \frac{ae}{p-1}\le e-m,\cr
a(p-1-a)\le \frac{p-1}e\min\{p-m,\ (e-1)(p-1)-m\}.
\end{cases}
\end{equation}
Case 3 is nonempty if and only if 
\begin{equation}\label{n9}
\begin{cases}
\frac{ae}{p-1}\in\Bbb Z,\cr
a\le p-3,\cr
m\le \frac{ae}{p-1}< e-m,\cr
a(p-1-a)\le \frac{p-1}e\min\{p-m-1,\ (e-1)(p-1)-m-1\}.
\end{cases}
\end{equation}
 Case 4 is nonempty if and only if 
\begin{equation}\label{n10}
\begin{cases}
\frac{ae}{p-1}\in\Bbb Z,\cr
m\le \frac{ae}{p-1}< e-m,\cr
a(p-1-a)\le \frac{p-1}e\min\{p-m-1,\ (e-1)(p-1)-m-1\}.
\end{cases}
\end{equation}

\medskip

\subsection{Parameters satisfying the conditions of Theorem~\ref{T4.3}}\

\begin{prop}\label{P4.4}
All $\alpha,\beta$ satisfying conditions (i) -- (iii) of Theorem~\ref{T4.3} are enumerated as follows.

\medskip

\noindent{\bf Case 1.}
\vspace{-2mm}

\centerline{\hrulefill}

\noindent $\alpha=(\underbrace{p-1,\dots,p-1}_l,0,\dots,0)_p.$

\centerline{\hrulefill}

\medskip

\noindent{\bf Case 2.1.}

\vspace{-2mm}

\centerline{\hrulefill}

\noindent 
$\alpha=(\frac{p-1}e,\dots,\frac{p-1}e)_p$,

\smallskip

\noindent 
$b\ge\max\bigl\{p-\frac{p-1}e,\frac{p-1}e+2\bigr\}$, or $b\le\frac{p-1}e$.

\centerline{\hrulefill}

\medskip

\noindent{\bf Case 2.2.}

\vspace{-2mm}

\centerline{\hrulefill}

\noindent 
$\alpha=(\frac{2(p-1)}e,\dots,\frac{2(p-1)}e)_p$,

\smallskip

\noindent 
$p-\frac{2(p-1)}e\le b\le\frac{2(p-1)}e+1$.

\centerline{\hrulefill}

\medskip

\noindent{\bf Case 3.1.1.}

\vspace{-2mm}

\centerline{\hrulefill}

\noindent 
$\alpha=(\frac{p-1}e-1,\frac{p-1}e,\dots,\frac{p-1}e,\frac{p-1}e+1)_p$,

\smallskip
\noindent 
$b\ge\max\{\frac{(e-1)(p-1)}e,\frac{p-1}e+1\}$, or $\left\{\begin{matrix}e=2,\hfill\cr b\le\frac{p-5}2,\end{matrix}\right.$ or $\left\{\begin{matrix}e\ge 3,\hfill\cr b\le\frac{p-1}e-1.\end{matrix}\right.$

\centerline{\hrulefill}

\medskip

\noindent{\bf Case 3.1.2.}

\vspace{-2mm}

\centerline{\hrulefill}

\noindent 
$\alpha=(\frac{2(p-1)}e-1,\frac{2(p-1)}e,\dots,\frac{2(p-1)}e,\frac{2(p-1)}e+1)_p,\quad e=3\ \text{or}\ 4$,

\smallskip
\noindent 
$\frac{(e-2)(p-1)}e\le b\le\frac {2(p-1)}e$.

\centerline{\hrulefill}

\medskip

\noindent{\bf Case 3.2.1.}

\vspace{-2mm}

\centerline{\hrulefill}

\noindent 
$
\alpha=(\frac{p-1}e-1,\frac{p-1}e+1,\frac{p-1}e,\dots,\frac{p-1}e)_p,\quad e\ge 3$,

\smallskip
\noindent 
$
b\ge p-\frac{p-1}e$, or $\left\{\begin{matrix}p\ge 5,\hfill\cr b\le\frac{p-1}e-1.\end{matrix}\right.$

\centerline{\hrulefill}

\medskip

\noindent{\bf Case 3.2.2.}
\nopagebreak
\vspace{-2mm}

\centerline{\hrulefill}

\noindent 
$
\alpha=(\frac{2(p-1)}3-1,\frac{2(p-1)}3+1,\frac{2(p-1)}3)_p,\quad e=3$,

\smallskip
\noindent 
$p\equiv 1\pmod 3$,

\smallskip
\noindent 
$
\frac{p+2}3\le b\le\frac {2(p-1)}3$.

\centerline{\hrulefill}

\medskip

\noindent{\bf 
Case 3.2.3.}

\vspace{-2mm}

\centerline{\hrulefill}

\noindent $p=5$, $e=4$, $\alpha=(2,4,3,3)_5$, $b=2$.

\centerline{\hrulefill}

\medskip

\noindent{\bf 
Case 3.3.}

\vspace{-2mm}

\centerline{\hrulefill}

\noindent $\alpha=(\frac{p-1}e-1,\frac{p-1}e,\dots,\frac{p-1}e,\overbrace{\textstyle \frac{p-1}e+1}^l,\frac{p-1}e,\dots,\frac{p-1}e)_p,\quad 2\le l\le e-2$,

\smallskip
\noindent $
b\ge p-\frac{p-1}e$, or $b\le\frac{p-1}e-1$.

\centerline{\hrulefill}
\end{prop}

\begin{proof}
It is easy to see that those $\alpha$ and $\beta$ enumerated in the proposition all satisfy (i) -- (iii) of Theorem~\ref{T4.3}.

Now assume that $\alpha,\beta$ satisfy (i) -- (iii) of Theorem~\ref{T4.3}. By (i),
\[
\alpha+1=(a,\dots,a)_p+p^l,
\]
where $0\le a\le p-2$. 

{\bf Case 1.} Assume $a=0$. Then
\[
\alpha=(\underbrace{p-1,\dots,p-1}_l,0,\dots,0)_p.
\]

For the rest of the proof, assume $a>0$. Then
\[
\alpha=(a,\dots,a)_p,\quad l=0,
\]
or
\[
\alpha=(a-1,a,\dots,a,\underset l{a+1},a\dots,a)_p,\quad 1\le l\le e-1.
\]
Let $c$ be the number of carries in $(\alpha p)^\dagger\oplus\beta$. By (iii) we have
\[
p=w_p((\alpha p+\beta)^\dagger)=w_p((\alpha p)^\dagger)+p-c(p-1)=ae+p-c(p-1).
\]
So $c=\frac{ae}{p-1}$.

{\bf Case 2.} Assume $\alpha=(a,\dots,a)_p$. We have
\[
\begin{matrix}
(\alpha p)^\dagger&\kern-2mm=\kern-2mm&(\kern-2mm &a&a&a&\cdots&a&\kern-2mm )_p,\cr
\beta&\kern-2mm =\kern-2mm &(\kern-2mm &b&p-b&0&\cdots&0&\kern-2mm )_p.
\end{matrix}
\]
The only possible carry-giving positions in $(\alpha p)^\dagger\oplus\beta$ are $0$ and $1$, so $c=1$ or $2$.

{\bf Case 2.1.} Assume $c=1$. Then $a=\frac{p-1}e$. We have
\[
\begin{cases}
a+b\ge p,\cr
a+p-b\le p-2,
\end{cases}
\quad {\text or}\quad
\begin{cases}
a+b\le p-1,\cr
a+p-b\ge p,
\end{cases}
\]
i.e.,
\[
b\ge\max\Bigl\{p-\frac{p-1}e,\;\frac{p-1}e+2\Bigr\},\ \text{or}\ b\le\frac{p-1}e.
\]

{\bf Case 2.2.} Assume $c=2$. Then $a=\frac{2(p-1)}e$. We have
\[
\begin{cases}
a+b\ge p,\cr
a+p-b\ge p-1,
\end{cases}
\]
i.e.,
\[
p-\frac{2(p-1)}e\le b\le\frac{2(p-1)}e+1.
\]

{\bf Case 3.} Assume $\alpha=(a-1,a,\dots,a,\underset l{a+1},a\dots,a)_p$, $1\le l\le e-1$.

{\bf Case 3.1.} Assume $l=e-1$. We have
\[
\begin{matrix}
(\alpha p)^\dagger&\kern-2mm =\kern-2mm &(\kern-2mm &a+1&a-1&a&\cdots&a&\kern-2mm )_p,\cr
\beta&\kern-2mm =\kern-2mm &(\kern-2mm &b&p-b&0&\cdots&0&\kern-2mm )_p.
\end{matrix}
\]
The only possible carry-giving positions in $(\alpha p)^\dagger\oplus\beta$ are $0$ and $1$, so $c=1$ or $2$.

{\bf Case 3.1.1.} Assume $c=1$. Then $a=\frac{p-1}e$. We have
\[
\begin{cases}
a+1+b\ge p,\cr
a-1+p-b\le p-2,
\end{cases}
\ {\text or}\
\begin{cases}
e=2,\cr
a+1+b\le p-2,\cr
a-1+p-b\ge p,
\end{cases}
\ {\text or}\
\begin{cases}
e\ge 3,\cr
a+1+b\le p-1,\cr
a-1+p-b\ge p,
\end{cases}
\]
i.e.,
\[
b\ge\max\Bigl\{\frac{(e-1)(p-1)}e,\;\frac{p-1}e+1\Bigr\},\ \text{or}\ \begin{cases}e=2,\cr b\le\frac{p-5}2,\end{cases}\ \text{or}\ \begin{cases}e\ge 3,\cr b\le\frac{p-1}e-1.\end{cases}
\]

{\bf Case 3.1.2.} Assume $c=2$. Then $a=\frac{2(p-1)}e$, which implies $e\ge 3$. We have
\[
\begin{cases}
a+1+b\ge p,\cr
a-1+p-b\ge p-1,
\end{cases}
\]
i.e.,
\[
\frac{(e-2)(p-1)}e\le b\le\frac {2(p-1)}e.
\]
Since $e\ge 3$, the above inequalities force $e=3$ or $4$.

{\bf Case 3.2.} Assume $l=1\le e-2$. We have
\[
\begin{matrix}
(\alpha p)^\dagger&\kern-2mm =\kern-2mm &(\kern-2mm &a&a-1&a+1&a&\cdots&a&\kern-2mm )_p,\cr
\beta&\kern-2mm =\kern-2mm &(\kern-2mm &b&p-b&0&0&\cdots&0&\kern-2mm )_p.
\end{matrix}
\]
The only possible carry-giving positions in $(\alpha p)^\dagger\oplus\beta$ are $0,1,2$, so $c=1,2,3$.

{\bf Case 3.2.1.} Assume $c=1$. Then $a=\frac{p-1}e$. We have
\[
\begin{cases}
a+b\ge p,\cr
a-1+p-b\le p-2,
\end{cases}
\ {\text or}\
\begin{cases}
a+b\le p-1,\cr
a-1+p-b\ge p,\cr
a+1\le p-2,
\end{cases}
\]
i.e.,
\[
b\ge p-\frac{p-1}e,\ \text{or}\ \begin{cases}p\ge 5,\cr b\le\frac{p-1}e-1.\end{cases}
\]

{\bf Case 3.2.2.} Assume $c=2$. Then $a=\frac{2(p-1)}e$. 
We claim that $a\ne p-2$. (Otherwise, $p-2=\frac{2(p-1)}e$, which forces $p=3$, $e=4$, and $a=1$. Then $(\alpha p)^\dagger\oplus\beta=(1,0,2,1)_3\oplus(b,3-b,0,0)_3$ cannot have $2$ carries.) Thus
we must have
\[
\begin{cases}
a+b\ge p,\cr
a-1+p-b\ge p-1,\cr
a+1\le p-2,
\end{cases}
\]
i.e.,
\[
\begin{cases}
e\ge 3,\vspace{1mm}\cr
p-\frac{2(p-1)}e\le b\le\frac {2(p-1)}e.
\end{cases} 
\]
The above system forces $e=3$. Since $a\in\Bbb Z$, we must have $p\equiv 1\pmod 3$.

{\bf Case 3.2.3.} Assume $c=3$. Then $a=\frac{3(p-1)}e$, which implies $e>3$. We have
\[
\begin{cases}
a+b\ge p,\cr
a-1+p-b\ge p-1,\cr
a+1=p-1,
\end{cases}
\]
i.e.,
\[
p=5,\ e=4,\ a=3,\ b=2.
\]

{\bf Case 3.3.} Assume $2\le l\le e-2$. We have
\[
\begin{matrix}
(\alpha p)^\dagger&\kern-2mm =\kern-2mm &(\kern-2mm &a&a-1&a\;\cdots\; a&\overset{l+1}{a+1}&a\;\cdots\; a&\kern-2mm )_p,\cr
\beta&\kern-2mm =\kern-2mm &(\kern-2mm &b&p-b&0\hfill&\cdots&\hfill 0&\kern-2mm )_p.
\end{matrix}
\]
The only possible carry-giving positions in $(\alpha p)^\dagger\oplus\beta$ are $0$ and $1$, so $c=1$ or $2$. 
If $c=2$, then $a+b\ge p$ and $a-1+p-b\ge p-1$, which imply that $2a\ge p$. This is impossible since $a=\frac{2(p-1)}e$ and $e\ge 4$. So $c=1$.
We have
\[
\begin{cases}
a+b\ge p,\cr
a-1+p-b\le p-2,\cr
\end{cases}
\ {\text or}\
\begin{cases}
a+b\le p-1,\cr
a-1+p-b\ge p,\cr
\end{cases}
\]
i.e.,
\[
b\ge p-\frac{p-1}e,\ \text{or}\ b\le\frac{p-1}e-1.
\]
\end{proof}

\newpage
\section{Proof of Theorem~\ref{T3.1}}

\bigskip

When $q>3$ is odd,
\[
\begin{split}
&g({\tt y})^{2q^2+2}\equiv\cr
& 8 {\tt y}^{q^3-1}-2 {\tt y}^{q^3-3} -6 {\tt y}^{q^3-q^2+3 q-1} +12 {\tt y}^{q^3-q^2+3 q-3} -6 {\tt y}^{q^3-q^2+3 q-5} -18 {\tt y}^{q^3-q^2+2 q} 
\cr
& +28 {\tt y}^{q^3-q^2+2 q-2} -10 {\tt y}^{q^3-q^2+2q-4} -18 {\tt y}^{q^3-q^2+q+1} +20 {\tt y}^{q^3-q^2+q-1} -4 {\tt y}^{q^3-q^2+q-3} 
\cr
& -6 {\tt y}^{q^3-q^2+2} +4 {\tt y}^{q^3-q^2} -2 {\tt y}^{q^3-2 q^2+4 q-1} +4 {\tt y}^{q^3-2 q^2+4q-3} -2 {\tt y}^{q^3-2 q^2+4 q-5} -8 {\tt y}^{q^3-2 q^2+3 q}
\cr
& +12 {\tt y}^{q^3-2 q^2+3q-2} -4 {\tt y}^{q^3-2 q^2+3 q-4} -12 {\tt y}^{q^3-2 q^2+2 q+1} +13 {\tt y}^{q^3-2 q^2+2q-1} -4 {\tt y}^{q^3-2 q^2+2 q-3}
\cr
& +{\tt y}^{q^3-2 q^2+2 q-5} -8 {\tt y}^{q^3-2q^2+q+2} +6 {\tt y}^{q^3-2 q^2+q} -4 {\tt y}^{q^3-2 q^2+q-2} +2 {\tt y}^{q^3-2q^2+q-4}
\cr 
& -2 {\tt y}^{q^3-2 q^2+3} +{\tt y}^{q^3-2 q^2+1} -2{\tt y}^{q^3-2 q^2-1} +{\tt y}^{q^3-2 q^2-3} +2 {\tt y}^{q^3-3 q^2+3 q-1} -4 {\tt y}^{q^3-3 q^2+3 q-3}
\cr
& +2 {\tt y}^{q^3-3 q^2+3q-5} +6 {\tt y}^{q^3-3 q^2+2 q} -10 {\tt y}^{q^3-3 q^2+2 q-2} +4 {\tt y}^{q^3-3 q^2+2q-4} +6 {\tt y}^{q^3-3 q^2+q+1}
\cr
& -8 {\tt y}^{q^3-3 q^2+q-1} +2 {\tt y}^{q^3-3q^2+q-3} +2 {\tt y}^{q^3-3 q^2+2} -2 {\tt y}^{q^3-3 q^2} +{\tt y}^{q^3-4 q^2+4 q-1} -2 {\tt y}^{q^3-4 q^2+4q-3}
\cr
& +{\tt y}^{q^3-4 q^2+4 q-5} +4 {\tt y}^{q^3-4 q^2+3 q} -6 {\tt y}^{q^3-4q^2+3 q-2} +2 {\tt y}^{q^3-4 q^2+3 q-4} +6 {\tt y}^{q^3-4 q^2+2 q+1}
\cr
& -6 {\tt y}^{q^3-4q^2+2 q-1} +{\tt y}^{q^3-4 q^2+2 q-3} +4 {\tt y}^{q^3-4 q^2+q+2} -2{\tt y}^{q^3-4 q^2+q} +{\tt y}^{q^3-4 q^2+3}+{\tt y}^{4 q^2} -2 {\tt y}^{4 q^2-2}
\cr
& +{\tt y}^{4q^2-4} +4 {\tt y}^{3 q^2+q} -8 {\tt y}^{3 q^2+q-2} +4 {\tt y}^{3 q^2+q-4} +4{\tt y}^{3 q^2+1} -6 {\tt y}^{3 q^2-1} +2 {\tt y}^{3 q^2-3} +6 {\tt y}^{2 q^2+2 q}
\cr
& -12 {\tt y}^{2 q^2+2 q-2} +6 {\tt y}^{2 q^2+2q-4} +12 {\tt y}^{2 q^2+q+1} -18 {\tt y}^{2 q^2+q-1} +6 {\tt y}^{2 q^2+q-3} +6 {\tt y}^{2q^2+2}-6 {\tt y}^{2 q^2} 
\cr
& +{\tt y}^{2 q^2-2} +4 {\tt y}^{q^2+3 q} -8 {\tt y}^{q^2+3 q-2} +4{\tt y}^{q^2+3 q-4} +12 {\tt y}^{q^2+2 q+1} -18 {\tt y}^{q^2+2 q-1} +6 {\tt y}^{q^2+2 q-3} 
\cr
& +12{\tt y}^{q^2+q+2}-14 {\tt y}^{q^2+q} +6 {\tt y}^{q^2+q-2} -2 {\tt y}^{q^2+q-4} +4 {\tt y}^{q^2+3} -4{\tt y}^{q^2+1} +4 {\tt y}^{q^2-1} -2 {\tt y}^{q^2-3}
\cr
& +{\tt y}^{4 q}-2 {\tt y}^{4 q-2} +{\tt y}^{4 q-4} +4 {\tt y}^{3 q+1} -6{\tt y}^{3 q-1} +2 {\tt y}^{3 q-3} +6 {\tt y}^{2 q+2} -12 {\tt y}^{2 q}+13 {\tt y}^{2 q-2} -6 {\tt y}^{2 q-4} 
\cr
& +4 {\tt y}^{q+3}-14{\tt y}^{q+1} +20 {\tt y}^{q-1} -8 {\tt y}^{q-3} +{\tt y}^4-6 {\tt y}^2.
\cr
\end{split}
\]

\clearpage

When $q>3$ is even,
\[
\begin{split}
&g({\tt y})^{2q^2+q+3}\equiv\cr
& {\tt y}^{q^3-1} +{\tt y}^{q^3-5} +{\tt y}^{q^3-q+4} +{\tt y}^{q^3-q+2} +{\tt y}^{q^3-2 q+5} +{\tt y}^{q^3-2 q+1} +{\tt y}^{q^3-2 q-1} +{\tt y}^{q^3-2 q-3} 
\cr
& +{\tt y}^{q^3-2 q-5}+{\tt y}^{q^3-q^2+4 q-2} +{\tt y}^{q^3-q^2+4 q-6} +{\tt y}^{q^3-q^2+3 q-1} +{\tt y}^{q^3-q^2+3 q-3} +{\tt y}^{q^3-q^2+3 q-5} 
\cr
& +{\tt y}^{q^3-q^2+3 q-7}+{\tt y}^{q^3-q^2+2 q} +{\tt y}^{q^3-q^2+2q-2} +{\tt y}^{q^3-q^2+2 q-6} +{\tt y}^{q^3-q^2+q+1} +{\tt y}^{q^3-q^2+q-1}  
\cr
& +{\tt y}^{q^3-q^2-4} +{\tt y}^{q^3-q^2-q+1}+{\tt y}^{q^3-q^2-q-1} +{\tt y}^{q^3-q^2-q-3} +{\tt y}^{q^3-q^2-q-5} +{\tt y}^{q^3-q^2-2 q}
\cr
& +{\tt y}^{q^3-q^2-2 q-4} +{\tt y}^{q^3-2 q^2+6 q-1} +{\tt y}^{q^3-2 q^2+6 q-3} +{\tt y}^{q^3-2 q^2+6q-5} +{\tt y}^{q^3-2 q^2+6 q-7}
\cr
& +{\tt y}^{q^3-2 q^2+5 q-2} +{\tt y}^{q^3-2 q^2+5 q-6} +{\tt y}^{q^3-2 q^2+4 q-1} +{\tt y}^{q^3-2 q^2+4 q-3} +{\tt y}^{q^3-2 q^2+3 q-2}
\cr
& +{\tt y}^{q^3-2 q^2+3 q-4} +{\tt y}^{q^3-2 q^2+3 q-6} +{\tt y}^{q^3-2 q^2+2 q+3} +{\tt y}^{q^3-2 q^2+2 q+1} +{\tt y}^{q^3-2 q^2+2q-3}
\cr
& +{\tt y}^{q^3-2 q^2+q+2} +{\tt y}^{q^3-2 q^2+q-4} +{\tt y}^{q^3-2 q^2+3}+{\tt y}^{q^3-2 q^2-3}+{\tt y}^{q^3-2 q^2-5}+{\tt y}^{q^3-2 q^2-q+2} 
\cr
& +{\tt y}^{q^3-2 q^2-q} +{\tt y}^{q^3-2 q^2-q-2} +{\tt y}^{q^3-2 q^2-q-4} +{\tt y}^{q^3-2 q^2-2 q+1} +{\tt y}^{q^3-2 q^2-2 q-1} +{\tt y}^{q^3-3 q^2+2q}
\cr
& +{\tt y}^{q^3-3 q^2+2 q-4} +{\tt y}^{q^3-3 q^2+q+1} +{\tt y}^{q^3-3 q^2+q-1} +{\tt y}^{q^3-3 q^2+q-3} +{\tt y}^{q^3-3 q^2+q-5}
\cr
& +{\tt y}^{q^3-3 q^2+2} +{\tt y}^{q^3-3 q^2} +{\tt y}^{q^3-3 q^2-4} +{\tt y}^{q^3-3 q^2-q+3} +{\tt y}^{q^3-3 q^2-q+1} +{\tt y}^{q^3-3 q^2-2 q+2}
\cr
& +{\tt y}^{q^3-4 q^2+4q+1} +{\tt y}^{q^3-4 q^2+4 q-1} +{\tt y}^{q^3-4 q^2+4 q-3} +{\tt y}^{q^3-4 q^2+4 q-5} +{\tt y}^{q^3-4 q^2+3 q}
\cr
& +{\tt y}^{q^3-4 q^2+3 q-4} +{\tt y}^{q^3-4 q^2+2 q+1} +{\tt y}^{q^3-4 q^2+2 q-1} +{\tt y}^{q^3-4 q^2+q} +{\tt y}^{q^3-4 q^2+q-2}
\cr
& +{\tt y}^{q^3-4 q^2+q-4} +{\tt y}^{q^3-4 q^2+5} +{\tt y}^{q^3-4 q^2+3} +{\tt y}^{q^3-4 q^2-1} +{\tt y}^{q^3-4q^2-q+4}+{\tt y}^{q^3-4 q^2-q+2}
\cr
& +{\tt y}^{q^3-4 q^2-2 q+3} +{\tt y}^{6 q^2}+{\tt y}^{6 q^2-2} +{\tt y}^{6 q^2-4} +{\tt y}^{6 q^2-6} +{\tt y}^{6 q^2-2q} +{\tt y}^{6 q^2-2 q-2} +{\tt y}^{6 q^2-2 q-4}
\cr
& +{\tt y}^{6 q^2-2 q-6} +{\tt y}^{5 q^2-1} +{\tt y}^{5q^2-5} +{\tt y}^{5 q^2-q} +{\tt y}^{5 q^2-q-2} +{\tt y}^{5 q^2-q-4} +{\tt y}^{5 q^2-q-6}
\cr
& +{\tt y}^{5 q^2-2 q-1} +{\tt y}^{5 q^2-2 q-5} +{\tt y}^{4 q^2+2 q} +{\tt y}^{4 q^2+2 q-2} +{\tt y}^{4 q^2+2 q-4} +{\tt y}^{4 q^2+2 q-6} +{\tt y}^{4 q^2+q-1}
\cr
& +{\tt y}^{4 q^2+q-5} +{\tt y}^{4 q^2+2} +{\tt y}^{4 q^2-4} +{\tt y}^{4 q^2-q+1} +{\tt y}^{4 q^2-q-1} +{\tt y}^{4 q^2-q-3} +{\tt y}^{4 q^2-q-5}
\cr
& +{\tt y}^{4 q^2-2 q+2} +{\tt y}^{4 q^2-2 q-4} +{\tt y}^{3 q^2-3} +{\tt y}^{3q^2-q-2} +{\tt y}^{3 q^2-q-4} +{\tt y}^{3 q^2-2 q-3} +{\tt y}^{2 q^2+4 q}
\cr
& +{\tt y}^{2 q^2+4 q-2} +{\tt y}^{2 q^2+4 q-4} +{\tt y}^{2 q^2+4 q-6} +{\tt y}^{2 q^2+q-3} +{\tt y}^{2 q^2+4} +{\tt y}^{2q^2+2} +{\tt y}^{2 q^2-6}
\cr
& +{\tt y}^{2 q^2-q-1} +{\tt y}^{2 q^2-q-3} +{\tt y}^{2 q^2-2 q+4} +{\tt y}^{2 q^2-2 q+2} +{\tt y}^{2 q^2-2 q} +{\tt y}^{q^2+4 q-1} +{\tt y}^{q^2+4 q-5} 
\cr
& +{\tt y}^{q^2+3 q} +{\tt y}^{q^2+3 q-2} +{\tt y}^{q^2+3 q-4} +{\tt y}^{q^2+3 q-6} +{\tt y}^{q^2+2 q-3} +{\tt y}^{q^2+q} +{\tt y}^{q^2+q-6} +{\tt y}^{q^2+3}
\cr
& +{\tt y}^{q^2-1} +{\tt y}^{q^2-3} +{\tt y}^{q^2-5} +{\tt y}^{q^2-q+4}+{\tt y}^{q^2-q+2}+{\tt y}^{q^2-2 q+3} +{\tt y}^{6 q}+{\tt y}^{6q-2} +{\tt y}^{6 q-4} 
\cr
& +{\tt y}^{6 q-6}+{\tt y}^{5 q-1} +{\tt y}^{5 q-5} +{\tt y}^{4 q+2} +{\tt y}^{4 q-4} +{\tt y}^{3 q+1} +{\tt y}^{2 q+4} +{\tt y}^{2 q} +{\tt y}^{2 q-2} +{\tt y}^{2 q-4}
\cr
& +{\tt y}^{2 q-6} +{\tt y}^{q+3} +{\tt y}^{q+1}+{\tt y}^{q-5} +{\tt y}^6 +{\tt y}^2.
\end{split}
\]

\vfill\eject
\end{document}